\documentclass{amsart}

\usepackage{amsmath,amsthm,amsfonts,amssymb,mathrsfs}
\usepackage[margin=1in]{geometry}
\usepackage[T1]{fontenc}
\usepackage[backref]{hyperref}
\usepackage{verbatim}
\usepackage{todonotes}
\usepackage{appendix}
\usepackage{tikz}
\usepackage{tikz-cd}
\usepackage{xspace}
\usetikzlibrary{matrix,arrows}
\usepackage{fancyvrb }
\usepackage{color}
\usepackage{tikz}
\usepackage{soul}
\usepackage[capitalise,noabbrev,nameinlink]{cleveref}
\usetikzlibrary{arrows,matrix}
\usepackage[all]{xy}
\usepackage{enumitem}

\usepackage{supertabular}
\usepackage{array}\newcolumntype{R}{>{$}r<{$}}

\newtheorem*{maintheorem}{Theorem}

\newtheorem{theorem}{Theorem}[section]

\newtheorem{proposition}[theorem]{Proposition}

\newtheorem*{conj}{Conjecture}

\newtheorem{corollary}[theorem]{Corollary}

\newtheorem{lemma}[theorem]{Lemma}

\theoremstyle{definition}

\newtheorem{definition}[theorem]{Definition}

\newtheorem{example}[theorem]{Example}

\newtheorem{remark}[theorem]{Remark}

\DeclareMathOperator{\im}{Im}

\DeclareMathOperator{\codim}{codim}

\DeclareMathOperator{\Hom}{Hom}
\DeclareMathOperator{\reg}{reg}

\DeclareMathOperator{\Lie}{Lie}

\DeclareMathOperator{\GL}{GL}

\DeclareMathOperator{\SL}{SL}
\DeclareMathOperator{\Sp}{Sp}
\DeclareMathOperator{\SO}{SO}

\DeclareMathOperator{\Irr}{Irr}

\DeclareMathOperator{\spec}{Spec}
\DeclareMathOperator{\Ad}{Ad}

\DeclareMathOperator{\spn}{span}
\DeclareMathOperator{\spr}{Spr}
\DeclareMathOperator{\regu}{reg}
\DeclareMathOperator{\nat}{nat}

\newcommand{\Pe}{\cP_\epsilon}
\newcommand{\Pee}{\cP^{sp}_{\epsilon,\,\epsilon'}}

\def\mini{\mathrm{min}}
\def\sp{\mathrm{sp}}

\newcommand{\A}{\mathbb{A}}
\newcommand{\C}{\mathbb{C}}

\newcommand{\F}{\mathbb{F}}
\newcommand{\Q}{\mathbb{Q}}

\newcommand{\Z}{\mathbb{Z}}

\newcommand{\cF}{\mathcal{F}}
\newcommand{\cG}{\mathcal{G}}

\newcommand{\cL}{\mathcal{L}}
\newcommand{\cM}{\mathcal{M}}
\newcommand{\cN}{\mathcal{N}}

\newcommand{\cP}{\mathcal{P}}
\newcommand{\cS}{\mathcal{S}}

\newcommand{\cU}{\mathcal{U}}

\newcommand{\ba}{\mathbf{a}}
\newcommand{\bb}{\mathbf{b}}
\newcommand{\bc}{{\mathbf{c}}}

\newcommand{\m}{\mathrm{m}}

\newcommand{\sP}{\mathcal{P}}
\newcommand{\covP}{\widetilde{\mathcal{P}}}

\newcommand{\fg}{\mathfrak{g}}

\newcommand{\frakm}{\mathfrak{m}}

\newcommand{\g}{\mathfrak{g}}

\newcommand{\cg}{\mathfrak{c}}

\newcommand{\frakS}{\mathfrak{S}}
\newcommand{\fc}{\mathfrak{c}}
\newcommand{\fs}{\mathfrak{s}}

\newcommand{\fraksl}{\mathfrak{sl}}

\newcommand{\Ab}{\bar{A}} 

\newcommand{\Hb}{\bar{H}}

\newcommand{\cm}{{\C^\times}}

\newcommand{\0}{\mathcal{O}}

\def\ic{\mathbf{IC}}

\newcommand{\Phat}{\hat{\mathcal P}(\0)}  

\title{Lusztig's special pieces conjecture}

\author{Daniel Juteau, Paul Levy, Eric Sommers and Shilin Yu}

\date{\today}

\begin{document}

\begin{abstract}
Let $\0$ be a special nilpotent orbit in the Lie algebra $\g$ of a simple algebraic group $G$.
We give two proofs of the result that every special piece ${\mathcal P}(\0)$ in $\g$
is the quotient of a smooth $G$-variety $X$ by the action of a certain finite group $H$.  
We first deduce the result from a similar result for transverse slices, established in earlier work of the first three authors and Fu.
Then we give a more explicit construction of $X$, as a subvariety of the closure of a $G$-orbit
in the direct sum of $\g$ and some fundamental weight representations of $G$.   
Both methods apply to classical $\g$, where we give new proofs of this result, which was first proved by Kraft and Procesi.
The result in the exceptional groups was conjectured by Lusztig.
Our first proof shows that there can be several $G$-varieties $X$ that satisfy the conjecture, related to a natural embedding of $H$
in the fundamental group of $\0$.  In an appendix, we relate this natural embedding to Lusztig's definition of $H$ that arises from the 
family in the Weyl group of $G$ attached to $\0$ and from the Springer correspondence.
\end{abstract}

\maketitle

\setcounter{section}{-1}

\section{Introduction}\label{introsec}
Let $G$ be a connected simple algebraic group over $\C$ and $\g=\Lie(G)$ its Lie algebra.
Denote by ${\mathcal N}$ the set of nilpotent elements of $\g$, by $\0$ a $G$-orbit in ${\mathcal N}$, and by $G^e$ the centralizer of an element $e\in \0$.
The component group $G^e/(G^e)^\circ$ will be denoted $A(e)$, or by abuse of notation, $A(\0)$.
The component group depends on the center of $G$; when $G$ is simply-connected, $A(\0)$ is the fundamental group $\pi_1(\0)$ of the orbit, and in general, $\pi_1(\0)$ is a central extension of $A(\0)$.

Lusztig defined a subset of the nilpotent orbits, called the {\it special} nilpotent orbits \cite{Lusztig:special}, as follows.  The Springer correspondence attaches to each irreducible representation of the Weyl group $W$ of $G$ a nilpotent orbit $\0$ and an irreducible representation $\pi$ of $A(\0)$. 
Lusztig observed that his special representations (a subset of the irreducible representations of $W$) arose only when $\pi$ is trivial; he called the corresponding nilpotent orbits special.  Some properties of special orbits are the following:  1) the poset of special nilpotent orbits (under the inclusion order of orbit closures) has a symmetric structure (induced by tensoring the attached special representation with
the sign representation);  2) while the $2$-sided cells in the (dual) affine Weyl group of $G$ are parameterized by the nilpotent orbits in $\g$, the cells that have non-trivial intersection with $W$ correspond to the special orbits; 3) special orbits are similarly identified among all nilpotent orbits as exactly those whose closure is the associated variety of a primitive ideal of the universal enveloping algebra of $\g$ with integral infinitesimal character. 

To each special orbit $\0$ is attached a {\it special piece} ${\mathcal P}(\0)$, the complement in $\overline\0$ of the union of all special nilpotent orbit closures not containing $\overline\0$.
A fundamental property (clear from Spaltenstein's definition and proof of uniqueness \cite[III.1,2]{Spaltenstein}) is that ${\mathcal N}(\g)$ is the disjoint union of its special pieces.  
Another property is that each $\cP(\0)$ contains a unique minimal nilpotent orbit, denoted $\0_m$.  

There is a unique minimal (non-zero) {\it special} nilpotent orbit in $\g$, contained in the {\it minimal special piece}, for which $\0_m$ is the minimal (non-zero) nilpotent orbit in $\g$. 
Lusztig observed \cite{Lusztig:green} that it follows from a computation of the number of ${\mathbb F}_q$-rational points that the minimal special piece is rationally smooth; he conjectured that any special piece is rationally smooth.  For exceptional Lie algebras, this was verified by Shoji in type $F_4$ \cite{Shoji:Green} and Beynon and Spaltenstein in types $E_6$, $E_7$ and $E_8$ \cite{Beynon-Spaltenstein}. 
For classical types, rational smoothness was proved by Kraft and Procesi \cite{Kraft-Procesi:special}, as a consequence of a stronger result: ${\mathcal P}(\0)$ is the quotient of a smooth variety by an elementary abelian 2-group $H$. 
The group $H$
was subsequently interpreted by Lusztig \cite{Lusztig:unipotent}, 
and extended to the exceptional types, as a 
certain normal subgroup of his canonical quotient $\Ab(\0)$ of $A(\0)$.
This led him to make the following conjecture about ${\mathcal P}(\0)$.

\begin{conj}[``Lusztig's special pieces conjecture'']
Each special piece ${\mathcal P}(\0)$ 
in an exceptional Lie algebra is isomorphic to the quotient of a smooth variety 
$\covP(\0)$
by the finite group $H$.
\end{conj}
Lusztig already understood the solution in case of the minimal special piece, which is only non-trivial in the non-simply-laced types $B_n$, $C_n$, $F_4$, and $G_2$: it follows from work of Brylinski and Kostant \cite{Brylinski-Kostant:JAMS} that 
$\covP$ is the closure of the minimal orbit in the associated simply-laced type $D_{n+1}$, $A_{2n+1}$, $E_6$, and $D_4$ under the folding action of $\mathfrak{S}_2$ (and $\mathfrak{S}_3$ in the $G_2$ case).


In general, there is a natural procedure for constructing candidates $\covP(\0)$ that solve the conjecture.
Let $\0$ be any nilpotent orbit in $\g$ and let $e\in\0$.  
Then $\0$ is isomorphic to the homogeneous space $G/G^e$ 
and any $G$-equivariant covering $\widetilde{\0}$ of $\0$ is isomorphic to $G/\Gamma$ 
where $\Gamma$ is a subgroup satisfying 
$(G^e)^\circ \subseteq \Gamma \subseteq G^e$. 
Since $\Gamma$ is the pre-image of a subgroup $K$ of $A(\0)$, 
we refer to $\widetilde{\0}$ as $\widetilde{\0}_K$. 
We denote by $\widetilde{X}_K$ the affine $G$-variety coming from the algebra $\C[\widetilde{\0}_K]$.
As explained in \cite[\S P.3]{Namikawa:covers}, any $\widetilde{X}_K$ has symplectic singularities.
In the case of $\0$ itself, $\widetilde{X}_{A(\0)}$ is the normalization of $\overline{\0}$ \cite{Jantzen}.
There are natural maps $\widetilde{X}_K \xrightarrow{p_K} \widetilde{X}_{A(\0)} \xrightarrow{\nu} \overline{\0}$.
Let $\pi_K = \nu \circ p_K$.

The covering is Galois if $K$ is normal, in which case $A(\0)/K$
acts on $\widetilde{X}_K$, via the action on $\widetilde{\0}_K$ on the right.
In that case, $p_K$ is a quotient map of affine varieties by taking $A(\0)/K$-invariants. 
Since we know from 
\cite{FJLS23} and \cite{Kraft-Procesi:classical} that $\cP(\0)$ is normal, $\nu$ is an isomorphism 
when restricted to $\nu^{-1}(\cP(\0))$.  
Thus, a natural candidate for a solution to the conjecture is 
the $G$-variety $\covP_K(\0):= \pi_K^{-1}(\cP(\0))$,
assuming we can find a normal subgroup $K$ where this variety is smooth and such that 
$A(\0)/K \simeq H$.
If such a $K$ exists, then 
$\cP(\0)$ is the quotient of 
$\covP_K(\0)$ by a group isomorphic to $H$.
It turns out that such a $K$ always exists and we observe that there can, in fact, be many choices for $K$, so that there is not necessarily a unique solution to the conjecture.

In the above picture, $H$ appears as a quotient group of $\pi_1(\0)$, but in Lusztig's formulation it is a subgroup of $\bar{A}(\0)$.  As explained in \cite{Achar-Sage-special}, $H$ can be described as a parabolic subgroup relative to a Coxeter basis of $\bar{A}(\0)$ defined by \cite{Achar-duality} (see also \cite{Achar-Sommers}). 
Let $\cS$ be the Slodowy slice in $\cP(\0)$ to $\0_m$ at the point $e_m$. In \cref{cor:i_A_injective_all_types}, 
we will give an intrinsic characterization of $H$, as follows.
Let $S^{\regu}$ be the regular part of $\cS$.  It turns out that 
$H \simeq \pi_1(\cS^{\regu})$ (see \cite{FJLS23}).
We show that the map from $\pi_1(\cS^{\regu})$ to $\pi_1(\0)$ induced by inclusion of spaces is injective.  We use the notation $H(\0)$ to refer to the image of this map in $\pi_1(\0)$, or sometimes its image in $A(\0)$ or $\bar{A}(\0)$, where we show that the image remain isomorphic to $H$. 
The image of $H(\0)$ in $\Ab(\0)$
identifies with Lusztig's and Achar-Sage's description of $H(\0)$ in $\bar{A}(\0)$, where the latter is denoted $\mathcal G_c$
and the former, $\mathcal G'_c$, in their work.

Our main theorem is the following, stated for $G$ simply connected, but valid in general, after replacing $\pi_1(\0)$ by $A(\0)$.
\begin{maintheorem}[``Lusztig's special pieces conjecture'', equivalent statement] 
Let $K$ be any normal subgroup of $\pi_1(\0)$ with $[\pi_1(\0):K]=|H(\0)|$ and $H(\0) \cap K = \{1\}$; that is, $\pi_1(\0)$ is a semidirect product of $K$ and $H(\0)$. There always exists at least one such $K$.
For any such $K$, 
the $G$-variety 
$\covP_K(\0)$ is a solution to Lusztig's original conjecture.  That is,
\begin{enumerate}
    \item The pre-image $\covP_K(\0) \subseteq \widetilde{X}_K$ of $\cP(\0)$ is smooth.  
    \item $\pi_1(\0)/K \simeq H(\0)$ and the restriction of $\pi_K$ to $\covP_K(\0)$ is a quotient map by $H(\0)$. 
\end{enumerate}
Moreover, for such any such solution, $\pi_K^{-1}(e)$ is a singleton set for any $e \in \0_{m}$.  
\end{maintheorem}\label{Theorem:main}

Among the choices for $K$ that satisfy the hypothesis of the theorem, there is a natural choice  $K^{\nat}$, which appears in \cite{Achar-Sage-perverse} and implicitly in \cite{Kraft-Procesi:special} for the classical groups.
We recall the definition of $K^{\nat}$ in \cref{definition:natural_complement_to_H}.
Achar and Sage \cite{Achar-Sage-perverse} previously constructed a normal, irreducible $G$-variety whose $H(\0)$-quotient is isomorphic to ${\mathcal P}(\0)$ and that contains $\widetilde{\0}_K$ for $K=K^{\nat}$, but they did not show that it is smooth.  

A new observation is that $K$ need not be unique, so 
there can multiple, non-isomorphic 
$G$-varieties that solve the original conjecture.  
Considering the situation where $G$ is a classical group, so $A(\0)$ is an elementary abelian $2$-group of rank $r$, 
there are $2^{(r-m)m}$ choices for $K$ that satisfy
the hypothesis of the Theorem, where $|H(\0)| = 2^m$.  
However, in the exceptional groups, when $G$ is adjoint, 
there is only the solution $K=K^{nat}$.
Also, if one does not care specifically that $\pi_1(\0)/K \simeq H(\0)$, which is equivalent to $\pi_K^{-1}(e)$ being a point, then there are many other non-minimal covers that solve the conjecture; in particular, the universal cover of $\0$ (where $K=1$) would always satisfy part (1) above.

The Main Theorem of \cite{FJLS23} is that $\cS \simeq V/H(\0)$ for an affine space $V$ where $H(\0)$ acts linearly and symplectically. 
 We will explain in \S \ref{conceptualsec} that the Theorem can be deduced by a relatively short conceptual argument from this local version of Lusztig's conjecture. 

 Although $\widetilde{X}_K$ is concrete, it is not in general explicit.
To this end, in \S \ref{explicitsec}, 
we consider an explicit $G$-variety $\hat{X}_K$ defined as the closure of 
the $G$-orbit through the point $(e,u) \in \g\oplus U$, where $U$ is a representation of $G$ and $u \in U$ is chosen so that $G.(e,u) \simeq \widetilde{\0}_K$. 
Let $\hat{X}_K$ be the closure of $G.(e,u)$ in $\g \oplus U$ and 
let $\hat\pi_K:\hat{X}_K \to \bar{\0}$ be the restriction of the projection 
of $\g\oplus U$ onto the first factor.  Although $\hat{X}_K$ may not be normal,
its normalization equals $\widetilde{X}_K$.
We will show that $\hat\pi_K^{-1}(\cS) \simeq V$ for any $K$ that 
satisfies the hypotheses of the Theorem.  
This implies that $\hat{\mathcal P}_K(\0):=\hat \pi^{-1}_K({\cP(\0)})$ is smooth
and hence that $\covP_K(\0) \simeq \hat{\mathcal P}_K(\0)$. 
Indeed, this shows that 
any irreducible quasi-affine $G$-variety satisfying 
Lusztig's conjecture and extending the covering map $\widetilde{\0}_K \to \0$ must be isomorphic to $\covP_K(\0)$.  
Actually, to simplify things, we work just with those $K$ that are pre-images of
subgroups in $A(\0)$  (see \S\ref{Canonical}).
In the classical types, this provides a new and slightly more explicit proof of Kraft and Procesi's precursor to the special pieces conjecture (see \S\ref{classicalsubsec}).  In exceptional types, the case of $H(\0)=\mathfrak{S}_2$ is covered in \S\ref{S2subsec}
and $H(\0)= \frakS_3$, in \S\ref{S3_subsec}.  

That leaves the two most difficult cases: 
$\0=F_4(a_3)$ in type $F_4$ where $H(\0)=\mathfrak{S}_4$ and 
$\0=E_8(a_7)$ in type $E_8$ where $H(\0)=\mathfrak{S}_5$.
These cases presented the greatest difficulty for the proof of \cite[Main Theorem]{FJLS23}, ultimately requiring brute force computations; we will here very briefly outline how similar computations can be used to prove smoothness of $\hat{\mathcal P}_K(\0)$ in \S\ref{exceptionalsubsec}.

We also prove a version of Lusztig's special pieces conjecture in the case of metaplectic nilpotent orbits.  The metaplectic special orbits are those orbits in type $C$ whose partition has transpose that corresponds to an orbit (necessarily special) in type $D$.  We denote this fourth classical type as $C'$ in \S\ref{subsub:classical_component_groups}.

Finally, we mention an application of Lusztig's special pieces conjecture in the representation theory of real reductive Lie groups. In \cite{Yu}, the fourth-named author has studied special nilpotent orbits that are dual to distinguished nilpotent orbits in the Langlands dual Lie algebras, under the Barbasch-Vogan-Lusztig-Spaltenstein duality. The main theorem above implies that the singular locus of the affinizations of the universal covers of these orbits always have codimension greater or equal than $6$. With this key ingredient, \cite{Yu} provides a uniform construction of special unipotent representations of the real forms of $G$ attached to such orbits, which are building blocks of the Arthur/Adams-Barbasch-Vogan packets (\cite{Arthur84, Arthur89, ABV}), together with a uniform proof of the unitarity of such representations. 

\vspace{0.3cm}
\noindent
{\bf Some notation and terminology}

The connected group $G$ will be an exceptional group of {\it adjoint} type or, as discussed in \S\ref{sec:component_groups_canonial_quotient}, a classical group.
The component group $A(\0)$ of a nilpotent orbit $\0$ is the component group in the connected component of $G$, and is isomorphic to $\mathfrak{S}_2^l$ (in classical types) or $\mathfrak{S}_r$ for $r\leq 5$ (in exceptional types).
We denote the centralizer in $G$, resp. $\g$, of an element $x\in\g$ or $x\in G$, by $G^x$, resp. $\g^x$.

Let $\0,\0'$ be nilpotent orbits in $\g$ with $\0'\subset\overline\0$. 
The {Slodowy slice} associated to an $\mathfrak{sl}_2$-triple $\{e,h,f\}$ with $e\in \0'$ is the affine subspace $S:=e+\g^f$ of $\g$.
The {\it Slodowy slice singularity from $\0'$ to $\overline\0$} is defined to be $\cS = \cS_{e, \0} := S\cap\overline\0$ (scheme-theoretic intersection).
It is well known that this is essentially independent of the choice of representative $e\in\0'$ or $\mathfrak{sl}_2$-triple containing $e$.
Note that in previous work of the first three authors and Fu (and in much of the literature) the equivalent definition $(f+\g^e)\cap\overline\0$ is given; the formulation we use here is more convenient for the calculations in \S \ref{explicitsec}.

As in the introduction, we use the notation $\widetilde{\mathcal P}_K(\0)$ and $\hat{\mathcal P}_K(\0)$ for the pre-images of ${\mathcal P}(\0)$ in $\widetilde{X}_K$ and $\hat{X}_K$ respectively. 
Similarly, $\widetilde{\mathcal S}$ and $\hat{\mathcal S}$ will denote the pre-images of the Slodowy slice singularity ${\mathcal S}$ in $\widetilde{X}_K$ and $\hat{X}_K$ respectively.
It follows from the definition of a transverse slice (see e.g. \cite[\S 5.1]{Slodowy:book}) that $\widetilde{\mathcal S}$, resp. $\hat{\mathcal S}$ is transverse to $\widetilde{X}_K$, resp. $\hat{X}_K$ at any point of $\pi_K^{-1}(e)$, resp. $\hat\pi_K^{-1}(e)$.

Throughout, we fix a maximal torus $T$ of $G$ and a system $\{ \alpha_1,\ldots ,\alpha_m\}$ of positive roots $\Phi^+$ in the root system of $G$ relative to $T$.
We use Bourbaki's numbering of the simple roots.
Denote by $\hat\alpha$ the highest root in $\Phi$ and by $e_{\hat\alpha}$ the corresponding element of a Chevalley basis for $\g$.
Unless otherwise stated, for any $\mathfrak{sl}_2$-triple as above we will assume $h\in\Lie(T)$ and that $h$ is {\it dominant}, i.e. $\alpha_i(h)\geq 0$ for all $i$.
By the $\mathfrak{sl}_2$-representation theory, any finite-dimensional $\g$-module $V$ decomposes as the direct sum of the eigenspaces $V(i;h):= \{ v\in V : h\cdot v=iv\}$ for $i\in{\mathbb Z}$.
In particular, $\g(0;h)=\g^h$.

The closure of the minimal nilpotent orbit of a simple Lie algebra of type $A_n$, $B_n$ etc is called a {\it minimal singularity} and is denoted $a_n$, $b_n$ etc.
For the non-simply laced types, the closure of the minimal special orbit is denoted $g_2^{\sp}$, $b_3^{\sp}$ etc.
We note in particular that $c_m\cong \C^{2m}/\mathfrak{S}_2$, where $\mathfrak{S}_2$ acts via $\pm 1$, and $g_2^{\sp}\cong d_4/\mathfrak{S}_3$, where $\mathfrak{S}_3$ acts via diagram (outer) automorphisms; see \cite[2.3.5]{FJLS23} for further discussion of the latter isomorphism.

\vspace{0.3cm}
\noindent
{\bf Acknowledgements.}
We thank Julius Grimminger for rapid computation of the Hilbert series of ${\mathcal C}_5$ (cf. Remark \ref{E8remark}) and Jean Michel for extensive correspondence related to the Chevie package.
The second author would like to thank the Universit\'e de Picardie Jules Verne for financial support and hospitality for his visit in May-June 2024.
We thank Baohua Fu and George Lusztig for helpful conversations.

The work of S.Y. has been partially supported by the Fundamental and Interdisciplinary Disciplines Breakthrough Plan of the Ministry of Education of China (FIDBP), by NSFC grants (Grant Nos. 12471028 and 12131018), and by the Natural Science Foundation of Fujian Province (Grant No. 2022J06005).

\section{Special pieces, the canonical quotient, and the group $H(\0)$}

\subsection{General set-up}\label{subsec:new_general_set_up}

Let $\0,\0'$ be nilpotent orbits in $\g$ with $\0'\subset\overline\0$. 
Let $\{ e,h,f\}$ be an $\mathfrak{sl}_2$-triple with $e \in \0'$. Let ${\cS:=\cS_{e,\0}}$ be the Slodowy slice singularity from $\0'$ to $\overline\0$, as defined in the Introduction.  

For a complex variety $X$ and a base point $x \in X$, we write $\pi_1(X,x) = \pi_1(X)$ for the algebraic/\'{e}tale fundamental group with respect to $x$. Note that in all of our applications, such as the case of a nilpotent $G$-orbit $\0$, the topological fundamental groups are always finite; hence, the two fundamental groups are naturally isomorphic. In fact, the topological fundamental group of the regular locus of any conical symplectic singularity is finite by \cite{Braun}.

Now the regular part $\cS^{\reg} = \cS \cap \0$ is embedded as a closed subvariety of $\0$.
Let $\pi_1(\cS^{\reg})$
be the fundamental group of $\cS^{\reg}$
with respect to any chosen geometric base point $x$ in $\cS^{\reg}$. 
The inclusion map $\cS^{\reg} \hookrightarrow \0$ induces a homomorphism 
$$i_\pi: \pi_1(\cS^{\reg}) \to \pi_1(\0)$$ of fundamental groups with respect to $x$. 
Since the definition of the fundamental groups depends on $x$, 
and the slice $\cS$ depends on the choice of the $\mathfrak{sl}_2$-triple (which is unique up to $G$-conjugation), the image of $\pi_1(\cS^{\reg})$ in $\pi_1(\0)$ is only defined up to conjugacy in $\pi_1(\0)$.
For an example where the image is not a normal subgroup, see Example \ref{Example:B_4}.

Suppose now that $\cS \simeq V / H$, where $H$ is a finite group acting effectively on a symplectic vector space $V$ by symplectic automorphisms.
The pre-image of $\cS^{\reg}$ under the quotient map $V \twoheadrightarrow V/H \simeq \cS$ is the dense open subset $V^\circ$ for which $H$ acts freely.
Furthermore, $\codim_V(V \backslash V^\circ) \geq 2$. 
In this setup, we have $\cS^{\reg} \simeq V^\circ / H$ and 
$$\pi_1(\cS^{\reg}) \simeq H.$$

\begin{example}\label{example:typeA_fundamental}
    Let $\g$ be of type $A_n$.  Let $\0$ be the regular nilpotent orbit and $\0'$ the subregular orbit.  In this case, $\cS \simeq \C^2/H$ where $H$ is cyclic of order $n+1$ (\cite{Slodowy:book}).  At the same time, $\pi_1(\0) \simeq H$,
    both being isomorphic to the center of $G=SL_n$, the simply-connected group for $\g$. 
    Then $i_{\pi}$ is an isomorphism by Proposition \ref{prop:typeA_fundamental}.
    For all other types, $\cS \simeq \C^2/\Gamma$ with $|\Gamma| > |\pi_1(\0)|$, so $i_\pi$ cannot be injective. 
\end{example}

When $\0$ is special and $\0'=\0_m$  is the minimal nilpotent orbit in the special piece $P(\0)$, 
the Main Theorem of \cite{FJLS23} is that $\cS \simeq V/H$ for a 
certain finite group $H$, and that $\cS$ satisfies the additional properties above.

\begin{definition}
    Let $\0$ be special.  Define $H(\0) \subseteq \pi_1(\0)$ to be the image of $\pi_1(\cS)$ under $i_\pi$. 
\end{definition}

As noted above, $H(\0)$ is only defined up to conjugacy.  A crucial result in our proof in \S\ref{conceptualsec} is that in this case $i_\pi$ is injective (see \cref{cor:i_A_injective_all_types}).  
In particular, $H(\0) \simeq H$.


\subsection{Component groups and Lusztig's canonical quotient}\label{Canonical}\label{sec:component_groups_canonial_quotient}

Let $\0$ be an arbitrary nilpotent orbit. Let $A_{adj}(\0)$ denote the the adjoint component group of $\0$.  That is, $A_{adj}(\0)$ is defined 
to be $G^e/(G^e)^\circ$ where $G$ is of adjoint type and $e \in \0$; the group is independent of the choice of $e$.  Now we define a group $A(\0)$ for all types that is a quotient of $\pi_1(\0)$ and for which $A_{adj}(\0)$ is a quotient.

\subsubsection{Exceptional types}
When $\g$ is exceptional, we define $A(\0) := A_{adj}(\0)$. 
In exceptional types, a non-trivial $A(\0)$ is isomorphic to either $\frakS_2,\frakS_3, \frakS_4$ or $\frakS_5$.
Lusztig defined his canonical quotient $\bar{A}(\0)$ of $A(\0)$ in \cite{Lusztig1984}.  When $A(\0) \simeq \frakS_2$,
then $\bar{A}(\0)$ is either trivial or $S_2$.  Otherwise, when $A(\0)$ is non-trivial, then $\bar{A}(\0) \simeq A(\0)$, except 
for the case of $\0 = E_8(b_6)$, when $A(\0) \simeq \frakS_3$ while $\bar{A}(\0) \simeq \frakS_2$.  See the tables in \cite{Sommers:Duality}.

\subsubsection{Classical types}\label{subsub:classical_component_groups}

Let $V$ be a vector space of dimension $n$ endowed with a
non-degenerate bilinear form satisfying
$\langle w,v\rangle=(-1)^\epsilon\langle v,w\rangle$
where $\epsilon \in \{0,1\}$.
Let $G^+$ be the stabilizer of this form and let
$\mathfrak g=\operatorname{Lie}(G^+)$. We also write
$\mathfrak g=\mathfrak g(V)$. 
When $\epsilon = 1$, we have $G^+ = \Sp(V)$
and when $\epsilon = 0$, we have $G^+ = {\rm O}(V)$, which is disconnected.
Let $G$ equal the connected component of $G^+$.

For a partition $\lambda$ of $n$ and a part $c$ of $\lambda$, let
$m_\lambda(c)$, or just $m(c)$ if $\lambda$ is clear,
denote the multiplicity of $c$ in $\lambda$.  Then $\lambda$ can be expressed in exponential notation
$\lambda = (c_1^{m(c_1)}, c_2^{m(c_2)}, \dots)$ for the distinct parts $c_i$ of $\lambda$.

The height $h_\lambda(c)$, or just $h(c)$ if $\lambda$ is clear, of the part $c$ is defined to be $$h_\lambda(c):=\sum_{b\geq c}m_\lambda(b).$$ 
The length of $\lambda$, denoted $\ell(\lambda)$, is the total number of parts of $\lambda$ counted with multiplicity, which is equal to the height of the smallest part of $\lambda$.

An $\epsilon$-partition is a partition $\lambda$ such that $m(c)$
is even for all parts $c$ with $c \equiv\epsilon\bmod 2$. 
Let $\cP_\epsilon(n)$ denote the set of all $\epsilon$-partitions of $n$. It is well known that the
nilpotent $G^+$-orbits in $\mathfrak g$ are in one-to-one
correspondence with the $\epsilon$-partitions of $n$; we refer to
\cite{C-M} for further details.
We denote the nilpotent orbit corresponding to a partition $\lambda$ by $\0_\lambda$.

Let $\0=\0_\lambda$ be the nilpotent orbit
corresponding to an $\epsilon$-partition $\lambda$. We denote by
$A^+(\0)$ the component group of the centralizer $(G^+)^e$ of any $e \in\0$ in $G^+$, and 
$A(\0)$ the component group of the centralizer $G^e$ of $e$ in $G$.
Let $C_\lambda$ be the reductive
part of $(G^+)^e$, which is a product of orthogonal and symplectic groups. 
Specifically, for a part $c$ of $\lambda$ such that $c \not\equiv \epsilon \bmod 2$ (resp. $c \equiv \epsilon\bmod 2$), 
$C_\lambda$ has a factor ${\rm O}_{m(c)}$ (resp. $\Sp_{m(c)}$).
Then $A^+(\0) \simeq C_\lambda/ C^\circ_\lambda$, which is an elementary abelian $2$-group.

Let 
$$\nu = [\nu_1 >  \nu_2 > \dots > \nu_l]$$ be 
the (distinct) parts of $\lambda$ that are not congruent to $\epsilon$ modulo $2$.
When $\epsilon = 1$, we make the convention that $\nu_l=0$ and $m(0)=0$ and so
the corresponding orthogonal group ${\rm O}_{m(\nu_l)} = {\rm O}_{0}$ is the trivial group.
When $\epsilon = 0$, we have $\nu_i>0$ for all $i$.
Since the symplectic factors are connected, only the orthogonal groups 
contribute to $A^+(\0)$.  
Let $x_{\nu_i} \in {\rm O}_{m(\nu_i)}$ be an element not in ${\rm SO}_{m(\nu_i)}$, except when
$\nu_\ell =0$, in which case $x_0 = 1$.
Then $A^+(\0)$ is generated by the $x_{\nu_i}$ and $A^+(\0) \simeq \frakS_2^{\ell}$
or $\mathfrak{S}_2^{\ell-1}$ depending on whether $\epsilon=0$ or $1$, respectively.
For $1 \leq i \leq \ell-1$, define
\[
\sigma_i=x_{\nu_i}x_{\nu_{i+1}}. 
\]
Then $A(\0)$ is minimally generated by $\{\sigma_i \ | \ 1\leq i\leq \ell-1 \}$
and hence is isomorphic to $\mathfrak S_2^{\ell-1}$.
We note that in type $B$, we have $A(\0) = A_{adj}(\0)$ and in type $C$, we have $A(\0) = \pi_1(\0)$,
while in type $D$, $A(\0)$ is (in general) different from both, but surjects to the former and is a quotient of the latter.

We now want to incorporate a fourth type $C'$ into the discussion. Let $X\in\{B,C,D,C'\}$. We attach to $X$ two parity parameters
\[
(\epsilon,\epsilon')=(\epsilon_X,\epsilon'_X)=
\begin{cases}
(0,1), & X=B;\\
(1,0), & X=C;\\
(0,0), & X=D;\\
(1,1), & X=C'.
\end{cases}
\]
where $C'$ denotes the {\it metaplectic} (\cite{BMSZ-metaplectic}), or {\it alternative} (\cite{JLS:Duality}), type $C$ theory. For a given type $X$, we will consider the partitions $\cP_X(n) := \cP_{\epsilon_X}(n)$ and the corresponding nilpotent orbits $\0 := \0_{\lambda}$ in $\g_{\epsilon_X}(V)$ for $\lambda \in \cP_X(n)$.
Note that type $C'$ has the same underlying symplectic nilpotent orbits as in type $C$, but it has different special orbits and different canonical quotients.
Note that $n$ is compatible with $(\epsilon,\epsilon')$
in the sense that $n \equiv \epsilon'(1-\epsilon) \bmod{2}$, otherwise the corresponding set of partitions is empty.


We now recall the description of Lusztig's canonical quotient group $\Ab(\0)$ in both the ordinary classical types and the metaplectic type, following
\cite[Section 5]{Sommers:Duality} (for the ordinary classical types) and \cite[Section 6.2]{JLS:Duality} (for all types).
For any natural number $m$, Write $[m]:= \{1, 2, \dots, m\}$.
Let 
\begin{align} \label{equation:vartheta}
&\vartheta_0 = \{ i \in [\ell-1] \ | \ h_\lambda(\nu_i) \equiv \epsilon'\}  \text{ and } \\
&\vartheta_1 = \{i \in [\ell-1] \ | \ h_\lambda(\nu_i) \not\equiv \epsilon'\}.
\nonumber
\end{align}



Let 
$\eta_A: A(\0)\twoheadrightarrow \Ab(\0)$ 
be the surjective homomorphism defining $\Ab$
and let
$$N = \ker(\eta_A).$$
The following description of $N$ in types $B,C,D$ follows easily from the discussions in \cite[Section 5]{Sommers:Duality}, and in type $C'$ is taken as the
definition of the metaplectic canonical quotient. This description has already appeared in \cite[Section 6.2]{JLS:Duality}. Note that $\Ab(\0)$ is defined for general orbits $\0$, not necessarily special.

\begin{proposition} \label{prop:ALT_A2Abar}
    View $A(\0)$ as an $\F_2$-vector space.
	Then the subspace $N$ of $A(\0)$ has basis     
    $$\{\sigma_i \ | \ i \in \vartheta_1\}.$$

    Define $\Ab'$ of $A(\0)$ to be the subspace with basis 
    $$\{\sigma_i \ | \ i \in \vartheta_0\}.$$
    Then $A(\0) = \Ab' \oplus N$.  Hence, 
    $\eta_A$ restricted to $\Ab'$ is an isomorphism onto $\bar{A}(\0) = A(\0)/N$. 
    Define $\bar\sigma_i= \eta_A(\sigma_i)$.
    Then $\bar{A}(\0)$ has a basis consisting of the elements $\bar\sigma_i$ for $i \in \vartheta_0$, and this basis makes $\Ab$ into a Coxeter group.
\end{proposition}

\begin{remark}
	By the proof of \cite[Proposition 5.3]{Achar-Aubert}, the basis $\{ \bar{\sigma}_i \}_{i \in \vartheta_0}$ forms exactly the \emph{superminimal} conjugacy classes defined in \cite[\S\,7.1]{Achar-duality} in type $B$, $C$ and $D$, which are defined in terms of a canonical partial order on the set of conjugacy classes of $\Ab(\0)$ ($\0$ not necessarily special and $\g$ arbitrary). 
    The choice of simple reflections in \cite{Achar-Sommers} consists precisely of representatives of superminimal conjugacy classes. 
\end{remark}

\begin{remark}\label{remark:center_G_in_Abar}
In types $C$ and $D$, the center of $G$ is $-I$, so its image in $A(\0)$ is the element 
$\prod_{m(\nu_i) \equiv 1} x_{\nu_i}$.  This is the same as $\prod_{i \in \vartheta_1} \sigma_i$, which lies in $N$. In type $B$, the group $G$ has trivial center.
Hence, $\Ab(\0)$ is a quotient of $A_{adj}(\0)$ in types $B,C,D$.
However, in type $C'$, this is not necessarily the case, e.g. for the metaplectic special orbit with partition $[4,2,2]$ in $C_4$, we have $\Ab(\0) = \pi_1(\0)$ is rank 2, while $A_{adj}(\0)$ is rank 1.
\end{remark}


\begin{example}\label{example:C15}
Let $\lambda = [10,8,6,4,2,0]$ in type $C_{15}$.  
Then $N$ has a basis $\{\sigma_1 = x_{10}x_8, \sigma_3=x_6x_4, \sigma_5 = x_2\}$ and $\Ab'$
has a basis $\{x_8x_6, x_4x_2\}$.  So $\Ab$
has rank $2$.  The center of $G$ corresponds 
to the product of all the $x_{\lambda_i}$ for $1 \leq i \leq 5$
(in this case, $\nu = \lambda$).
\end{example}

\subsection{Special pieces and proof of injectivity}\label{subsec:sp_pieces_and_injective}

\subsubsection{Special pieces} \label{subsubsec:sp_pieces}

Let $\0=\0_\lambda$ be a special nilpotent orbit for classical type $X$.
Recall that this means that for any part $c$ of $\lambda$ such that 
$ c\equiv\epsilon\bmod 2$,  then $h_\lambda(c)\equiv \epsilon'\bmod 2$.
Equivalently, for $\lambda \in \mathcal P_X(n)$, the orbit $\0_\lambda$ is special 
in type $X \in \{B,C\}$ if the transpose
partition $\lambda^t \in \mathcal P_X(n)$; and it is special in type
$C'$ (respectively, $D$) if $\lambda^t \in \mathcal P_D(n)$ (resp., 
$\lambda^t \in \mathcal P_{C'}(n))$.  To avoid confusion,
we refer to the special orbits in type $C'$ as the {\it metaplectic special orbits}.

Define
\begin{equation}\label{equation:def_J}
J = \{ i \in \vartheta_0 \ | \ \nu_{i} = \nu_{i+1}+2\}.
\end{equation}
The orbits in $P(\0)$ are indexed by subsets $\mathcal J \subseteq J$. 
For each $i \in \mathcal J$, the definition of $J$ ensures that 
$\lambda$ contains the following sub-partition
$$[\nu_i, \ (\nu_i-1)^{2b_i},\ \nu_i-2]$$
for some nonnegative integer $b_i$.
The non-special orbit $\0_{\mathcal J}$ in $\cP(\0)$ attached to ${\mathcal J}$ has 
partition $\lambda_{\mathcal J}$ obtained by replacing, for each $i \in {\mathcal J}$, one $\nu_i$ and one $\nu_i-2$ by two instances of $\nu_i-1$.  In other words, for each $i \in {\mathcal J}$, 
perform the operation 
\begin{equation}\label{equation:special_degenerations}
[\nu_i, \ (\nu_i-1)^{2b_i},\ \nu_i-2]
\quad\rightsquigarrow\quad [(\nu_i-1)^{2b_i+2}]
\end{equation}
to convert $\lambda$ to $\lambda_{\mathcal J}$. The minimal orbit $\0_m$ in $\cP(\0)$ corresponds to $\mathcal J=J$ and the partial order on the orbits in $\cP(\0)$ corresponds to partial order on subsets of $J$ by reverse inclusion.

As discussed in \S\ref{subsec:new_general_set_up}, it was shown in \cite{FJLS23} that the slice $\cS$ to $\0_m$
satisfies $\cS \simeq V/H$ for a finite group $H$.
The proof relied on the fact that $\cS$ is unibranch at $x \in \0_m$, which uses Green functions in the exceptional groups
and the normality results of Kraft-Procesi in the classical types \cite{Kraft-Procesi:classical}.  In the classical types, we have $H \simeq \frakS^{|J|}_2$. In the exceptional types, we have either $H=1$ or $H=A(\0)=\Ab(\0)$, with the exception of $\0 = E_8(b_6)$; if $\cP(\0) = \0$, then $H=H(\0)=1$, while if $\0 \subsetneq \cP(\0)$, then $H(\0)$ has full image in $\Ab(\0)$ and $H(\0) \simeq \Ab(\0)$.
We note that this matches Lusztig's definition of $\cG_\bc'$ for exceptional groups \cite{Lusztig:unipotent}.

\subsubsection{Injectivity} \label{subsubsec:injective}
In this section, we will prove that 
$i_\pi: \pi_1(\cS^{\reg}) \to \pi_1(\0)$ is injective.
Let $q_A:\pi_1(\0) \to A(\0)$ be the natural map and let
$i_A = q_A\circ i_\pi$.   Set $i_{\Ab}:= \eta_A\circ i_A$.
We will also show that $i_A$ and $i_{\Ab}$ are injective and 
determine the image of these maps.

The following proposition is needed for the proof.
Let $\0$ be special.  Let $\0'$ be an orbit in $\sP(\0)$ with $\0' \neq \0$.
Our proof of injectivity involves passing to a pseudo-Levi subgroup 
$M=G^t$ where $t \in G$ is semisimple. 
For such a subgroup with Lie algebra $\m$ meeting $\0$, 
let $e \in \frakm \cap \0$.

Suppose there exists $x \in \0' \cap \frakm$.
Let $\cS_{x, \0 \cap \frakm}$ be a Slodowy slice in $\frakm$.
We will be interested in writing the semisimple part of $M$ as an almost direct product $M_0 \cdot M_1$, and its Lie algebra
as $\frakm_0 \oplus \frakm_1$. With such a decomposition, 
write $x = x_0 + x_1$ for $x_i \in \frakm_i$ and $e = e_0+e_1$ for $e_i \in \frakm_i$.
Write $t=t_0t_1$ for $t_i \in M_i$.
In the proposition we identify a subgroup $Q$ of $M_0$ that is a 
product of special linear groups.

\begin{proposition}\label{prop:key_slice_typeA}
It is possible to choose a triplet $(M,e,x)$ and decomposition $M=M_0 \cdot M_1$ so that $e_1=x_1$ and one of the following holds:   
\begin{enumerate}
    \item $Q \simeq \rm{SL}_2 \times \dots \times \rm{SL}_2$ and  
    $\cS_{x, \0 \cap \frakm} = x+\cN_0$ where 
    $Q$ is the semisimple part of $M_0^{x_0}$
    and $\cN_0$ is the nilpotent cone in $\rm Lie(Q)$.
    \item $M_0 \simeq \rm{SL}_d$ for $d \in \{3,4,5\}$ and 
    $\cS_{x, \0 \cap \frakm} = e_1 + \cS'$ where $\cS'$ equals the slice in $M_0$ of the regular orbit to the subregular orbit.  Set $Q=M_0$.
    \item $M_0 \simeq \rm{SL}_3 \times \rm{SL}_2$ and
    $\cS_{x, \0 \cap \frakm} = e_1 + \cS'$ where $\cS'$ equals the slice in $M_0$ of the regular orbit to the orbit which is subregular
    in the first factor and zero in the second. Set $Q=M_0$.
\end{enumerate}
Furthermore, $t_0$ lies in the center $Z$ of $Q$ and 
the image of the map $\pi_1(\cS^{\reg}_{x, \0}) \to A(\0)$ contains the image of $s_0$ in $A(\0)$.
\end{proposition}

\begin{proof}
Assuming that one of (1)--(3) holds, $Z$ identifies with $\pi_1(\cS^{\reg}_{x, \0\cap\frakm})$
as in Example \ref{example:typeA_fundamental}.
We have the following commutative diagram:
\begin{center}
	\begin{tikzcd}
        Z = \pi_1(\cS^{\reg}_{x, \0\cap\frakm})  \ar[r] \ar[d]  & \pi_1(\0 \cap \frakm) \ar[d]  \ar[r, two heads] & \pi_1^{M}(\0 \cap \frakm) := M^e/(M^e)^\circ \ar[d] \\
        H = \pi_1(\cS^{\reg}_{x, \0}) \ar[r] & \pi_1(\0) \ar[r, two heads] & \pi_1^{G}(\0) := A(\0) 
	\end{tikzcd}
\end{center}

Let $G$ be of classical type $X_n$.  Let $T$ be a maximal torus in $G$.
For $1 \leq r < n$, we can choose $t \in T$ of order $2$ so that 
$M = G^t$ is equal to
$\rm{Sp}_{2r} \times \rm{Sp}_{2n-2r}$ in type $C_n$ and $C'_n$, 
to $\rm{SO}_{2r} \times \rm{SO}_{2n-2r}$ in type $D_n$,
and 
to $\rm{SO}_{2r} \times \rm{SO}_{2n-2r+1}$ in type $B_n$.
Furthermore, $t$ is restricts to $-I$ in the first factor $M_0$ of $M$ 
and the identity matrix $I$ in the second factor $M_1$.

Recall from \S\ref{subsubsec:sp_pieces} that if $\0$ has partition $\lambda$, then 
$\0'$ corresponds to a nonempty subset ${\mathcal J} \subset J$ and the partition for $\0'$
is $\lambda_{\mathcal J}$ obtained by performing, for every $i\in {\mathcal J}$, the substitution in $\lambda$ of
\[
	[s_i\!+\!1,\ s_i^{2b_i},\ s_i\!-\!1] \quad\rightsquigarrow\quad [s_i^{2b_i+2}].
\]
where we have set $s_i = \nu_i\!-\!1$ to simplify notation.
Let $$r = \sum_{i \in {\mathcal J}} s_i$$
and define $M = M_0 \times M_1$ to be subgroup defined above for this choice of $r$ and 
$t \in T$ to be the corresponding
order $2$ element for which $M = G^t$.
Let $\lambda^{0}$ be the partition of $2r$ in $\frakm_0$ that is the union of 
the partitions $[s_i+1,s_i-1]$ for $i \in {\mathcal J}$ and 
let $\lambda_{\mathcal J}^{0}$ be the union of $[s_i, s_i]$ for $i \in {\mathcal J}$.
Define $e_0 \in \frakm_0$ to be a nilpotent element with partition $\lambda^0$, 
and $x_0\in\frakm_0$ to be a nilpotent element with partition $\lambda_{\mathcal J}^0$.  
Let $\lambda^1$ be the union of the remaining parts of 
$\lambda$, which are the same
as the remaining parts of $\lambda_{\mathcal J}$. 
Let $e_1 \in \frakm_1$ be a nilpotent element with partition $\lambda^1$.
Set $e = e_0+e_1$ and $x =x_0+e_1$.  Then $e \in\0$ and and $x \in \0'$ by construction.
Note that the image of $t = t_0$ in $A(\0)$ is equal to
$$x_{\mathcal J} :=\prod_{i\in \mathcal J} x_{\nu_i}x_{\nu_{i+1}} = 
\prod_{i\in \mathcal J} \sigma_i.$$


By \S \ref{subsub:classical_component_groups}, 
the semisimple part of 
$M_0^{x_0}$ is isomorphic to 
$\rm{Sp}_2 \times \rm{Sp}_2 \times \dots \times \rm{Sp}_2$ (with $|\mathcal J|$ factors).  Define this to be $Q$, as in part (1). It follows that the $s$ restricts to the Coxeter element in $Z$.
Let $Y$ be the closure of the minimal nilpotent orbit in $\mathfrak{sl}_2$.  
Recall that $Y = \C^2/\frakS_2$, where the $\frakS_2$ is acting via $-I \in \rm{SL}_2 = \rm{Sp}_2$.
By \cite{FJLS23} and \cite{Kraft-Procesi:classical}, and the fact that $e_1=x_1$,
the slice $\cS_{x, \0 \cap \frakm}$ is equal to a product of $|\mathcal J|$ copies of $Y$.  
Thus the natural map of $Z \simeq \pi_1(\cS^{\reg}_{x, \0 \cap \frakm})$ to  $M^e/(M^e)^\circ$ is injective; 
the image of the map equals the image of $Z$ in the component group (see 
also \cref{prop:typeA_fundamental}).  
By the commutativity of the diagram we have proved the statement for the classical types and also identified
an element, namely $x_{\mathcal J}$, in the image of $\pi_1(\cS^{\reg}_{x, \0})$ in $A(\0)$.

For the exceptional groups, we have $A(\0) = A_{adj}(\0)$.  We gather the choices for $(M,e,x)$ in Table \ref{table:attaching_conj_classes}
making use of \cite{Sommers:Bala-Carter}.
The cases where $\frakm$ has a factor $\frakm_0$
of type $A_{r-1}$ for $2 \leq r \leq 5$ or $A_2 \times A_1$ are clear since the element in $A(\0)$ from \cite{Sommers:Bala-Carter} ensures that $M_0$ is a product of special linear groups.  Moreover, 
\cref{prop:typeA_fundamental} ensures
that $Z$ maps injectively into $\pi_1^M(\0)$.

The only difficulty is in part (1) when no such type $A$ factor of $M$ exists.  In these five remaining cases, we need to show that the semisimple part $Q$ of $M_0^{x_0}$ is $SL_2$ or $\rm{SL}_2 \times \rm{SL}_2$.  The cases involved are when $\0'$ equals
$B_2$ in type $F_4$; and $2A_3$, $D_6(a_2)$, $A_7$ and $D_7$ in $E_8$.

For the $F_4$ case, 
where $\frakm$ is of type $B_4$, then $M=M_0 \simeq \rm{Spin}(9)$.
The partition for $\0\cap\frakm$ in $B_4$
is $[5,3,1]$ and $x_0=[5,1^4]$.  Then 
$Q$ is conjugate to the semisimple part of the 
Levi subgroup of type $D_2$, which is $\rm{SL}_2 \times \rm{SL}_2$.

Consider now the four $E_8$ cases.  
For $e = A_4+2A_1$, we have 
$\frakm = D_4+A_3$.
Then $x= A_3+A_3$, with $e_0 = D_4(a_1)$ and $x_0=A_3$ in the $D_4$-factor and $e_1 =x_1$ the regular element in the second factor. 
Here, the semisimple part of the $D_4$ Levi subgroup is $\rm{Spin}(8)$.
Then $Q$ is of type $A_1$, but the nonzero nilpotent element in $\mathfrak{q} = \rm{Lie}(Q)$ is of type $2A_1$ in $D_4$, so $Q\simeq \rm{SL}_2$ sits diagonally in the Levi subgroup of type $\rm{SL}_2 \times \rm{SL}_2$.

The other three cases occur with $\frakm$ of type $D_8$ and 
the key fact is that $M=M_0=G^s$ is (famously) not $\rm{SO}(16)$, but rather one of the other $2$-fold quotients of $\rm{Spin}(16)$.  
This ensures that the Levi subalgebras of $M$ of type $D_2$ and $4A_1$ are all products of $SL_2$'s.
\begin{enumerate}
 \item For $e = E_8(a_7)$ and $e_0= D_8(a_5) = [7,5,3,1]$ and 
 $x=x_0 =D_6(a_2)= [7,5,1^4]$.  
 Here $Q \simeq \rm{SL}_2 \times \rm{SL}_2$ is a Levi subalgebra of type $D_2$.
    \item For $e = E_8(b_6)$ and $e_0= D_8(a_3)=[9,7]$ and $x_0=A_7=[8,8]$, 
    then $Q \simeq \rm{SL}_2$  sitting diagonally in a Levi subalgebra of type $4A_1$.
    \item For $e=E_8(a_5)$ and $e_0= D_8(a_1) = [13,3]$ and $x_0 = D_7 = [13,1^3]$, then 
    $Q \simeq \rm{SL}_2$ sits in a Levi subalgebra of type $D_2$.
\end{enumerate}
In all these cases, $s$ can be chosen to restrict to $-I$ in each $SL_2$ factor. This completes the proof.  

We note that another way to describe situation in the part (1) cases is that for an $\mathfrak{sl}_2$-triple through the regular element in $\mathcal N_0$, the corresponding image of $SL_2$ in $G$ is an $SL_2$.  
\end{proof}

\begin{corollary}
Let $G$ be of classical type $X_n$.
Then for each $\0$ and $x:=x_{\mathcal J} \in \0'$ as above, the map
$$\pi_1(\cS^{\reg}_{x, \0}) \to A(\0)$$
is injective and its image is equal to the subgroup generated
by $\{ \sigma_i \ | \ i \in \mathcal J\}$.  
\end{corollary}

\begin{proof}
Let $\mathcal J$ be a subset of $J$.
Consider a subset $\mathcal J' \subseteq \mathcal J$. 
Then we have homomorphisms
$$\pi_1(\cS^{\reg}_{x_{\mathcal J'}, \0}) \to \pi_1(\cS^{\reg}_{x_{\mathcal J}, \0}) \to
\pi_1(\0) \twoheadrightarrow A(0).$$  
By the proposition, the image of the composite map contains $x_{\mathcal J'}$.
Since $\pi_1(\cS^{\reg}_{x, \0}) \simeq H$  for $H \simeq
\frakS_2^{|\mathcal J|}$ and this holds for every such $\mathcal J'$,
the result holds.
This also show that the composite of maps from the upper left corner to the lower right corner in the proof of the proposition is injective in the classical groups.
\end{proof}

\begin{corollary} \label{cor:i_A_injective_all_types} 
The map $i_A$ is injective in all types. 
For classical types, the image $H(\0)$ of $i_A$ is 
the subgroup of $A(\0)$ generated by $\{\sigma_i \ | \ i \in J\}$, which lies in $\Ab'$.  
In all types, the maps $i_\pi$ and $i_{\Ab}$ are also
injective.
\end{corollary} 

\begin{proof}
For classical types, this follows from the previous corollary and the description of $\bar A'$ and $N$ in Proposition \ref{prop:ALT_A2Abar}.  

For exceptional groups, we use Table \ref{table:attaching_conj_classes} to see
that $i_A$ is injective and its image equals $A(\0)$ for non-trivial special pieces, except for $E_8(b_6)$, where $H(\0)$ can be any of the conjugate order $2$ subgroups of $A(\0)\simeq S_3$.  This shows that $H(\0)$ need not be normal in general.
\end{proof}

\begin{remark}
In Proposition \ref{prop:key_slice_typeA}, 
for each $\0'$ in $\cP(\0)$, it is possible to make the choice 
of $\frakm$ unique (up to $G$-conjugacy), and thus the conjugacy class $C$ in $A(\0)$, if we impose the following three conditions:  (1) $e \in \frakm$ is even; (2) any minimal Levi subalgebra $\frakm'$ in $\frakm$ that meets $\0'$ is also a Levi subalgebra in $\g$; and (3) $\0'$ is maximal among all orbits in $\cP(\0)$, under the inclusion ordering, that satisfy Proposition \ref{prop:key_slice_typeA} and (1) and (2).  

For example, when $\0 = G_2(a_1)$ in $G_2$ and $\frakm$ is of type $A_1+\tilde{A}_1$, the proposition holds for both $\0'=A_1$  and $\0'=\tilde{A_1}$, but $\tilde{A_1}$ would be maximal, so we would associate $\frakm$ to $\tilde{A_1}$
while the $A_2$ pseudo-Levi subalgebra would be associated with $A_1$.

In Table \ref{table:attaching_conj_classes}, the pairs $\frakm$ and $\frakm'$ satisfy the conditions (1)--(3).  We also note that the definition of $x_0 \in \frakm_0$ is clear in the table, except for the case of $\0'= A_5+A_2+A_1$.  In the other case, the components embed in the given order and the zero orbit is omitted.   

Finally, with the uniqueness conditions in effect,
we recover a bijection between 
the orbits in $\cP(\0)$ and the conjugacy classes of $H(\0)$, 
which with the one described in \cite[Theorem 0.4]{Lusztig:unipotent}
in the exceptional groups and presumably the classical groups too.
\end{remark}

We finish the section by defining the a natural complement $K^{\nat}$ to $H$
in $A(\0)$.
In classical types, we have described $A(\0)$ as Coxeter group with simple reflections $\sigma_i$ with $i \in \vartheta_0 \cup \vartheta_1$, with natural splitting, by Proposition \ref{prop:ALT_A2Abar}, into $N \times \bar{A}'$ with $N$ and $\bar{A}'$ as parabolic subgroups.  We now have described $H$ as parabolic subgroup of $\bar{A}'$.  Let $B$ be the parabolic subgroup of $A(0)$ generated 
by $\{\sigma_i \ | \ i \in \vartheta_0 \backslash J \}$.  Hence
the decomposition into an internal direct product of parabolic subgroups 
$A(\0) = N \times B \times H$. 

\begin{definition}\label{definition:natural_complement_to_H}
For classical groups define $K^{\nat} = N \times B$, a subgroup of $A(\0)$.  For exceptional groups and non-trivial special pieces, $K^{\nat}=1$, except for $E_8(b_6)$ where $K^{\nat}$ is the cyclic group of order $3$ in $A(\0)\simeq \frakS_3$.
\end{definition}

The image in $\bar{A}(\0)$ of the splitting matches the description in the introduction of \cite{Achar-Sage-perverse}, where $H$ is called $F$.  Hence, the choice of $K^{\nat}$ gives the solution to Lusztig's conjecture that matches the one by Kraft-Procesi, according to the last sentence of the introduction of {\it loc. cit}.
Note that $K^{\nat}$ is always normal in $A(\0)$, by observation in the 
exceptional groups and by the fact that $A(\0)$ is abelian in the classical
types.

\begin{table}
		\caption{Proposition \ref{prop:key_slice_typeA} for the exceptional types}\label{table:attaching_conj_classes}
		\begin{center}
			\begin{tabular}{|c | c| c|c |c| c| c|}
				\hline
				$G$ & $\0$ & $H(\0)$ & $\0'$ & $(\frakm,e)$ &  Slice in $\frakm$ & $C$ \\
                \hline & & & &  & & \\[-0.3cm]
                $G_2$ & $G_2(a_1)$ & $\frakS_3$ & $\tilde{A}_1$ & $A_1 + \tilde{A}_1$ & $A_1$ & $(12)$ \\
                      &   &   & $A_1$ & $A_2$ & $A_2$  & $(123)$ \\
				\hline & & & &  & & \\[-0.3cm]
				$F_4$ & $\tilde{A}_1$ & $\frakS_2$ & $A_1$ & $2A_1$ &  $A_1$ & $(12)$\\
                        & $F_4(a_3)$ & $\frakS_4$ & $C_3(a_1)$ & $C_3(a_1)+A_1$ & $A_1$ & $(12)$\\
                        &  & & $A_1 + \tilde{A}_2$ & $A_2 + \tilde{A}_2$ & $A_2$ & $(123)$ \\
				         &  & & $B_2$ & $B_4(a_2)$ & $A_1\times A_1$ & $(12)(34)$ \\
                         &  & & $A_2 + \tilde{A}_1$ & $A_3 + \tilde{A}_1$ & $A_3$ & $(1234)$ \\

				\hline & & & &  & & \\[-0.3cm]
				$E_6$ & $A_2$ & $\frakS_2$ & $3A_1$ & $4A_1$ &  $A_1$ & $(12)$\\
                      & $D_4(a_1)$ & $\frakS_3$ & $A_3+A_1$ & $A_3+2A_1$ &  $A_1$ & $(12)$\\  
                      & & & $2A_2+A_1$ & $3A_2$ &  $A_2$ & $(123)$\\  
                    & $E_6(a_3)$ & $\frakS_2$ & $A_5$ &  $A_5+A_1$ &  $A_1$ & $(12)$\\
                
				\hline & & & & & &  \\[-0.3cm]
				$E_7$ & $A_2$ & $\frakS_2$ & $(3A_1)'$ & $(4A_1)'$ &  $A_1$ & $(12)$\\
                    &  $A_2+A_1$ & $\frakS_2$ &  $4A_1$ & $5A_1$ & $A_1$ & $(12)$\\
                     & $D_4(a_1)$ & $\frakS_3$ & $(A_3+A_1)'$ & $(A_3+2A_1)'$ &  $A_1$ & $(12)$\\  
                      & & & $2A_2+A_1$ & $3A_2$ &  $A_2$ & $(123)$\\  
				    & $D_4(a_1)+A_1$ & $\frakS_2$ & $A_3+2A_1$ & $A_3+3A_1$ &  $A_1$ & $(12)$\\  
								 
				& $D_5(a_1)$ & $\frakS_2$ & $D_4+A_1$ & $D_4+2A_1$  &  $A_1$ & $(12)$\\  

                & $E_6(a_3)$ & $\frakS_2$ & $(A_5)'$ & $(A_5+A_1)'$ &  $A_1$ & $(12)$\\  
                 & $E_7(a_5)$ &  $\frakS_3$ & $D_6(a_2)$ & $D_6(a_2)+A_1$ &  $A_1$ & $(12)$\\  
                    &  &  & $A_5+A_1$ & $A_5+A_2$ &  $A_2$ & $(123)$\\  
                & $E_7(a_3)$ &  $\frakS_2$ & $D_6$ & $D_6+A_1$ &  $A_1$ & $(12)$\\  
				\hline & & & & & &  \\[-0.3cm]
				
				$E_8$ 
				& $A_2$ & $\frakS_2$ & $3A_1$ & $(4A_1)''$ &  $A_1$ & $(12)$\\

                & $A_2+A_1$ & $\frakS_2$ & $4A_1$ & $5A_1$ & $A_1$ & $(12)$\\
                    
				& $2A_2$ & $\frakS_2$ & $A_2+3A_1$ & $A_2+4A_1$	&  $A_1$ & $(12)$\\		
                   & $D_4(a_1)$ & $\frakS_3$ & $A_3+A_1$ & $(A_3+2A_1)''$ &  $A_1$ & $(12)$\\  
                      & & & $2A_2+A_1$ & $3A_2$ &  $A_2$ & $(123)$\\   
                   & $D_4(a_1)+A_1$ & $\frakS_3$ & $A_3+2A_1$ & $A_3+3A_1$ &  $A_1$ & $(12)$\\  
                      & & & $2A_2+2A_1$ & $3A_2+A_1$ &  $A_2$ & $(123)$\\  
				& $D_4(a_1)+A_2$ & $\frakS_2$ & $A_3{+}A_2{+}A_1$ & $A_3{+}A_2{+}2A_1$ &  $A_1$ & $(12)$\\

	           & $A_4+2A_1$ & $\frakS_2$ & $2A_3$ & $D_4(a_1)+A_3$ & $A_1$ & $(12)$\\
                & $D_5(a_1)$ & $\frakS_2$ & $D_4+A_1$ & $D_4+2A_1$  &  $A_1$ & $(12)$\\  
                & $E_6(a_3)$ & $\frakS_2$ & $A_5$ & $(A_5+A_1)''$ &  $A_1$ & $(12)$\\
                 & $E_8(a_7)$ &  $\frakS_5$ & $E_7(a_5)$ & $E_7(a_5)+A_1$ &  $A_1$ & $(12)$\\  
                    &  &  & $E_6(a_3)+A_1$ & $E_6(a_3)+A_2$ &  $A_2$ & $(123)$\\ 
                    &  &  & $D_6(a_2)$ & $D_8(a_5)$ &  $A_1 \times A_1$ & $(12)(34)$\\ 
                    &  &   & $D_5(a_1)+A_2$ & $D_5(a_1)+A_3$ &  $A_3$ & $(1234)$\\ 
                    &  &   & $A_5+A_1+0$ & $A_5+A_2+A_1$&  $A_2 \times A_1$ & $(123)(45)$\\ 
                    &  &   & $A_4+A_3$ & $2A_4$  &  $A_4$ & $(12345)$\\ 
				  			
				& $D_6(a_1)$ & $\frakS_2$ & $D_5+A_1$ & $D_5+2A_1$ &  $A_1$ & $(12)$\\  	 
                & $E_8(b_6)$ & $\frakS_2$ & $A_7$ & $D_8(a_3)$ &  $A_1$ & $(12)$\\  	
            & $E_7(a_3)$ & $\frakS_2$ & $D_6$ & $D_6+A_1$ &  $A_1$ & $(12)$\\ 
                & $E_8(b_5)$ & $\frakS_3$ & $E_7(a_2)$ & $E_7(a_2)+A_1$ &  $A_1$ & $(12)$\\  	
	          & & & $E_6+A_1$ & $E_6+A_2$ &  $A_2$ & $(123)$\\ 

                & $E_8(a_5)$ & $\frakS_2$ & $D_7$ & $D_8(a_1)$ &  $A_1$ & $(12)$\\  
                
                & $E_8(a_3)$ & $\frakS_2$ & $E_7$ & $E_7+A_1$  &  $A_1$ & $(12)$\\  
                \hline 
			\end{tabular}
		\end{center}
	\end{table}

\section{Conceptual proof}\label{conceptualsec}

\subsection{General framework}\label{subsec:framework}

The following general lemma is standard, which we record here for completeness.

\begin{lemma}\label{lem:lifted-slice-contraction}
	Let $\pi:\widetilde S\to S$ be a finite $\C^\times$-equivariant morphism of affine varieties. Assume that the $\C^\times$-action on $S$ contracts $S$ to a point $e$. Then the lifted $\C^\times$-action on $\widetilde S$ extends to an $\A^1$-action. Each connected component $Z$ of $\widetilde S$ is contracted to a unique point $\tilde e_Z\in \pi^{-1}(e)$, and the map
	\[
		\operatorname{Comp}(\widetilde S) \to \pi^{-1}(e),\qquad Z \mapsto \tilde e_Z
	\]
	is a bijection.
	\end{lemma}
	
\begin{proof}
	Put $A=\C[S]$ and $B=\C[\widetilde S]$, then they are graded algebras with grading induced by the $\C^\times$-actions and graded components denoted by $A_d$ and $B_d$, $d\in \Z$, respectively. The condition that the $\C^\times$-action on $S$ contract $S$ to a point $e$ is equivalent to the condition that the induced grading on $A$ be nonnegative with zero-th degree component $A_0\simeq \C$.
	
	Since $\pi$ is finite and $\cm$-equivariant, $B$ is a finite graded $A$-algebra. We first show that $B$ is also nonnegatively graded. If $0\ne b\in B_d$ with $d<0$, then $b$ is integral over $A$, hence satisfies a monic equation
	\[
	b^n+a_1b^{n-1}+\cdots+a_n=0.
	\]
	Taking the homogeneous part of degree $nd$, we may assume $a_i\in A_{id}$. Since $id<0$ and $A$ is nonnegatively graded, all $a_i$ vanish. Thus $b^n=0$, contradicting the reducedness of $B$, which follows from the normality of $\widetilde S$. Hence $B_d=0$ for $d<0$. Therefore the lifted $\C^\times$-action extends to an $\A^1$-action $\bar{\beta}:\A^1\times \spec(B)\to \spec(B)$, i.e., its restriction to the open subvariety $\cm \times \spec(B) \to \spec(B)$ is the $\cm$-action morphism $\beta$. The corresponding morphism $\C[t]\otimes B\to B$ between coordinate rings is given by sending a homogeneous element $b\in B_d$ to $t^d b\in \C[t]\otimes B$.
	
	Let $Z$ be an irreducible component of $\widetilde S$. Since $\C^\times$ is connected, $Z$ is $\C^\times$-stable. The extended action gives a morphism $\rho_Z:Z\to \widetilde S$, $z\mapsto \lim_{t\to0}t\cdot z$. By equivariance, $\pi(\rho_Z(z))=\lim_{t\to0}t\cdot \pi(z)=e$, so $\rho_Z(Z)\subseteq \pi^{-1}(e)$. Since $\pi^{-1}(e)$ is finite and $Z$ is irreducible, $\rho_Z$ is constant. Thus $Z$ is contracted to a unique point $\tilde e_Z\in\pi^{-1}(e) \cap Z$. Every point of $\pi^{-1}(e)$ lies in some connected component of $\widetilde S$, and is fixed by $\C^\times$, because $\pi^{-1}(e)$ is finite and $\C^\times$ is connected. Hence the above map is surjective. 
\end{proof}

We now consider the following special situation: let $\0,\0'$ be nilpotent orbits in $\g$ with $\0'\subset\overline\0$ and let $\{ e,h,f\}$ be an $\mathfrak{sl}_2$-triple with $e\in\0'$. Let ${\mathcal S}$ be the Slodowy slice singularity from $\0'$ to $\overline\0$, as defined in the Introduction. Then the regular part of $\cS^{\reg} = \cS \cap \0$ is embedded as a closed subvariety of $\0$.

Recall the \emph{Kazhdan $\cm$-action} on $\fg$, defined by 
\begin{equation}\label{eq:Kazhdan}
   z \cdot \xi = z^{-2} \Ad(\gamma(z))(\xi), \qquad z \in \cm, \ \xi \in \fg. 
\end{equation}
where $\gamma: \cm \to G$ is the co-character associated to $h$. This stabilizes the Slodowy slice $\cS = \cS_{e,\0}$ from $\0'$ to $\overline{\0}$, contracting it onto $e$. 

Let $K$ be any subgroup of $\pi_1(\0)$, then (the conjugacy class of) $K$ determines a covering of $\0$, which we denote by $\pi_K: \widetilde{\0}_K \to \0$ or just $\pi:\widetilde\0\to\0$ . Let $\widetilde{X}_K = \C[\widetilde{\0}]$ denote the affinization of $\widetilde{\0}$ and we still use $\pi_K$ to denote the natural projection $\widetilde{X}_K \to \overline{\0}$. 
Consider the pre-image $\widetilde{\mathcal S}:=\pi_K^{-1}({\mathcal S}) = \widetilde{X}_K \times_{\bar{\0}} \cS$ of $\cS$ in $\widetilde{X}_K$. This is a closed subvariety of $\widetilde{X}_K$, possibly disconnected. We will also consider coverings coming from subgroups $K$ of $A(\0)$ (resp. $\Ab(\0)$). Namely, we take the pre-image $K'$ of $K$ in $A(\0)$ (resp. $\Ab(\0)$) and consider the covering $\widetilde{\0}_{K'}$ of $\0$ and its affinization $\widetilde{X}_{K'}$. By abuse of notation, we will simply denote them by $\widetilde{\0}_K$ and  $\widetilde{X}_K$ respectively and denote the corresponding projection maps to $\overline{\0}$ by $\pi_K$. 

Set $\widetilde{\cS}^\circ := \widetilde{\0} \times_{\0} \cS^{\reg}$, which is a closed subvariety of $\widetilde{\0}$ and is a (possibly disconnected)  finite \'{e}tale covering of $\cS^{\reg}$.

We first show the following basic properties of the slice $\widetilde{\cS}$, which should be well-known to experts. But we record the proof for completeness.

\begin{proposition}\label{prop:slice}
	The following statements hold:
	\begin{enumerate}[ref=(\arabic*)]
		\item \label{item:slice_smooth} 
		The $G$-action morphism $G \times \widetilde{\cS} \to \widetilde{X}_K$ is smooth;

		\item \label{item:slice_reduced_normal}
		$\widetilde{\cS}$ is a reduced closed normal subscheme/subvariety of $\widetilde{X}_K$;
		
		\item \label{item:slice_bijections}
		We have natural bijections among the following sets:
		\begin{enumerate}
			\item the set of connected components of $\widetilde{\cS}$;
			\item the set of irreducible components of $\widetilde{\cS}$;
			\item the set of connected components of $\widetilde{\cS}^{\circ}$;
			\item the pre-image $\pi_K^{-1}(e)$ of $e$ in $\widetilde{X}_K$.
		\end{enumerate}
		More precisely, irreducible components of $\widetilde{\cS}$ coincide with its connected components; the bijection from (c) to (a) or (b) is given by taking closure of the component in $\widetilde{X}_K$; the bijection from (a) or (b) to (c) is given by taking intersection of the component with $\pi_K^{-1}(e)$.

		\item \label{item:slice_dim}
		$\widetilde{\cS}$ is equidimensional of dimension $\codim_{\overline{\0}}(\0')$;
		
		\item \label{item:slice_surjective}
		$\pi_K$ maps each irreducible component of $\widetilde{\cS}$ surjectively to a unique irreducible component of $\cS$;
	\end{enumerate}
\end{proposition}

\begin{proof}
We first prove part \ref{item:slice_smooth}. By the same argument in \cite[\S\,12.4]{Kraft-Procesi:classical} and the well-known fact that $S$ is transverse to every adjoint orbit which it meets (see, for example, \cite[Section 2]{Gan-Ginzburg:quant_Slodowy}), the $G$-action morphism $G \times S \to \fg$ is smooth at $(1,x)$ for any $x\in S$ by \cite[Proposition 17.11.1]{Grothendieck:EGA4}, and hence is smooth for any $(g,x) \in G \times S$ by left translation. By base change, the $G$-action morphisms $G \times \cS \to \overline{\0}$ and $G \times \widetilde{\cS} \to \widetilde{X}_K$ are also smooth. 

Given part \ref{item:slice_smooth}, the reducedness (resp. normality) of $G \times \widetilde{\cS}$ (resp. $\widetilde{\cS}$) then follows from that of $\widetilde{X}_K$ and $G$ by \cite[Proposition 17.5.7]{Grothendieck:EGA4} (resp. \cite[Proposition 17.5.8]{Grothendieck:EGA4} and Serre's criterion for normality). This proves part \ref{item:slice_reduced_normal}.

Now consider part \ref{item:slice_bijections}. Since $\widetilde{X}_K$ is normal, connected components of $\widetilde{\cS}$ are the same as irreducible components, each of which is also normal. This identifies the set (a) with (b) in \ref{item:slice_bijections}. Again by the argument in \cite[\S\,12.4]{Kraft-Procesi:classical}, $e$ is an isolated point in $\cS \cap \0'=G \cdot e$ and hence $\cS \cap \0' = \{e\}$ thanks to the contracting Kazhdan $\cm$-action. By \Cref{lem:lifted-slice-contraction}, the lifted Kazhdan $\cm$-action on $\widetilde{X}_K$ contracts every connected component of $\widetilde{\cS}$ to a unique point $\widetilde{e} \in \pi_K^{-1}(e)$, which defines a bijection between $\pi_K^{-1}(e)$ and connected components of $\widetilde{\cS}$, denoted as $\widetilde{e} \mapsto \widetilde{\cS}_{\widetilde{e}}$. Clearly $\widetilde{\cS}_{\widetilde{e}} \cap \pi_K^{-1}(e) = \{\widetilde{e}\}$.

Note that $\cS^{\reg} = \cS \cap \0$ is open dense in $\cS$ (in particular, nonempty) since the action morphism $G \times \cS \to \overline{\0}$ is smooth and $\overline{\0}$ is irreducible. For similar reason, $\widetilde{\cS}^{\circ}$ is also open dense in $\widetilde{\cS}$.
Then the subset $U_{\widetilde{e}} := \widetilde{\cS}_{\widetilde{e}} \cap \widetilde{\0}_K = \widetilde{\cS}_{\widetilde{e}} \cap \widetilde{\cS}^{\circ}$ is open and dense in $\widetilde{\cS}_{\widetilde{e}}$, and hence connected. Therefore $U_{\widetilde{e}}$ is a connected component of $\widetilde{\cS}^{\circ}$. This sets up the bijections among the sets (a)-(d).

Now consider part \ref{item:slice_dim}. Since $\overline{\0}$ is irreducible, we have $\dim_e \cS = \codim_{\overline{\0}}(\0')$ by \cite[\S\,12.4]{Kraft-Procesi:classical}. The same argument applies to $\widetilde{\cS}$ and implies that for any irreducible component $\widetilde{\cS}_{\widetilde{e}}$ of $\widetilde{\cS}$, we have $\dim_{\widetilde{e}} \widetilde{\cS}_{\widetilde{e}} = \codim_{\overline{\0}}(\0')$. 

Finally, the image $\pi_K(\widetilde{\cS}_{\widetilde{e}})$ of $\widetilde{\cS}_{\widetilde{e}}$ is closed and irreducible, and is contained in a unique irreducible component $C$ of $\cS$ of the same dimension $\codim_{\overline{\0}}(\0')$. Therefore $\pi_K(\widetilde{\cS}_{\widetilde{e}}) = C$ and this proves part \ref{item:slice_surjective}.
\end{proof}


Suppose now that $\cS \simeq V / H$, where $H$ is a finite group acting on a (symplectic) vector space $V$ by symplectic reflections, so that the pre-image of $\cS^{\reg}$ under the quotient map $V \twoheadrightarrow V/H \simeq \cS$ is the dense open subset $V^\circ$ for which $H$ acts freely and $\codim_V(V \backslash V^\circ) \geq 2$. We have $\cS^{\reg} \simeq V^\circ / H$ and the (algebraic) fundamental group $\pi_1(\cS^{\reg}) = \pi_1(\cS^{\reg}, x)$ of $\cS^{\reg}$ (with respect to any chosen geometric base point $x$) is $H$. The inclusion map $\cS^{\reg} \hookrightarrow \0$ induces a homomorphism $i_\pi: \pi_1(\cS^{\reg}) = H \to \pi_1(\0)$ of fundamental groups. Note that the definition of the fundamental group depends on the choice of a base point in $\cS^{\reg}$ and the slice $\cS$ also depends on the choice of the $\mathfrak{sl}_2$-triple $\{ e,h,f\}$ (which is unique up to $G$-conjugation). Therefore the image of $H$ in $\pi_1(\0)$ is only well-defined up to conjugation in $\pi_1(\0)$. In general, this image might not be a normal subgroup of $\pi_1(\0)$, see Example \ref{coverexamples}, c).
The composition of $i_\pi$ with the quotient maps $\pi_1(\0) \to A(\0)$ (resp. $\pi_1(\0) \to \Ab(\0)$) is denoted as $i_A: H \to A(\0)$ (resp. $i_{\Ab}: H \to \Ab(\0)$). Note the the three maps $i_\pi$, $i_A$ and $i_{\Ab}$ are not necessarily injective in general, see again Example \ref{coverexamples}, c). Let $H_\pi$ (resp. $H_A$, $H_{\Ab}$) denote the image of $i_\pi$ (resp. $i_A$, $i_{\Ab}$).

\begin{proposition}\label{prop:decomp}
	The sets (a)-(d) in \cref{prop:slice} \ref{item:slice_bijections} are in further bijection with the set of double cosets $H \backslash \pi_1(\0) / K$, where $H$ acts by left translation on the left coset space $\pi_1(\0) / K$ via the homomorphism $i_\pi: H \to \pi_1(\0)$.

	Moreover, we have the decomposition of $\widetilde{\cS}$ into connected components,
	\[
		\widetilde{\cS} \simeq \bigsqcup_{g \in H \backslash \pi_1(\0) / K } V/ K_g, 
	\]
	where $g$ runs through representatives of all double cosets in $H \backslash \pi_1(\0) / K$ and $K_g := i_\pi^{-1}(gK g^{-1})$.
\end{proposition}

\begin{proof}
	We follow the same notation as in the proof of \cref{prop:slice}.
	Note that the restriction of $\pi_K$ to each $U_{\widetilde{e}}$ is a connected finite \'{e}tale covering of $\cS^{\reg}$. The claimed bijections and the remaining statement follow easily from Grothendieck's definition of algebraic fundamental groups in terms of automorphism groups of fiber functors \cite{SGA1} (see also \S\,5.4 and Theorem 5.4.2 of \cite{Szamuely}), as follows. Pick any geometric point $x: \spec(\C) \to \cS^{\reg}$, then the geometric fiber $\operatorname{Fib}_x(\widetilde{\0}) := \widetilde{\0} \times_\0 \spec(\C)$ of $\widetilde{\0}$ over $x$ is a transitive (left) $\pi_1(\0, x)$-set, which can be identified as the coset space $\pi_1(\0, x) / K$. $\operatorname{Fib}_x(\widetilde{\0})$ coincides with the geometric fiber $\operatorname{Fib}_x(\widetilde{\cS}^\circ)$ of $\widetilde{\cS}^\circ$ over $x$. The action of $H = \pi_1(\cS^{reg}, x)$ on $\operatorname{Fib}_x(\widetilde{\cS}^\circ)$ factors through $\pi_1(\0,x)$. 
	
	The connected components of $\widetilde{\cS}^\circ$ correspond bijectively to the $H$-orbits of $\operatorname{Fib}_x(\widetilde{\0}) \simeq \pi_1(\0, x) / K$. The connected component $U_{\widetilde{e}}$ corresponding to the double coset $H g K$ is a connected covering of $\widetilde{S}^{reg}$ corresponding to the $H$-stablizer of $gK$, which is exactly $K_g$. Therefore we have $U_{\widetilde{e}} \simeq V^\circ / K_g$. Since $\widetilde{X}_K$ is normal, the closure  of $U_{\widetilde{e}}$, which is the connected component $\widetilde{\cS}_{\widetilde{e}}$ of $\widetilde{\cS}$, is also normal by \cref{prop:slice}. Therefore by $\codim_V(V \setminus V^\circ) \geq 2$,  $\widetilde{\cS}_{\widetilde{e}}$ is isomorphic to $V/K_g$.
\end{proof}

\begin{corollary}\label{cor:H_injection}
	The following conditions are equivalent:
	\begin{enumerate}
		\item The map $i_\pi: H \to \pi_1(\0)$ is injective.

		\item For any $K < \pi_1(\0)$, the pre-image $\widetilde{\mathcal S}$ of $\cS$ in $\widetilde{X}_K$ is smooth if and only if $g K g^{-1} \cap H_\pi$ is trivial for any $g \in \pi_1(\0)$.
	\end{enumerate}
	When the conditions are satisfied, so that $H$ is identified with $H_\pi$, $\widetilde{\cS}$ is isomorphic to a disjoint union of copies of $V$.
    If furthermore $K$ is normal, then $HK$ is a subgroup (the semi-direct product of $H$ and $K$), and $\widetilde{\cS}$ is comprised of 
    $$[\pi_1(\0):HK] = \frac{|\pi_1(\0)|}{|H||K|} \text{ copies of } V. $$ 
\end{corollary}

\begin{proof}
	It follows by Proposition \ref{prop:decomp} that the pre-image $\widetilde{\cS}$ in (2) is smooth if and only if the group $K_g$ for each connected component of $\widetilde{\cS}$ is trivial. 
(Here we recall that $H$ preserves the symplectic form on $V$, so it contains no complex reflections.)
On the other hand, $K_g = i_\pi^{-1}(gK g^{-1}) = i_\pi^{-1}(gK g^{-1} \cap H_\pi) = i_\pi^{-1}(1) = \ker i_\pi$, so $K_g$ is trivial if and only if $i_\pi$ is injective. 
For the last statement, since $K$ is normal, the orbits of $K$ on the set $\pi_1(\0)/H$ all have the same size, namely $[K:K\cap H]$. This is equal to $|K|$, so the formula follows.
\end{proof}

One application of \cref{prop:decomp} and \cref{cor:H_injection} is the following proposition. The case of $\g=\mathfrak{sl}_5$ was previously stated in \cite[Proposition 12.9]{FJLS}, while part (3) for general $\g=\mathfrak{sl}_n$ was previously stated before \cite[Theorem 1.4]{BBFJLS}, both without proof. 
Part (2) of \cref{prop:typeA_fundamental} was previously mentioned in Example \ref{example:typeA_fundamental} and was used in an essential way in \cref{prop:key_slice_typeA}.

Recall that the center $Z(G)$ of $G=\SL_n$ is $\mu_n \cdot I_n$, where $\mu_n$ is the multiplicative group of $n$-th roots of unity and $I_n$ is the identity matrix.

\begin{proposition} \label{prop:typeA_fundamental}
	Let $\g = \mathfrak{sl}_n$.  Let $\0 = \0_{reg}$ be the regular nilpotent orbit and $\0' = \0_{subreg}$ be the subregular orbit. Let $\widetilde{X}$ be the affinization of the universal cover $\widetilde{\0}$ of $\0$, and $\pi: \widetilde{X} \to \overline{\0}$ be the projection. Then the following hold:

	\begin{enumerate}
		\item $\cS \simeq \C^2/\mu_n$, where $\mu_n$ is the multiplicative group of $n$-th roots of unity acting on $\C^2$ by 
			\[ \xi \cdot (u,v) = (\xi u,\xi^{-1}v), \quad \forall\, \xi \in \mu_n.\]
		In particular, $H \simeq \pi_1(\cS^{\reg}) \simeq\mu_n$;  
		\item The homomorphism $i_\pi: H \to \pi_1(\0)$ is an isomorphism;
		\item For any $e \in \0'$, the pre-image $\pi^{-1}(e)$ is a singleton, that is, the restriction $\pi|_{\0'}: \pi^{-1}(\0') \to \0'$ is an isomorphism of varieties;
		\item The restriction $\pi: \pi^{-1}(\overline{\0'}) \to \overline{\0'}$ of $\pi$ to the (reduced) pre-image of $\overline{\0'}$ is an isomorphism of varieties;
		\item $\widetilde{X}$ is smooth at the points over the subregular orbit $\0'$.
	\end{enumerate}
\end{proposition}

\begin{proof}
	Part (1) is proved in \cite{Slodowy:book}. Now consider part (2).
	Let $G=SL_n(\mathbb C)$, let
	\[
	e=\sum_{i=1}^{n-2}E_{i,i+1}
	\]
	be the subregular nilpotent element of Jordan type $(n-1,1)$, where $E_{i,j}$ stands for the elementary $n \times n$ matrix with $1$ in the $(i,j)$-th entry and $0$ elsewhere. Choose the standard $\mathfrak{sl}_2$-triple $(e,h,f)$ supported on the upper-left $(n-1)\times(n-1)$-block. Put
	\[
		x=e+E_{n-1,n}=\sum_{i=1}^{n-1}E_{i,i+1}.
	\]
	Then $x \in \cS^{\reg} = \cS \cap \0$.
	Let
	\[
	\zeta=e^{2\pi i/n},\qquad c(t)=e^{2\pi i t/n},\qquad t \in [0,1],
	\]
	and define
	\[
	g(t)=\operatorname{diag}\bigl(c(t),c(t),\ldots,c(t),c(t)^{1-n}\bigr)\in SL_n(\mathbb C).
	\]
	Via the $G$-equivariant isomorphism $\0 \simeq G/G^x$, 
	$g(t)$ induces the loop
	\[
		\gamma(t)=\operatorname{Ad}_{g(t)}(x) = e + e^{2\pi i t}E_{n-1,n}, \quad t \in [0,1],
	\]
	in $\0$, based at $x$.
	Since $g(t)$ centralizes the chosen $\mathfrak{sl}_2$-triple $(e,h,f)$ and hence stabilizes the Slodowy slice $\cS$ and $\cS^{\reg}$, the loop $\gamma(t)$ also lies in $\cS^{\reg}$. 
	The homotopy class of $\gamma(t)$ defines an element $[\gamma]$ of $\pi_1(\cS^{\reg},x)$.

	We compute the image of the loop $\gamma(t)$ in $\0$ and its image in $\pi_1(\0)$. Since $G=SL_n(\mathbb C)$ is simply connected, we have isomorphisms
		\[ \pi_1(\0) \simeq A_G(\0) \simeq G^x/(G^x)^\circ  \]
	by \cite[Lemma 6.1.1]{C-M}. The composition of the natural inclusion $Z(G) \hookrightarrow G^x$ with the quotient map $Z(G) \twoheadrightarrow Z(G)/(G^x)^\circ$ induces  isomorphisms
		\[ \mu_n \simeq Z(G) = \mu_n \cdot I_n \xrightarrow{\sim} G^x/(G^x)^\circ \simeq A_G(\0) \simeq \pi_1(\0).\]
	The path $g(t)$ satisfies $g(0)=I_n$, while
	\[
	g(1)=\operatorname{diag}(\zeta,\ldots,\zeta,\zeta^{1-n})=\zeta I_n\in Z(G).
	\]
	Consequently the homomorphism $i_\pi: H \simeq \pi_1(\cS^{\reg}) \to \pi_1(\0)$ maps $[\gamma]$ to the generator $\zeta \cdot I_n \in Z(G) \simeq \pi_1(\0)$. Therefore $i_\pi$ is surjective, and hence an isomorphism since both groups are isomorphic to $\mu_n$.

	Now part (3) follows from part (2) and the following immediate corollary of \cref{prop:decomp}: when $\widetilde{X}$ is the affinization of the universal cover $\widetilde{\0}$ of $\0$, $\pi^{-1}(e)$ is a singleton if and only if $i_\pi: H \to \pi_1(\0)$ is surjective.

	A second proof follows by first showing (3).   Let $\ic_{\widetilde X}$ denote the intersection cohomology complex of the affine cover.
    Then
    \[
        \pi_* \ic_{\widetilde X} = \bigoplus_{\chi} \ic_\chi
    \]
    where $\chi$ ranges over the characters of the cyclic group $\mu_n = Z(G)$.

    By \cite{Lusztig-Spaltenstein:classical}, for a non-trivial central character $\chi$, say of order $d$, the only nilpotent orbits supporting a local system with central character $\chi$ have all parts divisible by $d$; the partition of the subregular orbit is $[n-1,1]$, hence it can only support local systems with trivial central character. Therefore, taking the fiber at $e$, we get
    \[
    (\pi_* \ic_{\widetilde X})_e
    = \bigoplus_{\chi \in \widehat{\mu_n}}(\ic_\zeta)_e = (\ic_\cN)_e = \Q;
    \]
    on the other hand,
    \[
    (\pi_* \ic_{\widetilde X})_e 
    = H^\bullet(\pi^{-1}(e), \ic_{\widetilde X}) 
    = \bigoplus_{x \in \pi^{-1}(e)} H^\bullet( \{x\}, \ic_{\widetilde X})
    = \bigoplus_{x \in \pi^{-1}(e)} (\ic_{\widetilde X})_{x}
    \]
    which contains at least $|\pi^{-1}(e)|$ copies of $\Q$;
    this implies that $\pi^{-1}(e)$ is a singleton.


	We now prove part (4). By \cite{Kraft-Procesi:typeA_normal}, the closures $\overline{\0}$ and $\overline{\0'}$ are both normal. Therefore the finite morphism $\pi: \widetilde{X} \to \overline{\0}$ satisfies the going-down property and hence every irreducible component of $\pi^{-1}(\overline{\0'})$ dominates $\overline{\0'}$. Therefore by part (3), $\pi^{-1}(\overline{\0'})$ is the closure of $\pi^{-1}(\0')$ in $\widetilde{X}$ and hence irreducible. Then the finite birational morphism $\pi: \pi^{-1}(\overline{\0'}) \to \overline{\0'}$ is an isomorphism of varieties by Zariski's Main Theorem.

	Finally, part (5) follows from part (2) and \cref{cor:H_injection}.
\end{proof}

\begin{remark}
In type $B_n$ it is also true that on the subregular orbit there are no local systems with non-trivial central character.
Hence the proof of (3) via intersection cohomology extends easily to this case.
It can be deduced from Prop. \ref{prop:decomp} that in this case $i_\pi$ is surjective.
Since in this case, ${\mathcal S}=\C^2/\Z_{2n}$ (with a monodromy action of $\mathfrak{S}_2\cong BD_{4n}/\Z_{2n}$, where $BD_{4n}$ denotes the binary dihedral group of order $4n$, it follows that the generic singularity of $\widetilde{X}$ is $\C^2/\Z_n$.
In types $E_6$ and $E_7$, the subregular orbit does support a local system with non-trivial central character, so this argument doesn't apply directly.
However, it can be checked from the intersection cohomology data in GAP that the stalk, at a subregular element, of a local system on $\0_{\reg}$ with non-trivial central character is zero.
Statement (3) follows in these cases, so that the generic singularity of $\widetilde{X}$ is a $D_4$, resp. $E_6$ singularity in type $E_6$, resp. $E_7$.
We expect statement (3) to also hold in types $C$ and $D$.
This would lead to the following uniform statement: if the subregular singularity is given by a pair $\Gamma_1\lhd \Gamma_2$ of finite subgroups of ${\rm SL}_2$, with $\Gamma_2/\Gamma_1$ acting via monodromy on $\C^2/\Gamma_1$, then the generic codimension 2 singularity of $\widetilde{X}$ is the quotient of $\C^2$ by the commutator subgroup $(\Gamma_1,\Gamma_2)$.
\end{remark}

\subsection{Main theorem} \label{subsec:mainthm}

We can now prove the main theorem of the paper, as stated in the introduction.
Recall the definition of the subgroup $H(\0) \subseteq \pi_1(\0)$ and 
the computation of its image from \cref{cor:i_A_injective_all_types}.  Given any $K < \pi_1(\0)$, we consider the 
covering $\pi_K: \widetilde{X}_K \to \overline{\0}$ and let $\widetilde{\mathcal P}_K := \pi_K^{-1} (\cP(\0))$.

\begin{theorem}\label{conceptualthm}
For any $K < \pi_1(\0)$, $\widetilde{\cP}_K$ is smooth if and only if $K \cap H(\0) = 1$. 

If $K$ is normal, $K\cap H(\0)=1$, and $[\pi_1(\0):K] = |H(\0)|$, then $\widetilde{\cP}_K$ is a solution to Lusztig's special pieces conjecture.  
The pre-image $K$  in $\pi_1(\0)$ of the subgroup
$K^{\nat} \subseteq A(\0)$ is one such solution.
\end{theorem}

\begin{proof}
By the Main Theorem in \cite{FJLS23}, ${\mathcal S}\cong V/H$, where $V$ is a sum of evenly many copies of the reflection representation for a Coxeter group $H$.  By \cref{cor:i_A_injective_all_types}, $i_\pi$ is injective, hence the result follows from \cref{cor:H_injection}.  

Since $K^{\nat}$ is normal in $A(\0)$, its pre-image $K$ is normal in $\pi_1(\0)$.  Also $K^{\nat}$ is a complement in $A(\0)$
to the image of $H(\0)$ in $A(\0)$.  Moreover, $H(\0)$ 
has trivial intersection with the image of $Z(G)$ in $\pi_1(\0)$ by \cref{cor:i_A_injective_all_types}.  
This implies that $\pi_1(\0)$ is a semidirect product of $K$ and $H(\0)$.
\end{proof}

We note that for the choice of $K= K^{\nat}$, the variety $\covP_K$ is equivariant for the adjoint group of $G$ in types $A-E$, but not 
necessarily in type $C'$ (see \cref{remark:center_G_in_Abar}).

\begin{remark}
In the classical types, our results apply for other degenerations beyond the ones corresponding to special pieces.  Let $\0_\lambda$ be any nilpotent orbit.  Replace $J$ by 
any subset $J'$ of $\theta = \theta_0 \cup \theta_1$ where
for all $i \in J'$, $\nu_i =\nu_{i+1}+2$ and such that if $i+1\in J$, then $h(i+1) \geq h(i)+2$.  Then all our results apply relative to the degenerations of the form in \eqref{equation:special_degenerations} and the corresponding minimal orbit $\0_m$ obtained.   In \cref{example:C15},
where $J = \{2,4\}$ and $H(\0) = \bar A'$, we could choose 
$J' = \{1,3,5\}$ or $\{1, 4\}$.  The former case corresponds to a metaplectic special piece, but the latter does not give a special piece either in type $C$ nor in $C'$.  It corresponds 
to $\0_m$ with partition $[9^2,6,3^2]$ and a ``piece" with $4$ orbits (two of which are special for type $C$ and two of which are special
for type $C')$.
\end{remark}


\begin{remark}\label{remark:meta_different_specials}
In the metaplectic case, there is the option to define both $\bar{A}(\0)$ and the special piece differently.  In this paper, we have adopted the definition 
of $\Ab(\0)$ from \cite{JLS:Duality}, but in that paper, $\bar{A}(\0)$ for type $D$ was not the one in this paper (typically, it has an extra factor of $\frakS_2$ and is not a quotient of $A_{adj}(\0)$).  So that dual orbits in types $C'$ and $D$ have matching canonical quotients, one could, instead of appending a $0$ to the end of $\lambda$ in type $C'$, define $\lambda_0$ to be $\infty$, viewed as an even number of height $0$.  Then $\nu$ becomes $\nu_1 = \infty > \nu_2 > \dots > \nu_{\ell} >0$ and $\sigma_1 = x_\infty x_{\nu_2} = x_{\nu_2}$, viewing $x_{\infty}$ as $1$.  In this way, the center of $G$ has image
$\prod_{i \in \vartheta_1} \sigma_i$ in $A(\0)$
and $\Ab(\0)$ is a quotient of $A_{adj}(\0)$, as in the other three classical types.  In the example of $[4,2,2]$ we would now have that $N$ is generated by $x_4$ and $\Ab'$ by $x_4x_2$, whereas in our definition 
from \S\ref{subsub:classical_component_groups}, $N$ is trivial and 
$\Ab(\0)$ has rank $2$.
 Following that line of thought, it would also change the special piece since $J$ would change and thus so would $H(\0)$.  For example, for $\lambda = [4,2,2]$, in our earlier definition $J = \{4,2\}$ since $0$ would be a part of $\lambda$.  The special piece $\cP(\0)$ has $4$ orbits with $\0_m = [3^2,1^2]$.  In this modified setting, $J = \{4\}$ and $\cP(\0)$ has $2$ orbits, with $\0_m = [3^2,2]$.  In both cases, $H(\0)=\Ab(\0)$.  
In this modified setting, we would only consider roughly half the type $C$ nilpotent orbits to be metaplectic orbits.  For example, $[3^2, 1^2]$  and $[1^8]$ would not be considered metaplectic at all.
\end{remark}

\section{Explicit construction of the smooth variety}\label{explicitsec}

\subsection{Generalities on construction of \texorpdfstring{$\hat{X}_K$}{X}}\label{Ohatsubsec}

Let $\0$ be a nilpotent orbit and let $e\in\0$.
As earlier, we let $\Gamma$ be a normal subgroup of $G^e$ of finite index and let $K=\Gamma/(G^e)^\circ\subseteq A(\0)$.
From now on we will assume, unless otherwise indicated, that $\{ e,h, f\}$ is an $\mathfrak{sl}_2$-triple with $h\in{\mathfrak h}$ dominant.
The action of $G$ by left multiplication on $G/\Gamma$ induces an action on ${\rm Spec}\, {\mathbb C}[G]^\Gamma$, hence on $\widetilde{X}_K$,
the affine variety whose coordinate ring is $\C[\widetilde\0_K]$.
Decomposing ${\mathbb C}[G]$ as a direct sum of irreducible representations leads to an embedding of $\widetilde{X}_K$ in a direct sum $\g\oplus V$, where $V=V_1\oplus\ldots \oplus V_r$ is a finite-dimensional sum of simple highest weight modules $V_i$.
Clearly, $\widetilde{X}_K$ must be the closure of the orbit of some $(e,v_1,\ldots ,v_r)$, where each $v_i$ is fixed by $\Gamma$.
Moreover, the projection $\pi:\widetilde{X}_K\rightarrow\overline\0$ is just the restriction to $\widetilde{X}_K$ of the $G$-equivariant projection $\g\oplus V\rightarrow \g$, and the fiber over $e$ is in bijection with the $\Gamma$-orbit of $(v_1,\ldots ,v_r)$, which must be free.
Given any $V$ and $(v_1,\ldots ,v_r)$ satisfying the same properties, we can set $\hat{X}_K=\overline{G\cdot (e,v_1,\ldots ,v_r)}$ to obtain an affine $\Gamma$-covering space of $\overline\0$. Then, as remarked in the introduction, $\widetilde{X}_K$ is the normalization of $\hat{X}_K$.

Constructing such a space $\hat{X}_K$ therefore comes down to finding, for each irreducible $A(\0)/K$-module $U$, a simple $G$-module $V$ containing a $G^e$-stable subspace which is fixed point-wise by $\Gamma$ and which spans a copy of $U$ as a $G^e$-module.
The problem of finding such $V$ is related to the work of the third author \cite{Sommers:Localsystems}, as follows.
Let $P$ be the (standard) parabolic subgroup of $G$ with Lie algebra $\oplus_{i\geq 0}{\mathfrak g}(i)$, let $L=G^h$, resp. ${\mathcal R}_u(P)$ be its Levi subgroup, resp. unipotent radical.
The main theorem of \cite{Sommers:Localsystems} embeds each irreducible representation of $A(\0)$ in some highest weight representation for $P$.
It follows from the irreducibility of $U$ that ${\mathcal R}_u(P)$ acts trivially, hence $U$ identifies with an irreducible representation of $L$, of the form $V_L(\lambda)$ for a weight $\lambda$ which restricts to a dominant weight on $L$ (i.e. such that $\lambda(\alpha_i)\geq 0$ whenever $\alpha_i(h) = 0$).

The classification of all possible $\lambda$ allows for a complete description of the representations of $G$ which contain a copy of $U$: given $\lambda$ such that $V_L(\lambda)|_{G^e}$ is isomorphic to $U$, let $\mu$ be the $W$-conjugate of $\lambda$ which is dominant (as a weight on $G$); then the irreducible highest weight module $V(\mu)$ contains a copy of $U$.

We will apply this construction to the case where $\0$ is a special nilpotent orbit and $K$ satisfies the hypothesis of Theorem \ref{conceptualthm}. 
For simplicity we assume that $K$ is the pre-image of a subgroup of $A(\0)$.
In classical types, we construct $U$ concretely as a direct sum of certain alternating powers of the natural representation of $G$.
In exceptional types, it turns out that one can almost always choose $V(\mu)$ to be the minimal faithful representation of the adjoint group of $G$ (i.e. the minimal faithful representation of $\g$ in types $F_4$, $G_2$ and the adjoint representation in types $E_6$, $E_7$ and $E_8$).
Equivalently, we can choose $\lambda$ to be a short root.


There is a natural ${\mathbb G}_m$-action on $\widetilde{X}_K$.
To see this, let $\{ e,h,f\}$ be an $\mathfrak{sl}_2$-triple.
Observe that for $t\in\C^\times$, ${\rm exp}(t\, {\rm ad}\, h)$ normalizes $\Gamma$ and thus acts `on the right' on the open orbit $G/\Gamma$ in $\widetilde{X}_K$, an action which clearly extends to $\widetilde{X}_K=\C[G/\Gamma]$.
Similarly, if $\hat{X}_K$ is the closure of the orbit of $(e,v_1,\ldots ,v_r)$, where $v_i$ is an $h$-eigenvector with eigenvalue $m_i$, then ${\mathbb G}_m$ acts by scaling on $V_i$ with weight $m_i$ and with weight 2 on $\g$.
This ${\mathbb G}_m$-action is modified when we are considering the intersection ${\mathcal S}$ of $\overline\0$ with a Slodowy slice $e'+\g^{f'}$ or the pre-image $\hat{\mathcal S}$: here one needs to compose the action on the {\it left} of ${\rm exp}(-t\, {\rm ad}\, h')$ with the action of $\exp(t\, {\rm ad}\, h)$ on the right.
In particular, this is the Kazhdan action on ${\mathcal S}$: decomposing $\g^{f'}=\sum_{i\leq 0} \g^{f'}(i)$ into its $h'$-graded components, we have:
$t\cdot (e'+x_0+x_{-1}+\ldots ) = e'+t^2 x_0+t^3 x_{-1}+\ldots$, where $x_{-i}\in\g^{f'}(-i)$.

We briefly summarize the construction of $\hat{X}_K\subset\g\oplus V$ in exceptional types, following \cite{Sommers:Localsystems}.
In all cases, $H(\0)=\mathfrak{S}_r$ with $r\leq 5$ and the action of $H(\0)$ splits in a sum of $(r-1)$ copies of the reflection representation.
This will allow us to define $\hat{X}_K$ as the closure of an orbit in $\g\oplus V$ where $V$ is a sum of $(r-1)$ isomorphic irreducible representations $V(\varpi_i)$.
For an $\mathfrak{S}_2$-cover, the reflection representation for $H(\0)$ is equal to the sign representation, hence in those cases $V=V(\varpi_i)$ is irreducible: in type $F_4$, $V$ is the minimal faithful representation; in all cases in types $E_6$, $E_7$, $E_8$ except $2A_2$ and $D_4(a_1)+A_2$ in type $E_8$, $V$ is equal to the adjoint representation.
The exceptional special orbits $F_4(a_3)$ and $E_8(a_7)$ were discussed above.
Finally, when $H(\0)=\mathfrak{S}_3$, $V$ is a sum of $2$ copies of the minimal representation in type $G_2$ or the adjoint representation in types $E_6$, $E_7$, $E_8$.
These statements can be double-checked by inspecting the tables in \cite{Lawther-Testerman}.

In the exceptional groups, there is a unique $K$ that satisfies the main theorem of the paper when $G$ is adjoint, so we often 
drop $K$ as a subscript in the earlier notations.

\subsection{New construction in the classical Lie algebras}\label{classicalsubsec}


We will now apply the construction outlined above to special pieces in classical Lie algebras, to provide another new proof of \cite[Thm. 6.2]{Kraft-Procesi:special}.
Our proof constructs the smooth variety in Lusztig's conjecture as a subset of the closure of a $G$-orbit in a representation of $G$.
This contrasts with the proof of Kraft and Procesi, where the smooth variety appears as a quotient.

We start by proving a lifting result for irreducible representations of $A(\0)$ in the classical types.
This is in the same spirit as \cite[Proposition 3.1]{Sommers:Localsystems} and does not require $\0$ to be special.

Let $V$ be the defining (or natural) representation of $G$ (see \ref{subsub:classical_component_groups}).  
As is well-known, 
the $i$-th exterior power $\Lambda^iV$ contains a unique copy of a fundamental weight representation for $G$, corresponding to the maximal highest 
weight in $\Lambda^iV$:
for $C_r$, $\Lambda^iV$ contains $V_{\varpi_i}$ for all $1 \leq i \leq r$;
for $B_r$, $\Lambda^iV$ contains $V_{\varpi_i}$ for all $1 \leq i \leq r-1$;
and for $D_r$, $\Lambda^iV$ contains $V_{\varpi_i}$
for all $1 \leq i \leq r-2$.
We also use the fact that, in type $C$, $\wedge^{\dim V}V$ is the trivial representation and $\wedge^{\dim(V)-1} V \simeq V$.

For $e \in \0_\lambda$, let $\fs$ be the subalgebra through
the $\mathfrak sl_2$-triple $\{e,h,f\}$.
For a part $c$ of $\lambda$, let
$V_c \subset V$ be the isotypic component for $\fs$ which is 
the direct sum of $m_\lambda(c)$ irreducible $\fs$-modules of dimension $c$. 
Hence, $e$ restricted to $V_c$ has partition $[c^{m_\lambda(c)}]$. 
We choose a normalized basis of $V$ as in \cite[\S 3.1]{Sommers:Localsystems}
so that $T$ preserves $V_c$
and $\ker e^r \cap V_c$ for any $r$.  

\begin{proposition}\label{prop:lifting_classical}
Let $G$ be of type $B$, $C$ or $D$. 
Let $U$ be an irreducible representation of $A(\0)$, which is necessarily one-dimensional.  Then there exists a highest weight irreducible representation $V$ of $G$ and $v \in V$, an extremal weight, such that $G^e$ preserves the line through $v$, $(G^{e})^\circ$ acts trivially on $v$, and the induced representation 
of $A(\0)$ on $v$ is isomorphic to $U$.    
\end{proposition}

\begin{proof}
Let $e = e_\lambda$. Define $\vartheta = \vartheta_0 \cup \vartheta_1$
(see \eqref{equation:vartheta}).
Let $Y \subset \vartheta$.
We will now define a fundamental weight representation $V_Y$ of $G$
and an extremal vector $v_Y \in V_Y$ such that 
$G^e$ leaves the line through $v_Y$ invariant.

For a part $c$ of $\lambda$, define 
$$r_Y(c)= |\{j \in Y \ | \ \nu_j \leq c \}|$$
and 
$$w(Y) = \sum_{j \in Y } h_\lambda(\nu_j).$$
Let $V_Y \simeq V_{\varpi_{w(Y)}}$ 
be the unique copy of this representation in $\Lambda^{w(Y)}(V)$.
Consider the subspace $\mathcal V$ of $V$ that is the direct sum of
\begin{equation} \label{equation:subspace}
\ker(e^{r_Y(c)}) \cap V_c 
\end{equation}
over all parts $c$ of $\lambda$.
for the action of $e$ on $V$.
The dimension of $\mathcal V$ is $w(Y)$.  
Hence,  $\Lambda^{w(Y)}\mathcal V$ is a line in $\Lambda^{w(Y)}V$. 
Let $v_Y \in\Lambda^{w(Y)}\mathcal V$ be nonzero.  

Now $\mathcal V$ is $T$-invariant by the assumption about
the normalized basis of $V$, so $v_Y$ is a weight vector 
and its weight is conjugate to $\varpi_{w(Y)}$ as in \cite{Sommers:Localsystems}.  
Hence $v_Y \in V_Y$ and is an extremal vector (i.e., it is $W$-conjugate to the highest weight).

The centralizer $G^e$ preserves $\ker(e^{r})$ for any power $r$.
Also $G^e$ preserves the subspace $\oplus_{d \geq c} V_{d}$. 
Since $r_Y(c) \geq r_Y(d)$ for $c \geq d$, 
it follows that $G^e$ preserves $\mathcal V$, hence acts on 
the line through $v_Y$.

Next, the reductive part $C$ of $G^e$ preserves $\ker(e^{r}) \cap  V_c$ for any $c$ and the connected component $C^\circ$ of $C$ has determinant $1$
when restricted to $\ker(e^{r}) \cap  V_c$ by \S\ref{subsub:classical_component_groups}.
It follows that $C^\circ$ and hence all of $(G^e)^\circ$ 
fixes $v_Y$ since the unipotent subgroup of $G^e$ automatically acts trivially.
This induces a one-dimensional representation of 
$A(\0)$ on the line through $v_Y$.
The induced character is non-trivial 
on $x_{\nu_i}$ if and only if $r_Y(\nu_i)$ is odd
since the orthogonal group ${\rm O}_{m(\nu_i)}(\C)$
acts diagonally inside $r_Y(\nu_i)$ copies of ${\rm O}_{m(\nu_i)}(\C)$ 
on $\ker(e^{r_Y(\nu_i)}) \cap  V_{\nu_i}$.



We thus obtain a bijection from characters of $A(\0)$ 
to subsets $Y \subset \vartheta$, using 
the following fact: 
the parity of $r_Y(\lambda_1)$ is equal to $|Y|$ and the parity of $r_Y$
changes at $\nu_i$ if and only if $i \in Y$.  The result follows since 
the number of characters of $A(\0)$ and the number of subsets of $\vartheta$ have cardinality $2^{|\vartheta|}$.
\end{proof}

In the proposition, note that while $G^e$ preserves the line through $v_Y$, the parabolic $P$ will not preserve the line in general.

Now we assume that $\0=\0_\lambda$ is a special orbit of $\g$.
Recall the definition of $J$ in \eqref{equation:def_J}.
For each $i\in J$, choose $Z_i \subset \vartheta \backslash J$ 
and define
\begin{equation}\label{equation:Y_i}
Y_i = \{i\} \cup Z_i.  
\end{equation}
Let $(V_{Y_i}, v_{Y_i})$ be as in the previous proposition and consider the weight vector $v$ in the following representation of $G$ defined as
\begin{equation}\label{equation:constructive_classical}
v:=(x, v_{Y_1}, \dots, v_{Y_{|J|}}) \in \g \oplus V_{Y_1} \oplus \dots \oplus V_{Y_{|J|}}.
\end{equation}
Let $\Gamma$ be the stabilizer of $v$.  From the first coordinate, $\Gamma \subset G^e$ and  $(G^e)^\circ \subset \Gamma$ by Proposition \ref{prop:lifting_classical}.
Hence, the orbit $G.v$ in the right side of \eqref{equation:constructive_classical}
is isomorphic to the cover $G/\Gamma$ of $\0$.  
Let $K = \Gamma/(G^e)^\circ$, a subgroup of $A(\0)$.
By abuse of notation, denote the (isomorphic) image of $H(\0)$ in $A(\0)$ also by $H(\0)$.

\begin{proposition}
The component group $A(\0)$ is the direct sum of $K$ and $H(\0)$.
Any subgroup $K$ of $A(\0)$ with this property arises for a unique choice of $Y_i$ for each $i \in J$.
When $Y_i = \{i\}$ for all $i \in J$, then $K=K^{\nat}$.
\end{proposition}

\begin{proof}
Recall that $H(\0)$ is the subgroup of $A(\0)$ of rank $|J|$ 
generated by $\{\sigma_i \ | \ i \in J\}$. 

Let $i\in J$ and consider the character $\chi_i$ of $A(\0)$ determined by $Y_i$ as in the proposition.  
Then for $j \in J$, we have 
\begin{equation} \label{equation:sigmas}
\chi_i(\sigma_j)=(-1)^{\delta_{ij}}    
\end{equation}
since $j \in Y_i$ if and only if $j=i$ and therefore $r_{Y_i}(\nu_j) \equiv r_{Y_i}(\nu_{j+1})$ modulo $2$ if and only if $j \neq i$.
It follows that every nontrivial element in $H(\0)$ acts non-trivially 
on some $v_{Y_i}$.  This shows that $K \cap H(\0) = \{1\}$.

Equation \eqref{equation:sigmas} implies that the characters $\chi_i$ are linearly independent.  Since $A(\0)$ is an elementary abelian group of rank $\ell-1$, it follows that $K$ is a subgroup of $A(\0)$ of rank $\ell-1-|J|$.  Thus $A(\0)$ is a direct sum of $K$ and $H(\0)$.

A subspace of dimension $s$ of a  vector space of dimension $t$ over a finite field with $q$ elements has $q^{s(t-s)}$ complements of dimension $s-t$.
So the number of $K$ sought is $2^{|J|(\ell-1-|J|)}$.  This is also the number of all possible choices of the $Y_i$'s, since there are $2^{\ell-1-|J|}$ choices for each $Y_i$.  It is also not hard to see that different choices of the $Y_i$'s lead to different $K$.  

The final statement follows from the fact that if $Y_i = \{ i \}$, then 
$\chi_i(\sigma_j) = 1$ for all $j \in \vartheta$ with $j \neq i$.
\end{proof}

 Let $e_m \in \0_m$.  Let $\{ h_\m,e_\m,f_\m\}$ be an 
 $\mathfrak{sl}_2$-triple with $e_\m\in\0_\m$ and let $\fs_m$ be its span.
For $j\in J$, let $C_j \simeq {\rm Sp}_{2b_j+2}(\C)$ be the factor of $G^{e_m}$ from \S\ref{subsub:classical_component_groups} corresponding to $s_j = \nu_j-1$.  Let $\cg_j=\Lie(C_j)$ and pick an element $e_j$ of the minimal nilpotent orbit of $\cg_j$.

Let $e =  e_\m+e_1+\cdots + e_{|J|}$. 
As shown in \cite{FJLS23}, ${\mathcal S}$ is the closure of the $C_1\times\cdots \times C_{|J|}$-orbit of $e$.  
Let $\{e_j, h_j, f_j\}$ be an $\mathfrak sl_2$-triple for $e_j$ in $\cg_h$
and $\fs_j$ its span.  We choose the $\mathfrak sl_2$-triples for $\fs_m$
and the $\fs_j$'s so that $e= e_m + \sum e_q$, 
$h =h_\m + \sum h_q$, and $f=f_m + \sum f_q$.  Recall 
that $\fs$ is the span of $\{e,h, f\}$.

Set $Y= Y_i$, where $Y_i$ is defined in \eqref{equation:Y_i} for $i \in J$.

\begin{proposition}\label{prop:affine_orbits_Cj}
Let $j \in J$.  
Then the closure of the orbit $C_j.v_{Y}$ in $V_{Y}$ is either a point or
an affine space $\C^{2b_j+2}$ isomorphic to the defining representation of $C_j \simeq {\rm Sp}_{2b_j+2}$.  
\end{proposition}

    
\begin{proof}
Let $V_{\m,c}$ of $V$ be the sum of $\fs_\m$-representations
of dimension $c$.  It is also a representation of $\fs$.
As an $\fs$-module, we have 
$$V_{\m,s_i} \simeq V_{s_i} \oplus V(s_i+1)\oplus V(s_i-1),$$
where $V(d)$ is the irreducible $\mathfrak sl_2$-module of dimension $d$.

If $j = i$, then  
$j \in Y$, so we have $r:=r_Y(s_j)=r_Y(s_j-1)$, but $r_Y(s_j+1)=r+1$.
Hence, $\mathcal V \cap  V_{\m,s_j}$
is a direct sum of $\ker e_m^{r} \cap V_{\m,s_j}$ and 
the unique line $\ell_m$ that lies in both the $1$-eigenspace of $h_j$ and 
the $(s_j-1-2r)$-eigenspace of $h_\m$.  Then $\ell_m$
is the kernel of $e_m^{r+1}$ and it is a highest weight space for the action of $C_j$ on the $(s_j-1-2r)$-eigenspace of $h_\m$, which is the defining representation of $C_j$.  
The dimension of $\mathcal V \cap  V_{\m,s_j}$ is thus $k:=r(2b_j+2)+1$
and the wedge product of a basis of $\mathcal V \cap  V_{\m,s_j}$ is a highest weight vector in the representation of $C_j$ on $\wedge^k(\ker e_m^{r+1})$,
where the latter is thus the defining presentation of $C_j$.  
Hence, the closure of $C_j.v_Y$ the affine space $\C^{2b_j+2}$.

If $j \neq  i$, then  
$j \not \in Y$, so we have $r:=r_Y(s_j)=r_Y(s_j-1) = r_Y(s_j+1)$
and so $\mathcal V \cap  V_{\m,s_j} = \ker e^r \cap  V_{\m,s_j}$.
If $r=0$, then $\mathcal V \cap  V_{\m,s_j} = \{0\}$ and $C_j.v_Y$ is a
point.
However if $r>0$, then 
$\ker e^r \cap  V_{\m,s_j}$
is a direct sum of $\mathcal A:= \ker e_m^{r-1} \cap V_{\m,s_j}$, 
the subspace $\mathcal B$ of $V_{\m,s_j} \cap \ker e_j$
where $h_m$ acts by $s_j+1-2r$, 
and a line 
$\ell_m =u_1+u_2$ where $u_1$ lies
in joint $(s_j-1-2r, 1)$-eigenspace for $(h_\m,h_j)$
and $u_2$ lies in the $(s_j+1-2r, -1)$-eigenspace for $(h_\m,h_j)$.

The dimension of $\mathcal V \cap  V_{\m,s_j}$ is now $k:=r(2b_j+2)$.
The subspace $\mathcal D_2$ spanned by $u_2$, $\mathcal A$, and $\mathcal B$
is $\ker e_m^{r} \cap V_{\m,s_j}$, so $\wedge^k (\mathcal D_2)$ is
a $C_j$-invariant 
line in $\wedge^k(\ker e_m^{r+1})$.   

Let $U(d)$ denote the $s_j+1-2d$-eigenspace of $V_{\m,s_j}$.
Then $u_1$ is highest weight vector for $C_j$ in $U(r+1)$  
and $\wedge^{2b_j+1}(\mathcal B)$ is a highest weight
vector for $C_j$ in $\wedge^{2b_j+1}(U(r))$, both defining
the natural representation for $C_j$.
Let $\mathcal D_1$ be the span of $u_1$,  $\mathcal A$, and $\mathcal B$.
It follows that $\wedge^k(D_1)$ in $\wedge^k(\ker e_m^{r+1})$
is highest weight vector that defines the natural represention of $C_j$.
\end{proof}

Let $K \subset A(\0)$ be the subgroup determined by $v$ in Equation \eqref{equation:constructive_classical}.
Dropping the subscript $K$ from the introduction, 
recall that $\hat{X}$ is the closure of the $G$-orbit of $v$.
Let $\hat\pi: \hat{X} \to \overline{\0}$ be projection to the first factor.  Let $\Phat  := \pi^{-1}({\mathcal P}(\0))$.
We can give another proof of the main result for $K \subset A(\0)$ in the classical types.  

\begin{theorem}\label{theorem:explict_model_classical}
	The $G$-variety $\Phat$ is smooth and is a solution to the Main Theorem.
\end{theorem}

\begin{proof}
The map $\pi$ restricts to a map from the closure of the 
$C_1 \times \dots \times C_{|J|}$-orbit of $v$ in $V_{Y_i}$ 
to the slice $\mathcal S$.  
By Proposition \ref{prop:affine_orbits_Cj} and the following  \cref{lemma:min_orbit_becomes_affine_space},
$\pi^{-1}(\mathcal S)$ is an affine space.
\end{proof}

\begin{lemma}\label{splem} \label{lemma:min_orbit_becomes_affine_space}
Let $G=\Sp_{2n}=\Sp(V)$, let $e_0\in\Lie(G)=\g$ belong to the minimal nilpotent orbit and let $\{ h_0,e_0,f_0\}$ be an $\mathfrak{sl}_2$-triple in $\g$.
Let $v$ be a non-zero element of the image (i.e. column space) of $e_0$, equivalently, a non-zero element of $V(1;h_0)$.
Then the closure of $G\cdot (e_0,v)$ in $\g\oplus \C^{2n}$ is isomorphic to $\C^{2n}$.
\end{lemma}

\begin{proof}
After conjugating and scaling $\C^{2n}$ if necessary, we may assume $$e_0=\begin{pmatrix} 0 & 0 & \cdots & 1 \\ 0 & 0 & \cdots & 0 \\ \vdots & \vdots & \ddots & \vdots \\ 0 & 0 & \cdots & 0 \end{pmatrix}\quad\mbox{and}\quad v=\begin{pmatrix} 1 \\ 0 \\ \vdots \\ 0 \end{pmatrix}.$$
Then $e_0$ identifies with $v\otimes v$ in $S^2(\C^{2n})\cong\g$, so the orbit closure of $(e_0,v)$ is isomorphic to the orbit closure of $v$, which is just $\C^{2n}$.
\end{proof}

\begin{example}\label{Example:B_4}
    As an example in $B_4$, consider $\0$ with partition $[5,3,1]$.  
Then $A(\0) \simeq \mathfrak{S}_2 \times \mathfrak{S}_2$.
The kernel of the map $\eta_A$ 
from $A(\0)$ to $\Ab(\0)$ 
is generated by $\sigma_2 =x_3x_1$ 
and so $\bar{A}(\0) \simeq \mathfrak{S}_2$.
The only non-special orbit in $\cP(\0)$
is $\0_m$ with partition $[4^2,1]$
and the image of $H(\0)$ in $A(\0)$ 
is generated by $\sigma_1=x_5x_3$ and $\eta|_A$ restricted to $H(\0)$ 
is seen to be injective.  
The natural choice of $K$ is $\ker(\eta_A)$, 
but also $K$ generated by $x_5x_1$ would work.  
If we wanted to work up $\pi_1(\0)$,
which is the dihedral group with $8$ elements, then 
there again two choices for $K$, the pre-images of the two subgroups that worked
in $A(\0)$.

The image of $H(\0)$ in $\pi_1(\0)$ is only well-defined up to conjugacy.  In this case, $H(\0)$ could be either 
of two conjugate elements of order $2$.  The representations in \cref{theorem:explict_model_classical} are the fundamental weight representations $V_{\varpi_1}$ or $V_{\varpi_3}$. 
\end{example}

\subsection{Smooth $\mathfrak{S}_2$-covers of special pieces in exceptional Lie algebras}\label{S2subsec}

We now carry out the same procedure as in the previous section for the special pieces $\cP(\0)$ in the exceptional groups where $H(\0) = \frakS_2$.
There are 21 such special pieces.
Let $(V=V(\varpi_i),v)$ be a fundamental weight representation and a vector inducing the sign representation for $H(\0)$, as specified in Table \ref{S2table}.
We define $\hat{X}$ to be the closure of the $G$-orbit of $(e,v)$ in $\g\oplus V$, with projection morphism $\hat\pi:\hat{X}\rightarrow\overline\0$.

Since $H(\0)=\mathfrak{S}_2$, there are two orbits in ${\mathcal P}(\0)$.
Let $\0_\m$ be the non-special orbit in ${\mathcal P}(\0)$, let $\{ h_\m,e_\m,f_\m\}$ be an $\mathfrak{sl}_2$-triple with $e_\m\in\0_\m$ and let ${\mathcal S}=(e_\m+\g^{f_\m})\cap\overline\0$ be the Slodowy slice singularity from $\0_\m$ to $\overline\0$.
It was observed in \cite{FJLS23} (based on the results of \cite{FJLS}) that ${\mathcal S}$ is isomorphic to $\C^{2k}/\mathfrak{S}_2$ where $k=(\dim\0-\dim\0_\m)/2$, i.e. to a $c_k$ singularity.
To establish smoothness of $\hat{\mathcal P}(\0)$, we will show,
as in the previous section, that $\hat{\mathcal S}$ is an affine space,
using the explicit isomorphism of ${\mathcal S}$ with $\C^{2k}/\mathfrak{S}_2$ and Lemma \ref{splem}.

Let $\cg=\g^{f_\m}\cap\g^{h_\m}$ be the reductive part of the centralizer.
In all cases, $\cg$ has a simple component $\cg_0$ isomorphic to $\mathfrak{sp}_{2k}$, and ${\mathcal S}$ is the set of all elements of the form $e_\m+x_0$, where $x_0$ belongs to the minimal nilpotent orbit closure in $\cg_0$.
Equivalently, if $\{ h_0,e_0,f_0\}\subset\cg_0$ is an $\mathfrak{sl}_2$-triple with $e_0$ a minimal nilpotent element, then $h_\m+h_0$ is conjugate to $h$.
Note that it is straightforward to check this by inspecting weighted Dynkin diagrams.
After conjugating if necessary, we assume that $\{e_\m+e_0, h_\m+h_0, f_\m+f_0\}=\{ e,h,f\}$ is an $\mathfrak{sl}_2$-triple in $\g$ with $e\in\0$ (and $h$ dominant).
Denote by $C_0$ the simple component of $G^{f_\m}\cap G^{h_\m}$ with Lie algebra $\cg_0$.
(In all but two cases, see below, we can further assume that $h_\m$ is dominant.)
The centralizer $\g^{h_\m}$ can be read off from the weighted Dynkin diagram for $\0_\m$.
We are interested in the sum $\g^{h_\m}_0$ of the simple factors of $\g^{h_\m}$ on which $\cg_0$ is supported.
In type $F_4$, only $(\0,\0_\m)=(\tilde{A}_1,A_1)$ occurs, and then $\cg_0=\mathfrak{sp}_6=\g^{h_\m}_0$.
For $\0=\0_{D_4(a_1)+A_1}$ in type $E_7$, $\cg_0=\mathfrak{sl}_2$ embedded as a root subalgebra of $\g^{h_\m}_0=\mathfrak{sl}_4$.
By a case-by-case check (e.g. using \cite{Lawther-Testerman}), we see that otherwise in types $E_6$, $E_7$ and $E_8$, $\cg_0=\mathfrak{sp}_{2k}$ is diagonally embedded in $\g^{h_\m}_0=\oplus \mathfrak{sl}_{2k}$.
Note that the number of summands of $\g^{h_\m}_0$ indicates the nilpotent $G$-orbit containing $e_0$.

Let $V=V(\varpi_i)$ be the irreducible representation indicated in 
the $\lambda^+$ column of 
Table \ref{S2table}.  It is always the adjoint representation,
except for two cases in $E_8$.
By a case check, we show that the one-dimensional sign subrepresentation $\C v$ of $V$ is unique; $v$ is often a highest weight vector.
In any case, we have $[h,v]=l v$ for some $l\geq 2$.
The proof now proceeds with the following straightforward verification: in each case $V(l-1;h_\m)$ contains a (unique) copy $U$ of the natural representation for $\cg_0$.
Then $v$ is simply a fixed vector for the action of $h_0$ on $V(l-1;h_\m)$, since such a vector belongs to $V(l;h)$ and 
$H(\0)$ acts on $v$ by the sign representation of $H(\0)$ (using Table \ref{table:attaching_conj_classes}).
We can define $\hat{X}$ to be equal to the closure of the $G$-orbit of $(e_\m+e_0,v)$.
Since $\hat{\mathcal S}$ is transverse to $\hat{X}$ at any point mapping to $e_\m$, all irreducible components of $\hat{\mathcal S}$ are of the same dimension $2k$.
On the other hand, each fiber of $\hat\pi$ contains at most two elements.
Therefore $\hat{\mathcal S}$ is equal to the $C_0$-orbit closure of $(e_\m+e_0,v)$ in $\g\oplus V$, hence to the $C_0$-orbit closure of $(e_0,v)$ in $\mathfrak{sp}_{2k}\oplus U$, which is isomorphic to $\C^{2k}$ by Lemma \ref{splem}.


When $V$ is the adjoint representation, the case check in the previous paragraph can be carried out by inspecting the data in \cite{Lawther-Testerman}.
Alternatively, one can argue directly, using only the degree $l$ of the subspace $\g(l;h)$ containing the sign representation for $A(\0)$.
(Note that $l$ is always even.)
The subspace $\g^e(l;h)$ is always of dimension $1$ or $2$.
A uniform statement can be made if we exclude the case $\0=D_6(a_1)$ in type $E_8$.
Then $\g^e(l;h)$ is the direct sum of $\g^{e_\m}(l;h_\m)$ and the image of ${\rm ad}\, e_0$ on $\g^{e_\m}(l-1;h_\m)$; furthermore, $l-1$ is the highest odd eigenvalue for ${\rm ad}\, h_\m$, and $\g^{e_\m}(l-1;h_\m)=\g(l-1;h_\m)$ is a copy of the natural representation for one of the simple factors of $\g_0^{h_\m}$.
The case $\0=D_6(a_1)$ is discussed in more detail in the following Example (b).


\begin{example}\label{nondomexample}
a) We go through the details for one case covered by the `standard' argument above.
Let $\0$ be the orbit in $E_8$ with Bala-Carter label $E_8(a_5)$; then $\0_\m=D_7$.
The weighted Dynkin diagram of $\0$ is $\scriptsize{\begin{aligned}20 &20020 \\[-1.8mm] &0\end{aligned}}$ and the weighted Dynkin diagram of $\0_\m$ is $\scriptsize{\begin{aligned}21 &01101 \\[-1.8mm] &1\end{aligned}}$.
Hence $[\g^{h_\m},\g^{h_\m}]\cong\mathfrak{sl}_2\oplus\mathfrak{sl}_2$ (with simple roots $\alpha_4,\alpha_7$).
In this case $\cg_0=\mathfrak{sl}_2$ is supported on both copies of $\mathfrak{sl}_2$ in $\g^{h_\m}$, hence we may assume that $h_0=\alpha_4^\vee+\alpha_7^\vee$.
In particular, we then have $h=h_\m+h_0$, with both $h$ and $h_\m$ dominant.
The maximum eigenvalue for ${\rm ad}\, h$ is $22$, and $\dim\g(22;h)=2$.
Similarly, $\g(22;h_\m)$ is 1-dimensional (spanned by the highest root element), and the next highest eigenvalue is $21$.
It is easy to see that $\g(21;h_\m)$ is a copy of the natural representation for the second $\mathfrak{sl}_2$ factor in $\g^{h_\m}$.

b) Let $\0$ be the orbit in $E_8$ with Bala-Carter label $D_6(a_1)$.
The weighted Dynkin diagram of $\0$ is $\scriptsize{\begin{aligned} 01 & 00012 \\[-1.8mm] & 1\end{aligned}}$.
Then $\0_\m=\0_{D_5+A_1}$, with weighted Dynkin diagram $\scriptsize{\begin{aligned} 10 & 01012 \\[-1.8mm] & 0 \end{aligned}}$.
Hence $[\g^{h_\m},\g^{h_\m}]=\mathfrak{sl}_4\oplus \mathfrak{sl}_2$; the subalgebra $\cg_0$ is a root subalgebra of $\g^{h_\m}_0=\mathfrak{sl}_4$.
The highest degree components of $\g^e$ are as follows: $\dim\g^e(14;h)=1$, $\dim\g^e(11;h)=4$, $\dim \g^e(10;h)=2$.
The (unique) copy of the sign representation lies in $\g^e(10;h)$.
(This is the only case where $l$ is not maximal and $\dim \g^e(l;h)=2$.)
By a similar calculation, the highest degree components of $\g^{e_\m}$ are as follows: $\dim\g^{e_\m}(14;h_\m)=1$, $\dim\g^{e_\m}(11;h_\m)=2$, $\dim \g^{e_\m}(10;h_\m)=5$ and $\dim \g^{e_\m}(9;h_\m)=2$.
It is easy to check that $\g^{e_\m}(11;h_\m)=\g(11;h_\m)$ is a copy of the natural representation for the $\mathfrak{sl}_2$ factor of $\g^{h_\m}$, but is acted on trivially by $\g^{h_\m}_0$.
On the other hand, $\g^{e_\m}(10;h_\m)$ is a codimension 1 subspace of $\g(10;h_\m)$, which is a copy of the 6-dimensional irreducible representation for $\g^{h_\m}_0$.
Since $\cg_0$ is a root subalgebra of $\g^{h_\m}_0$, it follows that $\g^{e_\m}(10;h_\m)$ decomposes over $\cg_0$ as a sum of two copies of the natural representation and a 1-dimensional trivial module; the latter is the trivial $A(e)$-submodule of $\g^e(10;h)$.
Finally, $\g^{e_\m}(9;h_\m)$ is a codimension 2 subspace of $\g(9;h_\m)$, which is a copy of the natural representation for $\g^{h_\m}_0$.
Since $\cg_0$ acts trivially on $\g(11;h_\m)$, it follows that $\g^{e_\m}(9;h_\m)$ is a copy of the natural representation for $\cg_0$.

c) Finally, we consider the two orbits for which $V(\varpi_i)$ is not the adjoint representation.
In either case, there is only one non-zero node $i$ on the weighted Dynkin diagram of $\0$, and $\hat{X}$ is the closure of the $G$-orbit of $(e,v)$, where $v$ is a highest weight vector in $V(\varpi_i)$.
It is immediate from the weighted Dynkin diagrams that $h_0=\alpha_i^\vee$, and that the maximum eigenvalue of $h_\m$ is $l-1$ where $l=\varpi_i(h)$ is the maximum eigenvalue of $h$.
Specifically, $i=1$ when $\0=2A_2$ and $i=2$ when $\0=D_4(a_1)+A_2$ (both in type $E_8$, see Table \ref{S2table}).

We remark that $\dim V(\varpi_1)=$ 3,875 and $\dim V(\varpi_2)=$ 147,250, so the underlying spaces for these covers are larger than for all other special pieces in exceptional types (including the exceptional special pieces considered in \S \ref{exceptionalsubsec}).
\end{example}

In Table \ref{S2table}, we summarize the key information required for the $\mathfrak{S}_2$-cases.
In the fourth column, we indicate the inclusion of $\cg_0$ into $\g^{h_\m}_0$ (where we omit $\g^{h_\m}_0$ if $\cg_0=\g^{h_\m}_0$); in the sixth column we specify $w(h_\m)$ such that $w(h_\m+h_0)=h$.
In the seventh column, $\lambda$ is the weight of minimal length that is both $P$-stable and realizes the sign representation of $A(e)$ on restriction to $G^e$. 
This weight $\lambda$ turns out always to be conjugate a fundamental weight $\lambda^+$ (usually the adjoint representation).
However, the 1-dimensional weight space for $\lambda$ in $V_{\lambda^+}$ may not be 
$G_e$-stable, but in all cases there is a $W$-conjugate (extremal) vector in $V_{\lambda^+}$  that is $G_e$ stable and yields the sign representation.  
We use this line to construct the cover. 
The last column records the degrees of the occurrences of $V_{\lambda^+}$ in functions on the cover of $\0$, with the unique copy not occurring in the $\C[\0]$ in bold-face.  
The extremal $G_e$-stable vector can be chosen to be a highest weight vector of $V_{\lambda^+}$
if and only if this unique copy occurs in degree $\frac{1}{2}\lambda^+(h)$, which is the highest possible degree.  
Finally, the orbits with a $*$ are known to have non-normal closure.


	\begin{table}[htp] 
		\caption{$\mathfrak{S}_2$-special pieces in exceptional types}\label{S2table}
		\begin{center}
			\begin{tabular}{|c | c| c| c| c| c| c|c | c|} 
				\hline
				$G$ & $\0$ & $\0_\m$ & $\mathfrak{c}_0\subset {\mathfrak g}^{h_\m}_0$ & $h$ & $w(h_\m)$ 
                & $\lambda$ & $\lambda^+$ 
                & $\lambda^+$ in $\C^i[\widetilde{\0}]$\\ 
				\hline 
				
				$F_4$ & $\tilde{A}_1$ & $A_1$ & $\mathfrak{sp}_6$ & $\scriptsize{0001}$ & $\scriptsize{1000}$ 
                &  $\varpi_4$ & $\varpi_4$ 
                & 1\\
				
				\hline & & & & & & & & \\[-0.3cm]
				
				$E_6$ & $A_2$ & $3A_1$ & $\mathfrak{sl}_2$ & $\scriptsize{\begin{aligned}00 &000 \\[-1.8mm] &2\end{aligned}}$ & $\scriptsize{\begin{aligned}00 &100 \\[-1.8mm] &0\end{aligned}}$  
                & $\varpi_2$ & $\varpi_2$ 
                & $1,\bf{2}$  \\
                
				& $E_6(a_3)$ & $A_5$ & $\mathfrak{sl}_2$   
				& $\scriptsize{\begin{aligned}20 &202 \\[-1.8mm] &0\end{aligned}}$ & $\scriptsize{\begin{aligned}21 &012 \\[-1.8mm] &1\end{aligned}}$  
				& $\scriptsize{\begin{aligned}01 &210 \\[-1.8mm] &1\end{aligned}}$ & $\varpi_2$ 
                & $1,4, 5, \bf{5}$ \\
				
				\hline & & & & & & & &  \\[-0.3cm]
				
				$E_7$ & $A_2$ & $(3A_1)'$ & $\mathfrak{sl}_2$ & $\scriptsize{\begin{aligned}20 &0000 \\[-1.8mm] &0\end{aligned}}$ & $\scriptsize{\begin{aligned}01 &0000 \\[-1.8mm] &0\end{aligned}}$ 
                & $\varpi_1$ & $\varpi_1$  
                & $1,\bf{2}$  \\

                & $A_2+A_1$ & $4A_1$ & $\mathfrak{sp}_6\subset\mathfrak{sl}_6$ & $\scriptsize{\begin{aligned}10 &0010 \\[-1.8mm] &0\end{aligned}}$ & $\scriptsize{\begin{aligned}00 &0001 \\[-1.8mm] &1\end{aligned}}$ 
                &  $\varpi_1$ 
                &  $\varpi_1$ 
                & $1,{\bf 2}$ \\  
                
				& *$D_4(a_1)+A_1$ & $A_3+2A_1$ & $\mathfrak{sl}_2\subset \mathfrak{sl}_4$ & $\scriptsize{\begin{aligned}01 &0001 \\[-1.8mm] &1\end{aligned}}$ & $\scriptsize{\begin{aligned}10 &0101 \\[-1.8mm] &0\end{aligned}}$ 
				& $\scriptsize{\begin{aligned}1 2 &3210 \\[-1.8mm] &2\end{aligned}}$ &  $\varpi_1$ 
                & $1,3,\bf{3}$  \\  
								 
				& $D_5(a_1)$ & $D_4+A_1$ & $\mathfrak{sp}_4\subset \mathfrak{sl}_4$ & $\scriptsize{\begin{aligned}20 &1010 \\[-1.8mm] &0\end{aligned}}$ & $\scriptsize{\begin{aligned}21 &0001 \\[-1.8mm] &1\end{aligned}}$ 
                & $\scriptsize{\begin{aligned} 0 1 &2221 \\[-1.8mm] &1\end{aligned}}$ &  $\varpi_1$ 
                & $1,{\bf 4},5$ \\  

                & $E_6(a_3)$ & $(A_5)'$ & $\mathfrak{sl}_2$ & $\scriptsize{\begin{aligned}02 &0020 \\[-1.8mm] &0\end{aligned}}$ & $\scriptsize{\begin{aligned}10 &1020 \\[-1.8mm] &0\end{aligned}}$  
                & $\scriptsize{\begin{aligned}1 2 &2100 \\[-1.8mm] &1\end{aligned}}$ &  $\varpi_1$ 
                & $1,5,\bf{5}$   \\

                & $E_7(a_3)$ & $D_6$ & $\mathfrak{sl}_2$ & $\scriptsize{\begin{aligned}20 &2022 \\[-1.8mm] &0\end{aligned}}$ & $\scriptsize{\begin{aligned}21 &0122 \\[-1.8mm] &1\end{aligned}}$ 
                & $\scriptsize{\begin{aligned}0 1 &2100 \\[-1.8mm] &1\end{aligned}}$ &  $\varpi_1$ 
                & $1,5,7,{\bf 8},9$ \\ 
				
				\hline & & & & & & & &  \\[-0.3cm]
				
				$E_8$ 
				& $A_2$ & $3A_1$ & $\mathfrak{sl}_2$ & $\scriptsize{\begin{aligned} 00 &00002 \\[-1.8mm] &0\end{aligned}}$ & $\scriptsize{\begin{aligned}00 &00010 \\[-1.8mm] &0\end{aligned}}$ 
                 & $\varpi_8$ & $\varpi_8$ 
                 & $1,{\bf 2}$\\

                & $A_2+A_1$ & $4A_1$ & $\mathfrak{sp}_8\subset \mathfrak{sl}_8$ & $\scriptsize{\begin{aligned}10 &00001 \\[-1.8mm] &0\end{aligned}}$ & $\scriptsize{\begin{aligned}00 &00000 \\[-1.8mm] &1\end{aligned}}$ 
                & $\varpi_8$ & $\varpi_8$ 
                & $\tiny{1, {\bf 2}}$ \\
                    
				& $2A_2$ & $A_2+3A_1$ & $\mathfrak{sl}_2$ & $\scriptsize{\begin{aligned}20 &00000 \\[-1.8mm] &0\end{aligned}}$ & $\scriptsize{\begin{aligned}01 &00000 \\[-1.8mm] &0\end{aligned}}$ 
                & $\varpi_1$	& $\varpi_1$ 
                & $2,{\bf 4}$ \\
				
				& $D_4(a_1)+A_2$ & $A_3{+}A_2{+}A_1$ & $\mathfrak{sl}_2$ & $\scriptsize{\begin{aligned}00 &00000 \\[-1.8mm] &2\end{aligned}}$ & $\scriptsize{\begin{aligned}00 &10000 \\[-1.8mm] &0\end{aligned}}$ 
                & $\varpi_2$ & $\varpi_2$ 
                & $4,{\bf 8}$ \\

	           & $A_4+2A_1$ & $2A_3$ & $\mathfrak{sp}_4\subset \mathfrak{sl}_4^{\oplus 2}$ & $\scriptsize{\begin{aligned}00 &10001 \\[-1.8mm] &0\end{aligned}}$ & $s_4s_5\!\left(\scriptsize{\begin{aligned}10 &01000 \\[-1.8mm] &0\end{aligned}}\right)$ 
               & $\varpi_8$ 
               & $\varpi_8$  
               & $\tiny{1,{\bf 4}}$ \\

            & *$D_5(a_1)$ & $D_4+A_1$ & $\mathfrak{sp}_6\subset \mathfrak{sl}_6$ & $\scriptsize{\begin{aligned}10 &00102 \\[-1.8mm] &0\end{aligned}}$ & $\scriptsize{\begin{aligned}00 &00012 \\[-1.8mm] &1\end{aligned}}$ 
            & $\scriptsize{\begin{aligned}2 3 &43210 \\[-1.8mm] &2\end{aligned}}$ 
            &  $\varpi_8$ 
            & $\tiny{1,{\bf 4},5}$  \\

				& *$E_6(a_3)$ & $A_5$ & $\mathfrak{sl}_2$ & $\scriptsize{\begin{aligned}20 &00020 \\[-1.8mm] &0\end{aligned}}$ & $\scriptsize{\begin{aligned}20 &00101 \\[-1.8mm] &0\end{aligned}}$ 
				& $\scriptsize{\begin{aligned}01 &22221 \\[-1.8mm] &1\end{aligned}}$ &  $\varpi_8$ 
                & $1,5,\bf{5}$  \\
				
				& *$D_6(a_1)$ & $D_5+A_1$ & $\mathfrak{sl}_2\subset\mathfrak{sl}_4$ & $\scriptsize{\begin{aligned}01 &00012 \\[-1.8mm] &1\end{aligned}}$ & $\scriptsize{\begin{aligned}10 &01012 \\[-1.8mm] &0\end{aligned}}$ 
                & $\scriptsize{\begin{aligned}1 2 &32100 \\[-1.8mm] &2\end{aligned}}$ 
                &  $\varpi_8$ 
                & $1,5, {\bf 5},7$ \\
				
            & *$E_7(a_3)$ & $D_6$ & $\mathfrak{sp}_4\subset\mathfrak{sl}_4$ & $\scriptsize{\begin{aligned}20 &10102 \\[-1.8mm] &0\end{aligned}}$ & $\scriptsize{\begin{aligned}21 &00012 \\[-1.8mm] &1\end{aligned}}$ 
            & $\scriptsize{\begin{aligned}0 1 &22210 \\[-1.8mm] &1\end{aligned}}$ &  $\varpi_8$ 
            & $\tiny{1,\!7, {\bf 8},9}$ \\

	           & $E_8(b_6)$ & $A_7$ & $\mathfrak{sl}_2\subset \mathfrak{sl}_2^{\oplus 4}$ & $\scriptsize{\begin{aligned}00 &20002 \\[-1.8mm] &0\end{aligned}}$ & $s_4\left(\scriptsize{\begin{aligned}10 &10110 \\[-1.8mm] &0\end{aligned}}\right)$ 
               & $\varpi_8$ & $\varpi_8$ 
               & $1,7, {\bf 8}$ \\

                & $E_8(a_5)$ & $D_7$ & $\mathfrak{sl}_2\subset \mathfrak{sl}_2^{\oplus 2}$ & $\scriptsize{\begin{aligned}20 &20020 \\[-1.8mm] &0\end{aligned}}$ & $\scriptsize{\begin{aligned}21 &01101 \\[-1.8mm] &1\end{aligned}}$ 
                & $\scriptsize{\begin{aligned}0 1 &22221 \\[-1.8mm] &1\end{aligned}}$ &  $\varpi_8$ 
                & $1,7, 11, {\bf 11}$ \\    
                
                & $E_8(a_3)$ & $E_7$ & $\mathfrak{sl}_2$ & $\scriptsize{\begin{aligned}20 &20222 \\[-1.8mm] &0\end{aligned}}$ & $\scriptsize{\begin{aligned}21 &01222 \\[-1.8mm] &1\end{aligned}}$ 
				& $\scriptsize{\begin{aligned}0 1 &21000 \\[-1.8mm] &1\end{aligned}}$ &  $\varpi_8$ 
                & $\tiny{1,\!7,\! 11,\! 13,\! {\bf 14},\!17}$ \\
				\hline
			\end{tabular}
		\end{center}
	\end{table}

\subsection{Special pieces which are $\mathfrak{S}_3$-quotients}\label{S3_subsec}

There are nine special orbits (all in exceptional Lie algebras) for which $H(\0)=\mathfrak{S}_r$ with $r\geq 3$; for seven of these, $r=3$.
We have $H(\0)=\mathfrak{S}_3$ if and only if there are three orbits in ${\mathcal P}(\0)$.
We call these $\mathfrak{S}_3$-special pieces.
In this section we will construct $\hat{X}$ such that $\hat{\mathcal P}(\0)$ is smooth for each $\mathfrak{S}_3$-special piece.

The prototype is the case of $\0=G_2(a_1)$ in type $G_2$, where we may assume $e=e_{\alpha_2}+e_{3\alpha_1+\alpha_2}$.
Let $\rho:G_2\rightarrow\SO(V)$ be the minimal faithful representation and let $v_1,\ldots ,v_7$ be the standard basis for $V$ (with $v_1$ a highest weight vector).
In this case we set $\hat{X}$ to be the closure of the $G_2$-orbit of $(e,v_1,v_2)$ in $\g_2\oplus V\oplus V$.
The following result will play a role similar to that played by Lemma \ref{splem} in the previous two subsections.

\begin{lemma}\label{G2lem}
a) With the above notation, $\hat{X}$ is isomorphic to the closure of the minimal nilpotent orbit in $\mathfrak{so}_8$.
In particular, it has a unique singular point, at the origin.

b) Let $\{ e,h,f\}$ be any $\mathfrak{sl}_2$-triple with $e\in\0$ and let $v,w\in V$ span the image of $d\rho(e)^2$.
Then $v,w$ also span the $2$-eigenspace for $d\rho(h)$, and the closure of the $G_2$-orbit of $(e,v,w)$ is isomorphic to $\hat{X}$.
\end{lemma}

\begin{proof}
Part (a) was essentially proved by Brylinski and Kostant \cite{Brylinski-Kostant:JAMS}, further discussed in \cite[\S 2.3.5]{FJLS23}.
Part (b) is immediate from the facts that $d\rho(e)$ has partition $[3,3,1]$ and that the closure of the orbit of $(e,v,w)$ is isomorphic to the closure of the orbit of $(e,av+bw,cv+dw)$ for all $\begin{pmatrix} a & b \\ c & d \end{pmatrix}\in\GL_2$.
\end{proof}

We conclude from Lemma \ref{G2lem} that $\hat{X}=\widetilde{X}$ and $\hat{\mathcal P}(\0)$ is isomorphic to the minimal nilpotent orbit in $\mathfrak{so}_8$, so in particular is smooth.

We construct $\hat{X}$ for the other six $\mathfrak{S}_3$-special pieces.
Let $\{ e,h,f\}$ be an $\mathfrak{sl}_2$-triple with $e\in\0$.
In each case, there exists a subspace $U$ of $\g^e$ which is fixed by $(G^e)^\circ$ and induces a copy of the reflection representation for $A(\0)=\mathfrak{S}_3$; moreover, this subspace is unique.
(In four cases, $U$ is the maximum weight space for ${\rm ad}\, h$; see below.)
We define $\hat{X}$ to be the closure of the $G$-orbit of $(e,v,w)$ in $\g\oplus \g^{\oplus 2}$, where $v,w$ span $U$.

Denote by $\hat\alpha$ the highest root in $\g$.
Recall that outside type $A$, there is a unique highest root $\beta$ in $\Phi\setminus\{ \hat\alpha\}$.
We call $e_{\hat\alpha}$ and $e_\beta$ the two highest root elements in $\g$.
Let $\0_\m$ be the minimal orbit in ${\mathcal P}(\0)$.
We specify in each case an orbit $\0'$ lying just below $\0_\m$ in the closure order; this orbit will allow us to apply Lemma \ref{G2lem}.


\begin{itemize}
\item[i)] If $\0=D_4(a_1)$ in type $E_6$, $E_7$, $E_8$ or $\0=D_4(a_1)+A_1$ in type $E_8$ then the maximum eigenvalue for ${\rm ad}\, h$ is $6$.
We have $U=\g(6;h)={\rm im}({\rm ad}\, e)^6$, hence we can choose $v,w$ any linearly independent elements.
If $h$ is dominant, then we can choose $v,w$ to be the two highest root elements in $\g$.
(For $\0=D_4(a_1)$ in type $E_8$, both $\overline\0$ and $\hat{X}$ are branched non-normal, as discussed below.)
In each of these cases we set $\0'=2A_2$.
\item[ii)] Let $\0=E_7(a_5)$ in type $E_7$.
The maximum eigenvalue of ${\rm ad}\, h$ is $10$, but the subspace $\g(10;h)$ is of dimension 3 and decomposes over $A(\0)$ as the sum of the reflection representation and a one-dimensional trivial representation.
We let $v,w$ be vectors spanning the reflection subrepresentation.
Set $\0'=A''_5$.
\item[iii)] For $\0=E_8(b_5)$ in type $E_8$, the maximum eigenvalue for ${\rm ad}\, h$ is $22$, and $\g(22)$ is a one-dimensional copy of the trivial representation for $G^e$.
The next largest eigenvalue for ${\rm ad}\, h$ on $\g^e$ is $18$, and $U=\g^e(18)$ is a copy of the reflection representation for $A(\0)$.
We set $\hat{X}$ to be the closure of the orbit of $(e,v,w)$, where $v,w$ span $U$.
Let $\0'$ be the orbit with label $E_6$.
\end{itemize}

Let ${\mathfrak h}_2$ denote the reflection representation for $\mathfrak{S}_3$.
We proved in \cite{FJLS23} that the Slodowy slice singularity from $\0_\m$ to $\overline\0$ is isomorphic to $({\mathfrak h}_2\oplus {\mathfrak h}_2^*)/\mathfrak{S}_3$ or $({\mathfrak h}_2\oplus {\mathfrak h}_2^*)^{\oplus 2}/\mathfrak{S}_3$.
The latter singularity occurs once, for $\0=D_4(a_1)+A_1$ in type $E_8$; we leave this case aside until further notice.
As in the proof of \cite[Main Theorem, \S 6]{FJLS23}, the remaining cases can all be understood by considering the degeneration to the orbit $\0'$.

Let $\{ h',e',f'\}$ be an $\mathfrak{sl}_2$-triple with $e'\in\0'$ and let ${\mathcal S}=(e'+\g^{f'})\cap\overline\0$ be the Slodowy slice singularity from $\0'$ to $\overline\0$.
Let $C=G^{h'}\cap G^{e'}$ be the reductive part of the centralizer and let $\cg$ be its Lie algebra.
In each case, $C$ contains a simple factor of type $G_2$, which we denote $C_0$, with Lie algebra $\cg_0$.
Let $\{ h_0,e_0,f_0\}\subset\cg_0$ be an $\mathfrak{sl}_2$-triple with $e_0$ subregular in $\cg_0$.
Then $e'+e_0\in\0$, or equivalently $h'+h_0$ is conjugate to $h$.
It follows from dimensions that $\overline{C_0\cdot (e'+e_0)}$ is an irreducible component of ${\mathcal S}$.
With the exception of $\0=D_4(a_1)$ in type $E_8$, $\overline\0$ is unibranch at $\0'$ and hence ${\mathcal S}$ is isomorphic to the closure of the subregular nilpotent orbit in $\cg_0$.
When $\0=D_4(a_1)$ in type $E_8$, then $\cg\cong\g_2\oplus\g_2$ and ${\mathcal S}$ is the union of the subregular nilpotent orbit closures in the two factors.
Note that in this case, $A(e')=\mathfrak{S}_2$ swaps the two factors $\g_2$, hence also swaps the branches of ${\mathcal S}$.
Because of this complication, let us also leave this case aside until further notice.

For the other four cases, $\cg_0$ is supported on a unique simple factor $\g^{h'}_0$ of $\g^{h'}$, and that simple factor is isomorphic to $\mathfrak{so}_8$ or $\mathfrak{so}_{10}$.
We may assume $h_0=2\beta^\vee$, where $\beta$ is the highest root in $\g^{h'}_0$.
Then $h$ and $h'$ are dominant with $h=h'+h_0$; let $e=e'+e_0$.
As specified above, there is a subspace $U$ of $\g^e(l;h)$ (for some $l$) which induces the reflection representation for $A(e)$.
It now follows from inspecting the weighted Dynkin diagram that $\g(l-2;h')$ is isomorphic to the natural representation for $\g^{h'}_0$ (after choosing a suitable ordering of the simple roots in type $D_4$).
It follows that $\g^{e'}(l-2;h')$ is isomorphic as a $\cg_0$-module to the sum of the minimal faithful representation and a trivial module.
Applying the same argument as in \S \ref{S2subsec}, we see that $\hat{\mathcal S}$ is equal to the closure of the $C_0$-orbit of $(e'+e_0,v,w)$, hence to the closure of the $C_0$-orbit of $(e_0,v,w)$.
Lemma \ref{G2lem} now implies that $\hat{\mathcal S}$ is isomorphic to the closure of the minimal nilpotent orbit in $\mathfrak{so}_8$.
In particular, it is smooth away from the origin, hence $\hat{\mathcal P}(\0)$ is smooth.
This settles all of the $\mathfrak{S}_3$-cases except $D_4(a_1)$ and $D_4(a_1)+A_1$ in type $E_8$.
The first of these can be dealt with using slight modifications of the above arguments, as we explain in the following example.

\begin{example}\label{extraS3examples}
a) Let $(\0,\0')=(E_7(a_5),A''_5)$ in type $E_7$.
We will go through this example to illustrate the above computation.
The weighted Dynkin diagram of $\0$ is $\scriptsize{\begin{aligned}00 &2002 \\[-1.8mm] &0\end{aligned}}$ and the weighted Dynkin diagram of $\0'$ is $\scriptsize{\begin{aligned}20 &0022 \\[-1.8mm] &0\end{aligned}}$.
As above, we set $h=h'+h_0$, $e=e'+e_0$.
We can assume that both $h$ and $h'$ are dominant (with $h_0=2\beta^\vee$ as above).
The highest weight space $\g(10;h)$ decomposes over $A(e)$ as the sum of the reflection representation and a one-dimensional trivial representation.
These arise from the highest weight spaces for ${\rm ad}\, h'$ as follows: $\dim\g(10;h')=1$, and this spans the trivial submodule of $\g(10;h)$; meanwhile, $\g(8;h')$ is a copy of the natural representation for $\g^{h'}_0$, hence $\g^{e'}(8;h')$ is the minimal faithful representation for $\cg_0$.
Now $U$ is the subspace of $\g^{e'}(8;h')$ of weight $2$ for $h_0$, which is two-dimensional and which spans a copy of the reflection representation for $A(e)$.
(This can be seen from the fact that $A(e)$ identifies with the component group for $e_0$ as an element of $\cg_0$.)
The argument above now establishes that $\hat{\mathcal P}(\0)$ is smooth.

b) Let $\0=D_4(a_1)$ in type $E_8$.
As noted above, in this case ${\mathcal S}$ is a union of two branches isomorphic to the subregular nilpotent orbit closure in $\g_2$.
These branches meet at $e'$ (and only at $e'$).
Both simple factors of $\cg$ are contained in $\g^{h'}_0=\mathfrak{so}_{14}$.
We can therefore assume as above that $h_0=2\beta^\vee$.
Then exactly the same argument applies, so that one irreducible component of $\hat{\mathcal S}$ is isomorphic to the closure of the minimal nilpotent orbit in $\mathfrak{so}_8$.
The action of $A(e')$ then shows that $\hat{\mathcal S}$ is a union of two such orbit closures, meeting only at $e'$.
In particular, $\hat{\mathcal S}$ is smooth away from the origin, hence $\hat{\mathcal P}(\0)$ is smooth.
Note however that $\hat{X}$ is branched non-normal.

\end{example}


Finally, we now address the case $\0=D_4(a_1)+A_1$ in type $E_8$.
As above, let ${\mathcal S}$ be the Slodowy slice singularity from $\0'=2A_2$ to $\overline\0$.
In \cite[Prop. 2.15]{FJLS23}, the first three authors and Fu proved that ${\mathcal S}$ is isomorphic to $(d_4\times d_4)/\mathfrak{S}_3$.
The proof included some computations in GAP, which can be adapted to establish smoothness of $\hat{\mathcal P}(\0)$.

We introduce some notation.
Let $\{ h',e',f'\}$ be an $\mathfrak{sl}_2$-triple with $e'\in\0'$ and let $\cg=\g^{h'}\cap\g^{f'}$ be the reductive part of the centralizer.
Then $\cg=\g_2\oplus\g_2$, contained in $[\g^{h'},\g^{h'}]=\mathfrak{so}_{14}$.
The $\cg$-module structure for the negative part of $\g^{f'}$ is as follows.
Let $V$, resp. $V'$ denote the minimal representation for the first, resp. second copy of $\g_2$; denote by $\{ v_1,\ldots ,v_7\}$ and $\{ v'_1,\ldots ,v'_7\}$ the standard bases.
Then $\g^{f'}(-2;h')$ is isomorphic over $\cg$ to $V\otimes V'\oplus\C f'$ and $\g^{f'}(-4;h')$ is isomorphic to $V\oplus V'$.

We define $\hat{X}$ to be the closure of the orbit of $(e,v,w)$ where $e\in\0$ and $v,w$ span the image of $({\rm ad}\, e)^6$.
The final step in the proof \cite[Prop. 2.15]{FJLS23} was the computation in GAP of a concrete representative $x=e'+e_0+e'_0+v$ of ${\mathcal S}\cap\0$, where $v$ corresponds to $v_1\otimes v'_2+v_2\otimes v'_1$ in $V\otimes V\subset \g^{f'}(-2;h')$.
For this representative, we can easily compute a basis for ${\rm im}({\rm ad}\, x)^6$, which turns out to be spanned by two elements of $\g^{f'}$.
Specifically, with the above identifications with the modules $V\otimes V'\oplus \C$ and $V\oplus V'$, the image of $({\rm ad}\, x)^6$ is spanned by $v_1+v'_1-v_2\otimes v'_2$ and $v_2+v'_2+v_1\otimes v'_1$.
It follows by dimensions and unibranchness that $\hat{\mathcal S}$ is isomorphic to the closure of the $C$-orbit of $$(e_0+e'_0, v_1\otimes v'_2+v_2\otimes v'_1, v_1,v'_1,v_2\otimes v'_2,v_2,v'_2,v_1\otimes v'_1)$$
in $\cg\oplus V\otimes V'\oplus (V\oplus V'\oplus V\otimes V')^{\oplus 2}$.
By projecting onto $\cg\oplus V\oplus V\oplus V'\oplus V'$, we see that $\hat{\mathcal S}$ is isomorphic to $d_4\times d_4$.
In particular, it is smooth on the complement of $(d_4\times \{ 0\})\cup (\{ 0\}\times d_4)$, which includes all points mapping to ${\mathcal P}(\0)$.

Similarly to \S \ref{S2subsec}, we include the details of all the $\mathfrak{S}_3$-special pieces in Table \ref{S3table}.
The labels are very similar to those in Table \ref{S2table}, with the following slight modifications: in the seventh column we give the weighted Dynkin diagram of the orbit $\0'$ discussed above (note that in contrast with \S \ref{S2subsec}, we always have $h'+h_0=h$ with $h'$ dominant); for the case of $E_7(a_5)$ in type $E_7$, it turns out not to be possible to choose $\varpi$ to be a dominant weight (on $\g$), hence in this case only we allow $\varpi$ non-dominant.

\begin{table}[htp] 
	\caption{$\mathfrak{S}_3$-special pieces in exceptional types}\label{S3table}
\begin{center}
\begin{tabular}{|c | c| c| c| c| c| c| c| c| c|}
\hline
$G$ & $\0$ & $\0_\m$ & $\0'$ & $\mathfrak{c}_0\subset {\mathfrak g}(0)$ & $h$ & $h'$ 
& $\lambda$ & $\lambda^+$ 
& $\lambda^+$ in $\C^i[\widetilde{\0}]$ \\
\hline 
$G_2$ & $G_2(a_1)$ & $A_1$ & $0$ & $\g$ & $02$ & $00$ 
& $\varpi_1$ & $\varpi_1$ 
& $1$ \\

\hline & & & & & & &  & & \\[-0.3cm]

$E_6$ & $D_4(a_1)$ & $2A_2+A_1$ & $2A_2$ & $\g_2\subset\mathfrak{so}_8$ & $\scriptsize{\begin{aligned}00 &200 \\[-1.8mm] &0\end{aligned}}$ & $\scriptsize{\begin{aligned}20 &002 \\[-1.8mm] &0\end{aligned}}$ 
& $\varpi_2$ & $\varpi_2$ 
& $1$, ${\bf 3}$, ${\bf 3}$ \\

\hline & & & & & & &  & & \\[-0.3cm]

$E_7$ & $D_4(a_1)$ & $2A_2+A_1$ & $2A_2$ & $\g_2\subset\mathfrak{so}_{10}$ & $\scriptsize{\begin{aligned}02 &0000 \\[-1.8mm] &0\end{aligned}}$ & $\scriptsize{\begin{aligned}00 &0020 \\[-1.8mm] &0\end{aligned}}$ 
& $\varpi_1$ & $\varpi_1$ 
& $1$, ${\bf 3}$, ${\bf 3}$ \\

& $E_7(a_5)$ & $A_5+A_1$ & $(A_5)''$ & $\g\subset\mathfrak{so}_8$ & $\scriptsize{\begin{aligned}00 &2002 \\[-1.8mm] &0\end{aligned}}$ & $\scriptsize{\begin{aligned}20 &0022 \\[-1.8mm] &0\end{aligned}}$ 
& $\scriptsize{\begin{aligned} 12 & 3210 \\[-1.8mm] & 2\end{aligned}}$ & $\varpi_1$ 
& $1$, $5$, ${\bf 5}$, ${\bf 5}$ \\

\hline & & & & & & &  & &  \\[-0.3cm]

$E_8$ & $*D_4(a_1)$ & $2A_2+A_1$ & $2A_2$ & $\g_2\subset\mathfrak{so}_{14}$  & $\scriptsize{\begin{aligned}00 &00020 \\[-1.8mm] &0\end{aligned}}$ & $\scriptsize{\begin{aligned}20 &00000 \\[-1.8mm] &0\end{aligned}}$ 
& $\varpi_8$ & $\varpi_8$ 
& $1$, ${\bf 3}$, ${\bf 3}$ \\

& $E_8(b_5)$ & $E_6+A_1$ & $E_6$ & $\g_2\subset \mathfrak{so}_8$ & $\scriptsize{\begin{aligned}00 &20022 \\[-1.8mm] &0\end{aligned}}$ & $\scriptsize{\begin{aligned}20 &00222 \\[-1.8mm] &0\end{aligned}}$ 
& $\scriptsize{\begin{aligned} 12 & 3210 \\[-1.8mm] & 2\end{aligned}}$ & $\varpi_8$ 
& $1$, $7$, ${\bf 9}$, ${\bf 9}$, $11$ \\
& $D_4(a_1)+A_1$ & $2A_2+2A_1$ & $2A_2$ & $\g_2\oplus\g_2\subset\mathfrak{so}_{14}$ & $\scriptsize{\begin{aligned}00 &00010 \\[-1.8mm] &0\end{aligned}}$ & $\scriptsize{\begin{aligned}20 &00000 \\[-1.8mm] &0\end{aligned}}$ 
& $\varpi_8$ & $\varpi_8$ 
& $1$, ${\bf 3}$, ${\bf 3}$ \\

 \hline
\end{tabular}
\end{center}
\end{table}

\subsection{Exceptional special pieces in $F_4$ and $E_8$}\label{exceptionalsubsec}

The last two special pieces to consider are those with $\0=F_4(a_3)$ in type $F_4$ and $\0=E_8(a_7)$ in type $E_8$.
In the former case $H(\0)=\mathfrak{S}_4$; in the latter $H(\0)=\mathfrak{S}_5$.
These two cases were the most recalcitrant in our proof of the Main Theorem in \cite{FJLS23}, ultimately requiring brute force computation.
We will only give sparse details.

\subsubsection{The $\mathfrak{S}_4$-special piece}

Let $\g$ be simple of type $F_4$ and let $\rho:\g\rightarrow\mathfrak{gl}(V)$ denote the minimal faithful representation for $\g$.
Let $\0$ be the orbit labelled $F_4(a_3)$, with component group $A(\0)=H(\0)=\mathfrak{S}_4$.
The weighted Dynkin diagram of $\0$ is $0200$.

Choose an $\mathfrak{sl}_2$-triple $\{ e,h,f\}$ with $e\in\0$, such that $h=2\varpi_2^\vee$ is dominant.
Note that $({\rm ad}\, e)^7=0$ and $({\rm ad}\, e)^6$ has rank 2; similarly, $\rho(e)^5=0$ and $\rho(e)^4$ has rank 3.
The top graded part $\g(6;h)$ is spanned by $e_{2342}$ and $e_{1342}$; similarly, $V_\mini(4;h)$ is spanned by $v_{1232}$, $v_{1231}$, $v_{1221}$ (where the subscripts indicate roots in $\g$, recalling that the non-zero weights in $V_\mini$ are the short roots in $\g$).
The non-trivial irreducible $A(\0)$-representations are (i) the sign representation $\varepsilon$; (ii) the reflection representation ${\mathfrak h}_2$ for the quotient $\mathfrak{S}_3$ of $A(\0)$; (iii) the reflection representation ${\mathfrak h}_3$ for $\mathfrak{S}_4$; (iv) ${\mathfrak h}_3\otimes\varepsilon$.
Let $\rho:\g\rightarrow\mathfrak{gl}(V)$ be an irreducible representation and let $m$ be the maximum eigenvalue of $\rho(h)$.
Since $\0$ is distinguished, $(G^e)^\circ$ acts trivially on $V(m;h)={\rm im}(\rho(e)^m)$.
We apply this to the highest weight modules for the four fundamental weights \cite[\S 2.2]{Sommers:Localsystems}:

\begin{itemize}
\item[i)] $V(\varpi_1)$ is the adjoint representation ${\mathfrak g}$; the top $h$-weight space ${\mathfrak g}(6;h)\cong {\mathfrak h}_2$;

\item[ii)] $V(\varpi_4)$ is the minimal faithful representation $V_{\mini}$ for ${\mathfrak g}$; the top $h$-weight space $V_{\mini}(4;h)\cong{\mathfrak h}_3$;

\item[iii)] $\Lambda^2(V_{\mini})\cong V(\varpi_3)\oplus \g$ and $\Lambda^2(V_{\mini})(8;h)\subset V(\varpi_3)$, hence the top $h$-weight space in $V(\varpi_3)$ is a copy of $\Lambda^2({\mathfrak h}_3)\cong {\mathfrak h}_3\otimes\varepsilon$;

\item[iv)] $\Lambda^2(\g)\cong V(\varpi_2)\oplus V(\varpi_3)$ and $\Lambda^2(\g)(12;h)\subset V(\varpi_2)$, so the top $h$-weight space in $V(\varpi_2)$ is a copy of $\varepsilon$.
\end{itemize}

By (iii)-(iv), one can define $\hat{X}$ without using the terms in $V(\varpi_2)$ or $V(\varpi_3)$.
Of course, the action of $\mathfrak{S}_4$ on ${\mathfrak h}_3$ is faithful, so $\g\oplus V^{\oplus 3}$ contains a copy of the universal cover of $\0$.
We define $\hat{X}$ to be the closure of the $G$-orbit of 
$$(e,v_{1232}, v_{1231}, v_{1221})\in {\mathfrak g}\oplus  V_\mini^{\oplus 3}.$$
As per the $\mathfrak{S}_3$-cases, $\hat{X}$ is essentially unchanged on replacing $e$ by any other representative of $\0$ and $v_{1232},v_{1231},v_{1221}$ by any triple spanning the image of $\rho(e)^4$.

Now let $\{ h_\m,e_\m,f_\m\}$ be an $\mathfrak{sl}_2$-triple with $e_\m\in\0_\m=A_2+\tilde{A}_1$ and let ${\mathcal S}=(e_\m+\g^{f_\m})\cap\overline\0$ be the Slodowy slice singularity from $\0_\m$ to $\overline\0$.
In \cite[\S 5.2]{FJLS23} we gave a computational proof that ${\mathcal S}$ is isomorphic to $({\mathfrak h}_3\oplus{\mathfrak h}_3^*)/\mathfrak{S}_4$.
Ultimately, this came down to a parametrization of the slice, i.e. an $\mathfrak{S}_4$-equivariant morphism $S:{\mathfrak h}_3\oplus {\mathfrak h}_3^*\rightarrow {\mathcal S}$.
Let $U$ be the open subset of ${\mathfrak h}_3\oplus {\mathfrak h}_3^*$ on which $\mathfrak{S}_4$ acts freely.
For each $\xi\in U$, it is a straightforward matter to compute a basis for the image of $\rho(S(\xi))^{4}$.
With suitable care, one can (for each such $\xi$) obtain from this basis a nice $\mathfrak{S}_4$-orbit of bases depending on $\xi$.
The key properties (verified computationally) are that the elements of the basis can be expressed as polynomials in $\xi$, with all linear terms supported, and that the action of ${\mathfrak S}_4$ on $\xi$ (which preserves $S(\xi)$) is equivalent to the action on the set of bases of simple coroots in ${\mathfrak h}_3$.
This allows one to define a morphism from $U$ to $\hat\pi^{-1}({\mathcal S}'\cap\0)$.
It follows that ${\mathcal S}'\cap\0$ is connected.
In particular, its normalization equals ${\mathfrak h}_3\oplus {\mathfrak h}_3^*$.
But now its coordinate ring is graded, and one sees by inspection that there are six independent generators of degree 1, hence ${\mathcal S}'\cap\0$ is isomorphic to ${\mathfrak h}_3\oplus {\mathfrak h}_3^*$.
Thus $\hat{\mathcal P}(\0)$ is smooth.

\begin{remark}\label{F4remark}
Let $\0'$ be the nilpotent orbit in $\g$ with Bala-Carter label $A_2$ and let ${\mathcal S}'$ be the Slodowy slice singularity from $\0'$ to $\overline\0$.
As remarked earlier, it was proved in \cite{FJLS23} that ${\mathcal S}'$ is isomorphic to the quotient $d_4/\mathfrak{S}_4$.
A natural question is whether the pre-image of ${\mathcal S}'$ in $\hat{X}$ is equal to a $d_4$ singularity.
In fact, one can see from the ${\mathbb G}_m$-action that it is not.
For this, recall our remarks in \S \ref{Ohatsubsec} that the induced ${\mathbb G}_m$-action on $\hat\pi^{-1}({\mathcal S}')$ is via $(2-{\rm ad}\, h')$ on $\g$ and $(4-\rho(h'))$ on the copies of $V$.
In particular, the terms of degree 2 come from the reductive part $\cg$ of the centralizer of $e'$ and $V(2;h')$.
We have $\dim V(2;h')=6$, hence the dimension of the degree 2 part of $\C[\hat{\mathcal S}']$ is 26, which is two less than required for the minimal nilpotent orbit in $\mathfrak{so}_8$.
We conclude that $\hat{X}$ is not normal (and its non-normal locus contains at least the pre-image of $\overline{\0'}$).
Note that if we define $\hat{X}$ to be the $G$-orbit closure of $$(e,(e_{2342},e_{1342}),(v_{1232},v_{1231},v_{1221}))$$ in $\g\oplus \g^{\oplus 2}\oplus V^{\oplus 3}$, then the ${\mathbb G}_m$-action leads to 28 generators of degree 2 as desired.
Hence it is possible that with this latter definition, $\hat{X}$ is isomorphic to $\widetilde{X}$.
\end{remark}

\subsubsection{The $\mathfrak{S}_5$-special piece}

Let $\g$ be simple of type $E_8$, let $\0$ be the exceptional special orbit $E_8(a_7)$ and let $\0_\m=A_4+A_3$.
The weighted Dynkin diagram of $\0$ is $\scriptsize{\begin{aligned} 00 &02000 \\[-1.8mm] &0\end{aligned}}$.
We have $H(\0)=A(\0)=\mathfrak{S}_5$.
Let $\{ e,h,f\}$ be an $\mathfrak{sl}_2$-triple with $e\in\0$; assume $h=2\varpi_5^\vee$ is dominant.
Note that $({\rm ad}\, e)^{11}=0$ and ${\rm rk}(({\rm ad}\, e)^{10})=4$.
The space ${\mathfrak g}(10;h)$ is spanned by the four highest root elements $u=e_{\hat\alpha}$, $v=e_{\hat\alpha-\alpha_8}$, $w=e_{\hat\alpha-\alpha_8-\alpha_7}$ and $z=e_{\hat\alpha-\alpha_8-\alpha_7-\alpha_6}$.

There are six non-trivial irreducible representations for $A(\0)$.
These can be realised as the top $h$-graded parts of the representations with fundamental highest weights $\varpi_1$, $\varpi_2$, $\varpi_8$, $\varpi_7$, $\varpi_6$, $\varpi_5$.
The top $h$-graded part $\g(10;h)\subset V(\varpi_8)$ is a copy of the reflection representation ${\mathfrak h}_4$ for $A(\0)$.
By similar arguments to the previous subsection, the top graded parts in $V(\varpi_7)$, $V(\varpi_6)$ and $V(\varpi_5)$ can be expressed as alternating powers of $\g(10;h)$, hence are not needed to define $\hat{X}$.
The top $h$-weight space in $V(\varpi_1)$ (resp. $V(\varpi_2)$) is a $5$-dimensional irreducible $W$ (resp. $W\otimes\varepsilon$) for $A(\0)$ \cite{Sommers:Localsystems}.
The $A(\0)$-action in each case is faithful.
Let $\hat{X}$ be the closure of the $G$-orbit of $(e,u,v,w,z)$ in $\g\oplus \g^{\oplus 4}$.


One proceeds in the same manner as the previous case, although the computations are much more laborious.
Letting ${\mathcal S}$ be the Slodowy slice singularity from $\0_\m$ to $\overline\0$, we proved computationally in \cite{FJLS23} that ${\mathcal S}$ is isomorphic to $({\mathfrak h}_4\oplus {\mathfrak h}_4^*)/\mathfrak{S}_5$.
This isomorphism is established via an explicit $\mathfrak{S}_5$-invariant morphism ${\mathfrak h}_4\oplus {\mathfrak h}_4^*\rightarrow {\mathcal S}$.
As above, let $U$ be the open subset of points of ${\mathfrak h}_4\oplus {\mathfrak h}_4^*$ acted on freely by $\mathfrak{S}_5$.
It is a fairly straightforward task to produce an $\mathfrak{S}_5$-orbit of bases for the image of $({\rm ad}\, e)^{10}$, depending linearly on ${\mathfrak h}_4\oplus {\mathfrak h}_4^*$, and inducing a morphism from ${\mathfrak h}_4\oplus {\mathfrak h}_4^*$ to $\hat{\mathcal S}$ satisfying the same properties as in the previous case.
By inspection of degrees and the fact that all linear generators of $\C[{\mathfrak h}_4\oplus {\mathfrak h}_4^*]$ are supported, this map is an isomorphism.

\begin{remark}\label{E8remark}
Let $\0'$ be the orbit with Bala-Carter label $A_4+A_2+A_1$.
As remarked in \cite[Rk. 5.5(c)]{FJLS23}, it has been conjectured by Hanany that the Slodowy slice ${\mathcal S}'$ from $\0'$ to $\overline\0$ is a quotient ${\mathcal C}_5/\mathfrak{S}_5$, where ${\mathcal C}_5$ is the Coulomb branch of the quiver gauge theory with one central $U(2)$ node and five outer $U(1)$ nodes.
It seems natural to expect also that the pre-image of ${\mathcal S}'$ in $\widetilde{X}$ is ${\mathcal C}_5$.
One can therefore ask whether $\hat\pi^{-1}({\mathcal S}')$ is isomorphic to ${\mathcal C}_5$; this is the analogue of the question asked in Rk. \ref{F4remark}.
The Hilbert series for the graded coordinate ring of ${\mathcal C}_5$ can be calculated using the monopole formula \cite{CHZ:Monopole}: one has $15$ generators of degree 2 and $32$ of degree 3 (perhaps with others of higher degree).
Using a similar calculation to Rk. \ref{F4remark}, one can show that the coordinate ring of $\hat\pi^{-1}({\mathcal S}')$ has $15$ generators of degree 2 but only $22$ of degree 3.
This indicates (assuming the above conjecture) that $\hat{X}$ is not normal at points lying above $\0'$.
One could try to ``fix'' the failure of normality by defining $\hat{X}$ to be a $G$-orbit closure in $\g\oplus \g^{\oplus 4}\oplus V(\varpi_1)^{\oplus 5}$ (choosing five vectors which span the top $h$-graded part of $V(\varpi_1)$).
It is easy to check that this construction gives the correct number of generators of degree 3.
On the other hand, it seems likely that even this orbit closure is non-normal, and one has to include a sum of five copies of $V(\varpi_2)$ too.
\end{remark}

\section{Further related examples}

Theorem \ref{conceptualthm} shows that the affinized cover of a nilpotent orbit closure $\0$ can be ``less singular'' than (the normalization of) $\overline\0$, in the sense that smooth points of $\widetilde{X}$ can map via $\pi$ to singular points of $\overline\0$.
Another important class of examples is given by the affinized universal cover of the regular nilpotent orbit in $\mathfrak{sl}_n$.
This was discussed in \cite{BBFJLS}, where it was proved that for $n\geq 4$ the generic singularity has codimension 4, and is an affine open subset of the blowup at the singular locus of the quotient of $\C^4$ by a dihedral group.
(For the case of $\mathfrak{sl}_3$, one obtains the minimal nilpotent orbit closure in type $G_2$ \cite{Brylinski-Kostant:JAMS}.)
It is an interesting problem to classify the generic singularities of arbitrary affinized covers of nilpotent orbit closures.
For classical $\g$ and for covers factoring through the adjoint group, the codimension 2 singularities of $\widetilde{X}$ were classified by Matvieievskyi in \cite{matvieievskyi}.
The problem is however open for universal covers in classical types, see e.g. (c) of the following example.
For exceptional $\g$, this can be tackled using tables of Green functions as in the proof of Theorem \ref{conceptualthm}.
(In types $E_6$ and $E_7$ we use Green functions for the {\it generalized} Springer correspondence.)

\begin{example}\label{coverexamples}
a) Let $(\0,\0')=(E_6(a_3),A_4+A_1)$ in type $E_6$ and let ${\mathcal S}$ be the Slodowy slice singularity from $\0'$ to $\overline\0$.
In \cite{FJLS23} it was proved that ${\mathcal S}$ is isomorphic to $S^2({\mathbb C}^2/\mu_3)={\mathbb C}^4/G(3,1,2)$, where $\mu_k$ denotes a cyclic group of order $k$ and $G(3,1,2)=(\mu_3\times\mu_3)\rtimes \mathfrak{S}_2$.
We have $\pi_1(\0)=\mu_3\times \mathfrak{S}_2$.
It can be checked using Green's functions that $\pi^{-1}(\0')$ is connected.
It follows that $\pi^{-1}({\mathcal S})$ is isomorphic to ${\mathbb C}^4/\Gamma$, where $\Gamma$ is a cyclic normal subgroup of $G(3,1,2)$ such that $G(3,1,2)/\Gamma\cong \mu_3\times\mathfrak{S}_2$.
Then $\Gamma$ must be the subgroup $\{ (x,x^{-1}) : x\in\mu_3\}$ of $\mu_3\times\mu_3$.
This confirms the final sentence in \cite[Rk. 5.9]{FJLS23}.

b) Similarly, let $(\0,\0')=(E_7(a_5),A_4+A_2)$ in type $E_7$ and let ${\mathcal S}$ be the Slodowy slice from $\0'$ to $\overline\0$.
In \cite{FJLS23} it was proved that ${\mathcal S}$ is isomorphic to $S^3({\mathbb C}^2/\mu_2)={\mathbb C}^6/G(2,1,3)$.
Here $G(2,1,3)=(\mu_2\times\mu_2\times\mu_2)\rtimes \mathfrak{S}_3$.
Note that $\pi_1(\0)=\mu_2\times\mathfrak{S}_3$.
By the same argument as the previous example, we must have $\pi^{-1}({\mathfrak S})\cong {\mathbb C}^6/\Gamma$, where $\Gamma$ is the subgroup $\langle (1,-1,-1), (-1,1,-1)\rangle$ of $\mu_2\times\mu_2\times\mu_2$, since this is the only normal subgroup of $G(2,1,3)$ such that the quotient is $\mu_2\times\mathfrak{S}_3$.
This confirms the final sentence of \cite[Rk. 5.7(c)]{FJLS23}.

c) 
For an example in classical types, let $\0$ be the orbit in $\mathfrak{so}_9$ with partition $[5,3,1]$.
This orbit is special, with $\pi_1(\0)=G(2,1,2)$, a dihedral group of order 8, $A(\0) \simeq \mathfrak{S}_2^2$ and $H(\0) \simeq \mathfrak{S}_2$. The natural map $\pi_1(\0) \twoheadrightarrow A(\0)$ is the quotient by the center of $\pi_1(\0)$. The special piece of $\0$ is the union of $\0$ and $\0_{[4^2,1]}$. The Slodowy slice singularity from $\0_{[4^2,1]}$ to $\overline\0$ is isomorphic to $\C^2/\mathfrak{S}_2$. The image of $H(\0)$ in $A(\0)$ is a subgroup of order $2$ and hence its image in $\pi_1(\0)$ is generated by a reflection of order $2$, and therefore is not normal.

Let $\pi:\widetilde{X}\rightarrow\overline\0$ be the affinized universal cover.
Consider first the Slodowy slice singularity ${\mathcal S}$ from $\0_{[3,3,3]}$ to $\overline\0$.
By Kraft-Procesi's column removal rule \cite{Kraft-Procesi:classical}, this is equivalent to the Slodowy slice singularity in $\mathfrak{sp}_6$ from $[2,2,2]$ to $[4,2]$, which by \cite{Fu:wreath} is ${\mathbb C}^4/G(2,1,2)$.
(Note however that the fundamental group of $[4,2]$ is $\mathfrak{S}_2^2$.)
By inspecting Green functions it is established that $\pi^{-1}(\0_{[3,3,3]})$ is connected and in bijection with $\0_{[3,3,3]}$.
It follows that $\pi^{-1}({\mathcal S})\cong {\mathbb C}^4$, hence $\pi^{-1}(\0_{[3,3,3]})$ is contained in the smooth locus of $\widetilde{X}$.

There is a unique orbit lying just below $[3,3,3]$ in the partial order, with partition $[3,3,1,1,1]$.
The interval from $[5,3,1]$ to $[3,3,1,1,1]$ is given by the following diagram:
$$
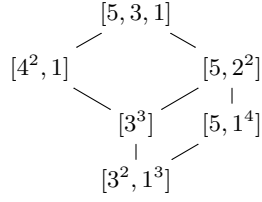
\small{\xymatrix@=.2cm{
{} & [5,3,1]\ar@{-}[dl]\ar@{-}[dr] & {} \\
[4^2,1]\ar@{-}[dr] & {} & [5,2^2]\ar@{-}[dl]\ar@{-}[d] \\
{} & [3^3]\ar@{-}[d] & [5,1^4]\ar@{-}[dl] \\
{} & [3^2,1^3] & {}}}
$$
Let ${\mathcal S}'$ be the Slodowy slice singularity from $\0_{[5,1,1,1,1]}$ to $\overline\0$.
By Kraft-Procesi's row removal rule, this is isomorphic to ${\mathcal N}(\mathfrak{so}_4)\cong {\mathbb C}^2/\mu_2\times{\mathbb C}^2/\mu_2$.
It can be checked that $\pi^{-1}({\mathcal S}')$ is also smooth, consisting of two disjoint copies of ${\mathbb C}^4$.
It follows (cf. the diagram) that the singular locus of $\widetilde{X}$ lies in $\pi^{-1}(\overline{\0_{[3,3,1,1,1]}})$.
Let ${\mathcal S}''$ be the Slodowy slice singularity from $\0_{[3,3,1,1,1]}$ to $\overline\0$.
Since $\0_{[3,3,3]}$ is normal and $\pi^{-1}(\0_{[3,3,3]})$ maps bijectively onto $\0_{[3,3,3]}$, then the same holds for $\0_{[3,3,1,1,1]}$.
Hence $\pi^{-1}({\mathcal S}'')$ is an irreducible isolated symplectic singularity, and the map to ${\mathcal S}''$ is the quotient by $G(2,1,2)$.
It appears very likely that $\pi^{-1}({\mathcal S}'')$ is an $a_3$ singularity.

\end{example}

\appendix
\section{Identification of the groups $H(\0)$ and $\cG_\bc'$} \label{sec:appendix}
 
\subsection{The finite groups attached to two-sided cells} \label{subsec:G_c}

Let $G$ be a simple algebraic group over the algebraic closure of a finite field $\F_q$ of good characteristic $p$, with $W$ being its Weyl group. Let $\cU$ denote the unipotent variety in $G$ and let ${\mathcal N}$ be the set of nilpotent elements of the Lie algebra ${\mathfrak g}$.
It is well-known that there is a $G$-equivariant homeomorphism from $\cU$ to ${\mathcal N}({\mathfrak g})$ \cite{sands}[III.3.12], hence a bijection between unipotent and nilpotent classes, respecting centralizers and component groups.
Moreover, these classes are in bijection with the nilpotent classes in the corresponding Lie algebra over $\C$, and outside type $A$ the component group of a nilpotent class is isomorphic to the component group of the corresponding class over $\C$ (see e.g. \cite{Premet}).
For any $G$-stable
 subvariety $Z$ of $\cU$, let $[Z]$ denote the set of unipotent classes
 contained in $Z$.  For each unipotent class $C \in [Z]$, we denote by $A(C)$
 the component group of the centralizer of some point $u \in C$, that is,
 $A(C) = G^u/(G^u)^\circ$. 

There is a one-to-one correspondence between special pieces and two-sided cells of $W$. Given a two-sided cell $\bc$, we have a unique special unipotent class $C_\bc$ with the property that the Springer representation of $W$ corresponding to $\C_\bc$ equipped with the trivial local system is the unique special character in $\bc$, where $(C_\bc,1)$ stands for the trivial local system over $C_\bc$. We denote by $P_\bc$ the special piece of $C_\bc$.
We denote by $P_\bc$ the special piece of $C_\bc$.

Given a two-sided cell $\bc$ of $W$, Lusztig defined  in~\cite{Lusztig1984} a certain finite group $\cG_\bc$  attached to $\bc$ (denoted as $\cG_\cF$ as in \cite{Lusztig1984}, where $\cF$ is the corresponding family of characters of $G$). He also defined a subgroup $\cG'_\bc \subset \cG_\bc$, see \cite[Theorem 0.4]{Lusztig:unipotent} and \cite[Theorem 2.1]{Achar-Sage-special}.
	
We define $H(C_\bc)$ to be the image of the subgroup $H(\0) \subset \Ab(\0)$ under the isomorphism $\Ab(\0) \simeq \Ab(C_\bc)$. The aim of the appendix is to prove the following theorem.
	
\begin{theorem}\label{thm:G'2Hl}
	There is an isomorphism $\cG_\bc \simeq  \Ab(C_\bc)$ of groups, which identifies the subgroup $\cG'_\bc$ with the subgroup $H(C_\bc)$ of $\Ab(C_\bc)$.
\end{theorem}


When $G$ is of exceptional type, $\cG_\bc'$ is trivial if $\cP(\0) = \0$, and $\cG_\bc'=\cG_\bc$ if $\0 \subsetneq \cP(\0)$. In this case, \cref{thm:G'2Hl} is trivial.
	
For classical types, the groups $\cG'_\bc \subset \cG_\bc$ can be desribed combinatorically. Theorem \ref{thm:G'2H_classical} also gives an alternative description. Recall for the classical root systems, $\cG_\bc$ are always finite products of copies of $\Z_2$.  In particular, they are always abelian and therefore the set $\Irr(\cG_\bc)$ of irreducible representations of $\cG$ can be identified with the group $\hat\cG_\bc = \Hom(\cG_\bc, \C^\times)$. Below we will give a particular realization of these groups following \cite[\S\,3.1]{Achar-Sage-special}.  
	
Suppose $\cG_\bc \simeq \Z_2^f$.  Let $\F_2 = \Z_2$ denote the field of $2$ elements. Let $\tilde V_\bc$ be the $(2f+1)$-dimensional vector over $\F_2$ with basis
\[ e_0, e_1, \ldots, e_{2f} \]
and $V_\bc$ be its quotient by the relation
\begin{equation}\label{eqn:Vrel}
	e_0 + e_2 + e_4 + \cdots + e_{2f} = 0,
\end{equation}
so that $\dim V_\bc = 2f$. We will use the same notations for the image of $e_i$'s in $V_\bc$.
We endow $\tilde V_\bc$ with a symmetric bilinear form $\langle \,, \,\rangle$ defined by
\begin{equation}\label{eq:bilinear_form}
\langle e_i, e_j\rangle =
\begin{cases}
	1 & \text{if $|i - j| = 1$,} \\
	0 & \text{otherwise.}
\end{cases}
\end{equation}
The kernel of this form is one-dimensional and is spanned by the element $e_0 + e_2 + \cdots + e_{2f}$, so the form descends to a nondegenerate symmetric bilinear form on $V_\bc$, still denoted as $\langle \, , \,\rangle$.  We make
the identification
\[
	\cG_\bc = \spn \{e_1, e_3, \ldots, e_{2f-1}\}  \subset V_\bc,
\]
viewed as a subspace of $V_\bc$. We regard $\cG_c$ as a (reducible) Coxeter group with the Coxeter generators $e_1, e_3, \ldots, e_{2f-1}$. Via the bilinear form $\langle \, , \,\rangle$ (or rather the multiplicative bilinear form $\exp(\pi \sqrt{-1} \langle \, , \,\rangle)$) we can also identify
\[
	\hat{\cG}_\bc = \Hom(\cG_\bc, \C^\times) =\spn \{e_0, e_2, \ldots, e_{2f}\} \subset V_\bc.
\]
The set $V_\bc$ can be identified with the set $M(\cG_\bc)=\cM(\cG_\cF)$ attached to the group $\cG_\bc=\cG_\cF$ introduced by Lusztig in \cite{Lusztig1984}, see \cite[\S\,3.1]{Achar-Sage-special}.

\subsection{Symbols and $u$-symbols} \label{subsec:symbol}

We will carry out the computations for classical Lie algebras using the combinatorial objects called \emph{symbols} and
\emph{$u$-symbols}, following \cite{Achar-Sage-special}.  They are  defined as certain finite sequences
$(a_0,a_1, \ldots, a_r)$ of nonnegative integers, satisfying various
additional conditions that depend on the type of the classical Lie algebra in question.  Both the set of symbols and the set of $u$-symbols
parameterize $\Irr(W)$ explicitly and hence there is a
bijection
\[
i: \{\text{$u$-symbols}\} \to \{\text{symbols}\}.
\]

Symbols and $u$-symbols are usually expressed as arrays consisting
of two rows of nonnegative integers, but here we will follow the convention in \cite{Lusztig:unipotent} and \cite{Achar-Sage-special} and write them
simply as finite sequences in one row. The even-numbered entries $a_0, a_2, \ldots$ is regarded as the upper row and the odd-numbered
entries as the lower row. In all three classicla types, symbols will have
strictly increasing upper and lower rows while $u$-symbols will have
$a_{i+2}-a_i\ge 2$.

Symbols and $u$-symbols both fall into families determined by their
entries. 

\begin{definition}
	Two symbols $\ba$ and $\bb$ are said to be \emph{congruent} if they contain
	the same entries (with the same multiplicities), possibly in different
	orders.  A symbol $(a_0, \ldots, a_r)$ is \emph{special} if $a_0 \le \cdots
	\le a_r$. 
	
	Two $u$-symbols $\ba$ and $\bb$ are said to be \emph{similar} if they
	contain the same entries (with the same multiplicities), possibly in
	different orders.  A $u$-symbol $(a_0, \ldots, a_r)$ is
	\emph{distinguished} if $a_0 \le \cdots \le a_r$. 
\end{definition}

In particular, each congruence class of symbols contains a unique special
symbol~\cite[Lemmas~1.6, 2.6, 3.6]{Lusztig:unipotent}, and each similarity class of
$u$-symbols contains a unique distinguished symbol~\cite[\S 11.5]{Lusztig:icc}.  The $u$-symbols in a similarity class correspond to an irreducible local system $(C, \rho) \in \cL oc(\cU)$ for a fixed unipotent class $C$ under the Springer correspondence. By~\cite{Lusztig:special}, we know that two symbols are congruent if and only if they correspond to representations of $W$ lying in the same two-sided cell. 

It is often easy to work with the indices of entries in a symbol or
$u$-symbol rather than the entries themselves. We will also use the notation $[k,l] = \{k, k+1, \ldots, l\}$. We refer the reader to \cite[Defintion 3.4]{Achar-Sage-special} for the notions of {\it isolated points, ladders, staircases} and {\it parts} (and their {\it bottoms, tops and lengths}).
Note that, for a distinguished $u$-symbol, the full set of indices $[0,r]$ is always the disjoint union of all the parts.

Recall the following construction.

\begin{definition}[{\cite[Definition 3.5]{Achar-Sage-special}}]
	Let $\ba = (a_0, \ldots, a_r)$ be a symbol 
	and 
	$\mu$ be a subset of $\{0, \ldots, r\}$, such that each $i \in \mu$ is an isolated point of $\ba$.  Let $A =
	\{a_0, a_2, a_4, \ldots\}$, $B = \{a_1, a_3, a_5, \ldots\}$, and $Z =
	\{a_i \mid i \in \mu\}$.  Define two new sequences as follows, both arranged in increasing
	order:
	\begin{align*}
		A' = (b_0, b_2, b_4 \ldots) &=
		\text{$(A \smallsetminus Z) \cup (B \cap Z)$,} \\
		B' = (b_1, b_3, b_5 \ldots) &=
		\text{$(B \smallsetminus Z) \cup (A \cap Z)$.}
	\end{align*}
	If the lengths of $A'$ and $B'$ are the same as the lengths of $A$ and
	$B$, respectively, then we can combine them into a new symbol
	$\bb = (b_0, b_1, b_2, \ldots, b_r)$, which is called the symbol obtained by \emph{twisting} $\ba$ by $\mu$ and is denoted by $\bb = \ba^\mu$.
\end{definition}

It is clear that the twisting operation preserves congruence
classes.

\subsection{Symbols, cells and isolated points}

Following~\cite{Lusztig:special}, we know that two symbols are congruent if and only if
they correspond to representations of $W$ lying in the same two-sided cell.
Moreover, the symbols we have called ``special'' are precisely those
corresponding to special representations of $W$. 

Observe that any two congruent symbols have the same set of entries at
isolated points, therefore we have a natural bijection between their respective sets of isolated points.  

\begin{definition}[{\cite[Definition 3.6]{Achar-Sage-special}}]
	Let $\ba_0$ be a special symbol and $\ba$ be a symbol in the congruence class of $\ba_0$.  An isolated point of $\ba$ (or the corresponding $u$-symbol $i^{-1}(\ba)$) is said to be \emph{displaced} if its parity is opposite to that of the corresponding isolated point of $\ba_0$.  Each symbol (and
	$u$-symbol) always has an even number of displaced
	isolated points. 
\end{definition}


Define $\tilde V(\ba_0)$ to be the power set of the set of isolated points of $\ba_0$. We regard $\tilde V(\ba_0)$ as an $\F_2$-vector space with addition given by symmetric difference of sets.  Define $V(\ba_0)$ to be the subspace of $\tilde V(\ba_0)$ consisting of sets with even cardinality and endow it with a symmetric bilinear form by setting $\langle v, w \rangle = |v \cap w| \mod 2$, where $|v \cap w|$ stands for the cardinality of the set $v \cap w$. We have a natural map
\[
\tilde\kappa:
\left\{\begin{array}{c}
	\text{symbols}\\
	\text{congruent to $\ba_0$}
\end{array}\right\}
\to V(\ba_0),
\qquad
\ba \mapsto
\left\{\begin{array}{c}
	\text{displaced}\\
	\text{isolated points of $\ba$}
\end{array}\right\}.
\]
If $\ba$ is a symbol and $\mu$ is a set
consisting of an even number of isolated points such that $\ba^\mu$ is
defined, then
$\tilde\kappa(\ba^\mu) = \tilde\kappa(\ba) + \mu$. 

Lusztig has shown that if $\bc$ is the two-sided cell corresponding to
the special symbol $\ba_0$, then there is a natural surjective map
$\pi$ from $V(\ba_0)$ to the space $V_\bc$ of \S,\ref{subsec:G_c}, which is an isomorphism (even an isometry) except in type
$D$ \cite{Lusztig1984}. This map can be described explicitly in each type explicitly. For details,
see \cite[pp.~86--88]{Lusztig1984} for types $B$ and $C$, and \cite[pp.~92--94]{Lusztig1984} for type $D$.

\subsection{The isomorphism $\cG_\bc \simeq \Ab(C_\bc)$}

Let $\ba = (a_0, \ldots, a_r)$ be a distinguished $u$-symbol. Set
$H(\ba)$ to be the power set of the set of ladders of $\ba$. We endow $H(\ba)$ with the structure of a $\F_2$-vector space whose addition is given by symmetric difference of subsets. 
Let $C$ be the unipotent class corresponding to $\ba$.  In each classical type, $\Irr(A(C))$
can be canonically identified with a quotient of some subspace of $H(\ba)$ (here the component group $A(C)$ is defined for the classical group $G$). For details, see  in \cite[\S 12]{Lusztig:icc} for type $C$ and \cite[\S 13]{Lusztig:icc} for types $B$ and $D$ (also cf. \cite[pp. 419--423]{Carter}.) We can extend the quotient map to a surjective homomorphism $p: H(\ba) \to \Irr(A(C))$ of abelian groups. 
The extension might not be unique, but we always require that the set of all ladders lies in its kernel.

There is a notion of \emph{blocks} of
$\ba$, which are certain unions of ladders and staircases and whose definition depend on the classical type in question.  Each block
contains at least one ladder and each ladder is contained in
a block.  Moreover, distinct blocks are disjoint.  
There are also the notions of the ``top'' and ``bottom'' of a block, but they are
not quite the same as those of ladders and staircases. The definitions vary by type (for type $C$, see Definition \ref{defn:symbol_C}. In general, see \cite[Section 4]{Achar-Sage-special}). 
Most importantly, not every block has both a top and a bottom.

Let $B(\ba)$ be the set of blocks and $\Hb(\ba)$ be the power set of
$B(\ba)$.  We have an injective homomorphism $\Hb(\ba) \hookrightarrow H(\ba)$ of abelian groups by assigning to a block the set of ladders contained within
it, so that we can regard $\Hb(\ba)$ as a subset of $H(\ba)$.  
For any subset $S \subset [0,r]$, let $\operatorname{iso}(B)$ denote the set of isolated points in $S$. Let $b$ be a block of a distinguished $u$-symbol $\ba$. We define $\lambda_b
\in V(i(\ba))$ by
\[
\lambda_b =
\begin{cases}
	\operatorname{iso}(b) &
	\text{if $\operatorname{iso}(b)$ has even cardinality;} \\
	\operatorname{iso}([0,r] \backslash b) &
	\text{otherwise.}
\end{cases}
\]
We can extend the definition of $\lambda_b \in V(i(\ba))$ to elements $b\in\Hb(\ba)$ by linearity, to get a homomorphism $\Lambda: \Hb(\ba) \to V_\bc$, $b \mapsto \pi(\lambda_b)$. 

From now on, we assume $i(\ba)$ is special and hence $\ba$ corresponds to the unique special class $C_\bc$ in $\cP_\bc$.
By Lemma 3.1, Proposition 3.10 and the paragraph immediately following it in \cite{Achar-Sage-special}, the image of $\Lambda$ is exactly $\hat{\cG} = \Irr(\cG_\bc)$ and its kernel is generated by $B(\ba)$. Note that the analog for distinguished $u$-symbol $\ba$ corresponding to general unipotent classes in $\cP_\bc$ also holds, see \emph{loc. cit.}


On the other hand, It can be shown that the kernel of the restriction map $p|_{\Hb(\ba)}: \Hb(\ba) \to \Irr(A(C_\bc))$ is also generated by $B(\ba)$ (see \cite[Prop. 3.13]{Achar-Sage-special}. In fact, $p|_{\Hb(\ba)}$ is canonically defined and independent of the choice of $p$, given the assumption that $B(\ba)$ lies in its kernel. Moreover, the image of $p|_{\Hb(\ba)}$ is exactly $\Irr(\Ab(C_\bc))$ by \cite[Prop. 3.14]{Achar-Sage-special}, where $\Irr(\Ab(C_\bc))$ is naturally regarded as a subset of $\Irr(A(C_\bc))$ since $\Ab(C_\bc)$ is a quotient of $A(C_\bc)$ (this fact will be verified again in the proof of Theorem \ref{thm:superminimal} below.)  Therefore the two surjective homorphisms from $\Hb(\ba)$ to $\Irr(\cG_\bc)$ and $\Irr(\Ab(C_\bc))$ induce a canonically defined isomorphism $ \Irr(\cG_\bc) \simeq \Irr(\Ab(C_\bc))$ of abelian groups, which corresponds to an isomorphism $\cG_\bc \simeq \Ab(C_\bc)$. 


\subsection{Main results for classical groups}

Let $G$ be a classical group. Let $\bc$ be a two-sided cell, and let $\lambda$ be the special partition corresponding to the special unique class $C_\bc$ or the corresponding nilpotent orbit $\0$. We make an additional convention when $\epsilon = 1$ (i.e., we are in type $C$): we take $h_\lambda(0)$ to be the smallest integer that is congruent to $\epsilon' \equiv 0 \bmod 2$ and larger than $\ell(\lambda)$.

We now list the parts of a partition $\lambda$ of $2n$ in increasing order as $\lambda_1 \le \lambda_2 \le \cdots \le \lambda_r$. Recall when $\g$ is of type $C$, we require $r$ to be even by setting $\lambda_1$ if necessary. Accordingly, we relabel $\nu_1>\cdots > \nu_l$ defined in \S\,\ref{subsub:classical_component_groups} in increasing order as $\eta_0 <\eta_1 < \cdots < \eta_{l-1}$, and such that $\eta_0 = 0$ in type $C$. 
We set $\omega_0 < \omega_1 < \cdots < \omega_{q-1}$ to be the subsequence of $\eta_0 <\eta_1 < \cdots < \eta_{l-1}$ consisting of those $\eta_i$ such that $h_\lambda(\eta_i) \equiv \epsilon' \bmod 2$. In all cases we have $\omega_0 = \eta_0 = \nu_l$, which is $0$ in type $C$ by our convention. Also $\omega_l = \infty$.

For each $1 \le j \le q-1$, $\omega_j = \nu_{i_j}$ for some $i_j \in [l-1]$. We define elements   
	\[\theta_j := x_{\nu_{i_j}} x_{\nu_{i_j+1}} \in A(\0), \quad \bar{\theta}_j = \overline{x_{\nu_{i_j}} x_{\nu_{i_j+1}}} \in \Ab(\0), \quad 1 \le j \le q-1.\] 
Note that in type $C$, we have $\theta_1 =  x_{\omega_1} x_{\omega_0} = x_{\omega_1} x_{0} = x_{\omega_1}$. Then $\{\bar{\theta}_j\}_{1 \le j \le q-1}$ is just the basis $\{ \bar{\sigma}_i \}_{i \in \vartheta_0}$ of $\Ab(\0)$ in \cref{prop:ALT_A2Abar} after reversing the order.

\begin{theorem}\label{thm:superminimal}
	The isomorphism $\cG_\bc \simeq \Ab(C_\bc)$ sends each basis element $e_{2i-1}$, $1 \le i \le f$ to the superminimal generator $\bar{\theta}_i$, $1 \le i \le q-1$. In particular, $f = q-1$.
\end{theorem}

\begin{theorem}\label{thm:G'2H_classical}
	The isomorphism $\cG_\bc \simeq \Ab(C_\bc)$ identifies the subgroup $\cG'_\bc$ with the subgroup $H(C_\bc)$ of $\Ab(C_\bc)$.
\end{theorem}

We will only prove the type $C$ case and leave details for type $B$ and $D$ to the readers.

\subsection{Proof of Theorem \ref{thm:superminimal} for type $C$}

Assume $\g = \mathfrak{sp}(2n)$. We give below the concrete definitions of the notions discussed in \S\, \ref{subsec:symbol}, such as symbols, $u$-symbols, blocks, etc.

\begin{definition}[{\cite[\S\,4.1]{Achar-Sage-special}}] \label{defn:symbol_C}
Let $\Psi'_{2n,m}$ (resp.~$\Phi'_{2n,m}$) be the set of
sequences $\ba = (a_0, a_1, \ldots, a_{2m})$ satisfying the following conditions respectively:
\begin{align*}
	0 &\le a_0, & 1 &\le a_1, & a_i &\le a_{i+2}-2, & \sum a_i &= n + 2m^2
	+ m & &\text{for $\ba \in \Psi'_{2n,m}$,} \\
	0 &\le a_0, & 0 &\le a_1, & a_i &\le a_{i+2}-1, & \sum a_i &= n + m^2 
	& &\text{for $\ba \in \Phi'_{2n,m}$}
\end{align*}
We have the \emph{shift operations} which are embeddings $S: \Psi'_{2n,m} \to \Psi'_{2n,m+1}$ given by
\[ 
	S(\ba) = (0,1,a_0+2,\ldots,a_{2m}+2) \quad \text{for $\ba \in \Psi'_{2n,m}$,} 
\]
and embeddings $S: \Phi'_{2n,m} \to \Phi'_{2n,m+1}$ given by
\[
	S(\ba) = (0,0,a_0+1,\ldots,a_{2m}+1) \quad \text{for $\ba \in
		\Phi'_{2n,m}$.} 
\]
Define $\Psi'_{2n} := \lim_{m\to \infty} \Psi'_{2n,m}$ and $\Phi'_{2n} := \lim_{m\to\infty} \Phi'_{2n,m}$.  We can fix an $m$
large enough so that $\Psi'_{2n} \simeq \Psi'_{2n,m}$ and $\Phi'_{2n} \simeq \Phi'_{2n,m}$.  Elements of $\Phi'_{2n}$ (resp.~$\Psi'_{2n}$) are called 
\emph{symbols} (resp.~\emph{$u$-symbols}) of type $C$.

We have the bijection $i: \Psi'_{2n} \to \Phi'_{2n}$ given by
\[
i(a_0, \ldots, a_{2m}) = (a_0, a_1 -1, a_2 -1, \ldots, a_{2m-1}-m, a_{2m}-m).
\]

\end{definition}

\begin{definition}[{\cite[Definition 4.1]{Achar-Sage-special}}] \label{defn:block_C}
	Given a distinguished $u$-symbol $\ba$, a \emph{block} for $\ba$ is a subset
	$[k,l]$ of $[0,2m]$ that satisfies one of the following conditions:
	\begin{itemize}
		\item $[k,l]$ is a union of consecutive parts $P_1, \ldots, P_r$, where
		$P_1$ and $P_r$ are ladders (it is possible $r = 1$), $P_r$ is the unique part with even top, and
		$P_1$ is the unique part with odd bottom. 
		\item $k = 0$, and $[0,l]$ is the union of consecutive parts $P_2, \ldots,
		P_r$, where $P_r$ is a ladder and the only part with even top and where
		all parts have even bottom. 
	\end{itemize}
	In either case, the \emph{top} of the block is the top of $P_r$, and
	its \emph{bottom} is the bottom of $P_1$, provided the latter is
	defined.  Blocks satisfying the second condition above are said to
	have no bottom.  Note that there is always exactly one block of this
	type.
\end{definition}

\begin{remark}\label{rmk:staircase}
	Note that a part belongs to some block if and only if
	it is not a staircase with odd bottom (and hence even top). 
\end{remark}

To prove Theorem \ref{thm:superminimal}, it suffices to show that for any $0 \le j \le f$, the values of $e_{2j} \in \cG_\bc$ at $e_{2i-1}$ coincide with the values at $\bar{\theta}_{i}$ for all $i$, since $\hat{\cG}_\bc =\spn \{e_0, e_2, \ldots, e_{2f}\} $. This in turn can be deduced from the following proposition. 

We now list the parts of a special $C$-partition $\lambda$ of $2n$ in increasing order as $\lambda_1 \le \lambda_2 \le \cdots \le \lambda_r$. Recall in type $C$, we require $r$ to be even by setting $\lambda_1=0$ if necessary.
For the special partition $\lambda$, we can compute its $u$-symbol $\ba = (a_0, a_1, \ldots, a_{2m})$, where $2m = r$, using \cite[\S\,11.6]{Lusztig:icc}. Namely, we first define $\lambda_i^* = \lambda_i + i -1$, which satisfy
\[ \lambda_1^* < \lambda_2^* < \cdots < \lambda_{2m}^*, \]
and then write the odd and even terms as
\[
	2y'_1+1<\cdots <2y'_m+1,
	\qquad
	2y_1<\cdots <2y_m.
\]
Then the $u$-symbol $\ba \in \Psi'_{2n,m}$ or $\ba \in \Phi'_{2n}$ attached to $\lambda$ is given by
\begin{equation} \label{eq:part2symbol}
	\ba = (0,~y_1+1, ~ y'_1+2, ~ y_2+2, ~ y'_2+3, ~y_3+3, ~\ldots, ~y'_{m-1}+m, ~y_m+m, ~y'_m+m+1). 
\end{equation}

Since the odd parts of the special partition $\lambda$ of type $C$ must have even multiplicities and even heights, it is easy to see that they give rise to staircases with odd bottom, which do not belong to any block by the remarks above. It is easy to see that all the staricases of $\ba$ arise this way.

We now describe all the ladders of $\ba$. Suppose $\omega_{i} = \eta_{w_i}$ for some index $w_i$, then we have 
\[ \omega_{i} = \eta_{w_i} < \eta_{w_i +1} < \cdots < \eta_{w_{i+1}-1} < \eta_{w_{i+1}} = \omega_{i+1} \]
In fact, all parts of $\lambda$ lying within the interval $[\eta_{w_i}, \eta_{w_{i+1}-1}]$ are even and hence must belong to the set 
\[ \{\eta_{w_i}, \eta_{w_i +1} , \ldots,  \eta_{w_{i+1}-1} \}, \] 
since odd parts must have even heights by the specialness and $\eta_{w_i +1}, ~\ldots, \eta_{w_{i+1}-1}$ all have odd heights (there might be odd parts of $\lambda$ in the interval $(\eta_{w_{i+1}-1}, \omega_{i+1})$).

For $1 \le j \le l-1$, the parts of length $\eta_{j}$ in $\lambda$ give rise to a ladder of length $m(\eta_j)$
	\[ L_j = [2m - h(\eta_j) + 1, 2m - h(\eta_j)+m(\eta_j)] \]
of $\ba$. We also have a ladder $L_0$ with bottom $0$. Note that if we took a representative of the symbol $\ba$ of length less than $r$, then $a_0$ might not be $0$ and there would be no extra ladder $L_0$. In this case $p$ is not well-defined. Then $L_0, L_1, \ldots, L_{l-1}$ exhaust all the ladders of $\ba$ in increasing order.
Therefore the parts of length $\omega_i$ gives rise to a ladder with odd bottom if $i \ge 1$, while the parts of length $\eta_j$ within the interval $(\omega_i, \eta_{w_{i+1}-1})$ gives rise to ladders with even bottom and odd top, and $\eta_{w_{i+1}-1}$ gives rise to a ladder of even bottom and and even top. Therefore the ladders coming from 
\[ \omega_{i} = \eta_{w_i} < \eta_{w_i +1} < \cdots < \eta_{w_{i+1}-1} \]
altogether form a single block by Definition \ref{defn:block_C}, denoted as $b_i$. Explicitly, we have 
	\[ b_i = [2m - h(\omega_i) + 1, 2m - h(\eta_{w_{i+1}-1})+m(\eta_{w_{i+1}-1})], \quad \forall\,i \ge 1, \]
and 
	\[ b_0 = [0, 2m-h(\eta_{w_1-1})+m(\eta_{w_1-1})].\]
Moreover, every block of $\ba$ is of this form. By \cite[Lemma 4.3, (3)]{Achar-Sage-special}, we see that each $b_i$ contains exactly two isolated points $k_{2i-1}$ and $k_{2i}$ (the two endpoints) for $i \ge 1$, while $b_0$ has even top and has a unique isolated point $k_0 = 2m-h(\eta_{w_1-1})+m(\eta_{w_1-1})$.

For $i \ge 1$, the block $b_i$ has even top and odd bottom, therefore by \cite[Lemma 4.3, (3)]{Achar-Sage-special}, it contains two isolated points $k_{2i-1}$ and $k_{2i}$, while the unique bottomless block $b_0$ has even top contains a single isolated point $k_0$. Therefore there are $2q-1$ isolated points in total, but the number of isolated points is also equal to $2f+1$, hence $f=q-1$. Regard $b_i$ as an element in $\Hb(\ba)$, then its image in $\hat\cG_\bc$ is exactly $e_{2i}$ when $i \ge 1$ (see the proof of \cite[Prop. 4.7]{Achar-Sage-special}).

On the other hand, by \cite[\S\,11.6]{Lusztig:icc}, the map $p: H(\ba) \to \Irr(A(C_\bc))$ is defined as follows: we have already seen above that the ladders of $\ba$ are given by $L_0, L_1, \ldots, L_{l-1}$, with each $L_i$ corresponds to $\eta_i$ when $i \ge 1$. Then for $i \ge 1$, $p(L_i)$ is a character $\psi_i$ of $A(C_\bc)$ defined by 
\[
	\psi_i(x_{\eta_j})=
	\begin{cases}
		-1 & \text{if } j=i, \\
		+1 & \text{if } j \neq i.
	\end{cases}
\]	
By the requirement that the sum of all ladders should lie in the kernel of $p$, we have 
\begin{equation} \label{eq:p(L_0)}
	p(L_0) = p(L_1) + \cdots p(L_{l-1}) = \sum_{i = 1}^{l-1} \psi_i,
\end{equation} 
that is, $p(L_0)$ is the sign character of $A(C_\bc) \simeq \frakS_2^{l-1}$. We will only need $p(L_i)$ for $i \ge 1$ below.

For $i\ge 1$, $p$ maps the block $b_i$ to
\[ \phi_i=\sum_{s=w_i}^{w_{i+1}-1}\psi_s, \]
which takes the value $-1$ on $x_{\omega_i},x_{\eta_{w_i+1}},\ldots, x_{\eta_{w_{i+1}-1}}$, and $+1$ on the remaining generators. By Proposition \ref{prop:ALT_A2Abar}, $\phi_i$ descends to a character $\overline\phi_i$ of $\Ab(C_\bc)$.
Therefore the isomorphism $\hat{\cG}_\bc \simeq \Irr(\Ab(C_\bc))$ matches $e_{2i}$ with $\overline{\phi}_i$ for all $i \ge 1$. 

For any $i \ge 1$, 
\[
	\overline{\phi}_i (\bar{\theta}_j) =
	\begin{cases}
		-1 & \text{if $i = j$ or $j-1$,} \\
		+1 & \text{otherwise.}
	\end{cases}
\]
Comparing this with the bilinear form \eqref{eq:bilinear_form}, we see that the
dual isomorphism $\cG_\bc \simeq \Ab(C_\bc)$ sends $e_{2i-1}$ to
$\bar{\theta}_i$ for $1\le i\le q-1$.  This proves
Theorem \ref{thm:superminimal} in type $C$.

\begin{remark}
	Thanks to \eqref{eq:p(L_0)}, one can check that $p(b_0)$ is the character $\phi_0$ of $\Ab(C_\bc)$, whose values at all $x_{\eta_i}$ with $i \ge w_1$ are equal to $-1$ and $+1$ otherwise. In other words, $\phi_0 = \phi_1 + \cdots + \phi_{q-1}$. Therefore, under the isomorphism $\hat{\cG}_\bc \simeq \Irr(\Ab(C_\bc))$, $e_0$ is sent to $\phi_0$. This is compatible with the relation \eqref{eqn:Vrel}.
\end{remark}

\subsection{Proof of Theorem \ref{thm:G'2H_classical} for type $C$}

Let $\ba$ be a distinguished $u$-symbol such that $i(\ba)$ is a special symbol. It corresponds to the unique special class $C_\bc$ in $\cP_\bc$. In the type $C$ case, the isomorphism $\pi: V(i(\ba)) \xrightarrow{\sim} V_\bc$ (where $\bc$ is the
two-sided cell corresponding to $i(\ba)$) of $\F_2$-vector spaces is defined by 
\[
\pi(\{k_{i-1}, k_i\}) = e_i, \qquad i = 1, \ldots, 2f
\qquad
(\text{hence }\pi(\{k_1, k_2, \ldots, k_{2f}\}) = e_0.)
\]

Given any $u$-symbol $\bb$ such that the corresponding symbol $\bb':=i(\bb)$ is congruent to $\ba' :=i(\ba)$, the symbol $\bb'$ is a twisting of $\ba'$ by $\tilde{\kappa}(\bb) \in V(\ba')$, which is the sum of displaced isolated points of $\bb'$. 
We write
\begin{align*}
	\ba &= (a_0, \ldots, a_{2m}), \quad
	i(\ba) = \ba' = (a'_0, \ldots, a'_{2m}) \\
	\bb &= (b_0, \ldots, b_{2m}), \quad
	i(\bb) = (\ba')^{\tilde{\kappa}(\bb)} = \bb' = (b'_0, \ldots, b'_{2m}).
\end{align*}
By \cite[Lemma 4.2]{Achar-Sage-special}, there are an odd number of isolated points of $\ba$ (and $\bb$). 
Let $k_0, \ldots, k_{2f}$
be the isolated points of $\ba$, numbered such that
\[
a'_{k_0} < a'_{k_1} < \cdots < a'_{k_{2f}}.
\]
Since $\ba'$ is special, we have $k_t \equiv t \pmod{2}$. 
By the proof of \cite[Prop. 4.6]{Achar-Sage-special} and \cite[\S\, 1.7]{Lusztig:unipotent}, $\bb = i^{-1}(\bb')$ is a distinguished $u$-symbol if and only if the following conditions are satisfied:
\begin{enumerate}[leftmargin=*]
	\item 
		the (nondisplaced) isolated points of $\ba$ corresponding to the displaced isolated points of $\bb$ must occur in pairs $\{k_{t-1}, k_{t}\} $ with $t$ odd (so that $k_{t-1}$ is odd and $k_t$ is even), so that the image $\pi(\tilde{\kappa}(\bb)) \in V_\bc$ is a sum of $e_t$'s with $t$ odd (see the proof of \cite[Prop. 4.5]{Achar-Sage-special}. In fact, $k_{t-1}$ must be the top of a  block of $\ba$ and $k_t$ is the bottom of the next block;) 
	\item
		for each pair of isolated points $\{k_{t-1}, k_{t}\}$ in (2), we have
		\[ (a_{k_{t-1}}, a_{k_{t-1}+1}, \ldots, a_{k_t}) = (a, a+2, a+2, a+4, a+4, \ldots, a+k_t-k_{t-1}-1, a+k_t-k_{t-1}-1, a+k_t-k_{t-1}+1) \]
		for some nonegative integer $a$ (note that $[k_{t-1}+1, k_{t}-1]$ is a staircase with odd bottom, hence does not belong to any block by Remark \ref{rmk:staircase}.)
\end{enumerate}
Moreover, $\bb$ can be obtained from $\ba$ by replacing the subsequence $(a_{k_{t-1}}, a_{k_{t-1}+1}, \ldots, a_{k_t})$ in (1) by 
 	\[ (b_{k_{t-1}}, b_{k_{t-1}+1}, \ldots, b_{k_t}) = (a+1, a+1, a+3, a+3, a+5, a+5, \ldots, a+k_t-k_{t-1}, a+k_t-k_{t-1})\]
for each pair of isolated points $\{k_{t-1}, k_{t}\}$ in (1).


Set 
\[ \Theta'(\ba) = \{ \bb \in \Psi'_{2n} \,|\, \bb~  \text{is distinguished and } i(\bb)~\text{is congruent to } i(\ba) \}. \]
We have a bijection 
	\[ \Theta'(\ba)  \xrightarrow{\sim} [\cP_\bc], \quad \bb \mapsto C_\bb, \]
whose inverse sends any unipotent class $C=C_\bb \in \cP_\bc$, or rather the corresponding partition, to its associated $u$-symbol $\bb$, via the formula \eqref{eq:part2symbol}. Let $\lambda$ be the partition of $C_\bc = C_\ba$ and $\lambda_{\bb}$ be the partition of $C_\bb$. An easy computation via the formula \eqref{eq:part2symbol}shows that, for any $i \geq 1$, subsequences $(a_{k_{2i-2}},  \ldots, a_{k_{2i-1}})$ of $\ba$ in (2) correspond bijectively with subpartitions of $\lambda$ of the form    
	\[ \gamma_i = [\omega_{i}-2, (\omega_i-1)^{m_i}_i, \eta_{w_i} = \omega_{i}], \] 
and the corresponding subsequence $(b_{k_{2i-2}}, \ldots, b_{k_{2i-1}})$ of $\bb$ in (1) arise from subpartitions $\lambda_\bb$ of the form 
	\[\gamma'_i = [(\omega_i-1)^{m_i+2}], \] 
which are obtained from $\gamma_i$ by collapsing. 
In other words, each twisting of $i(\ba)$ by the the pair $\{k_{2i-2}, k_{2i-1}\}$ contributes an elementary collapse
\[
	\gamma_i=[\,\omega_i-2,\ (\omega_i-1)^{m_i},\ \omega_i\,]
	\quad\rightsquigarrow\quad
	\gamma'_i=[\,(\omega_i-1)^{m_i+2}\,]
\]
at the level of partitions of the nilpotent orbits.

By \cref{cor:i_A_injective_all_types}, the subgroup $H(C_\bc)$ is generated by the superminimal generators attached to exactly these collapses, namely the elements $\bar{\theta}_{i_j}$ with $\eta_{i_j} = \omega_i$. \cref{thm:superminimal} identifies those generators with the corresponding odd basis elements $e_{2i_j-1}$ of $\cG_\bc$. Hence the subgroup $\cG'_\bc$ generated by the first coordinates of the classes in $\cP_\bc$ is identified with $H(C_\bc)$. This proves Theorem \ref{thm:G'2H_classical} in type $C$.

\bibliographystyle{alpha}
\bibliography{fjls}

\end{document}